\documentclass[11pt]{article}

\usepackage{authblk}
\usepackage{fullpage}  
\usepackage{float}
\usepackage{pb-diagram}  		
\usepackage{esvect}
\usepackage[utf8]{inputenc} 
\usepackage[T1]{fontenc}    
\usepackage{hyperref}       
\usepackage{url}            
\usepackage{booktabs}       
\usepackage{amsfonts}       
\usepackage{nicefrac}       
\usepackage{microtype}      
\usepackage{fullpage}  
\usepackage{float}
\usepackage{pb-diagram}  
\usepackage{makecell}		

\usepackage{url}

\usepackage{algorithm, algorithmicx,algpseudocode, listings} 
\usepackage{multirow} 
\usepackage{hhline}  
\usepackage{amsmath}
\usepackage{xcolor}

\usepackage{graphicx}
\usepackage{amscd}
\usepackage{amssymb,mathrsfs}
\input{epsf.sty}
\usepackage{amsthm,amscd}
\usepackage{color}
\usepackage{latexsym}
\usepackage{epic}
\usepackage{appendix}
\usepackage{enumerate}
\usepackage{longtable}
\usepackage{lscape}
\usepackage{extarrows}
\usepackage{epstopdf}
\usepackage{listings}
\newlength\myindent

\newcommand{\Rmnum}[1]{\expandafter\@slowromancap\romannumeral #1@}
\newcommand{\inner}[3][]{{\langle #2,#3 \rangle_{#1}}}

\allowdisplaybreaks

\newcommand{\delete}[1]{{}}

\newtheorem{theorem}{Theorem}[section]
\newtheorem{lemma}{Lemma}[section]

\newtheorem{assumption}{Assumption}[section]

\newtheorem{example}{Example}[section]
\numberwithin{equation}{section}

\DeclareMathOperator{\T}{\mathrm{T}}

\DeclareMathOperator{\grad}{\mathrm{grad}}

\DeclareMathOperator{\Exp}{\mathrm{Exp}}
\DeclareMathOperator{\Prox}{\mathrm{Prox}}
\DeclareMathOperator{\id}{\mathrm{id}}

\DeclareMathOperator{\N}{\mathrm{N}}

\DeclareMathOperator{\F}{\mathrm{F}}
\DeclareMathOperator{\D}{\mathrm{D}}

\DeclareMathOperator{\trace}{\mathrm{trace}}

\DeclareMathOperator{\diag}{\mathrm{diag}}
\DeclareMathOperator{\St}{\mathrm{St}}

\DeclareMathOperator*{\argmin}{arg\,min}
\DeclareMathOperator{\tr}{trace}

\begin{document}

\title{A Riemannian Optimization Approach to Clustering Problems\footnotetext{Corresponding author: Wen Huang (\url{wen.huang@xmu.edu.cn}). Wen Huang was partially  supported by the National Natural Science Foundation of China (NO. 12001455) and the Fundamental Research Funds for the Central Universities (NO. 20720190060). K. A. Gallivan was partially supported by the U.S. National Science Foundation under grant CIBR 1934157.}}
\author[1]{Wen Huang}
\author[2]{Meng Wei}
\author[3]{Kyle A. Gallivan}
\author[4]{Paul Van Dooren}

\affil[1]{School of Mathematical Sciences, Xiamen University, Xiamen, China. \vspace{.15cm}}
\affil[2]{Department of Mathematics, Florida State University, 208 Love Building, 1017 Academic Way, Tallahassee, FL 32306-4510, USA.\vspace{.15cm}}
\affil[3]{Department of Mathematics, Florida State University, USA.\vspace{.15cm}}
\affil[4]{Department of Mathematical Engineering, Universit\'e catholique de Louvain, Louvain-La-Neuve, Belgium.}

\maketitle


\begin{abstract}
This paper considers the optimization problem in the form of 
$
\min_{X \in \mathcal{F}_v} f(x) + \lambda \|X\|_1,
$
where $f$ is smooth, $\mathcal{F}_v = \{X \in \mathbb{R}^{n \times q} : X^T X = I_q, v \in \mathrm{span}(X)\}$, and $v$ is a given positive vector. The clustering models including but not limited to the models used by $k$-means, community detection, and normalized cut can be reformulated as such optimization problems. It is proven that the domain $\mathcal{F}_v$ forms a compact embedded submanifold of $\mathbb{R}^{n \times q}$ and optimization-related tools 
including a family of computationally efficient retractions and an orthonormal basis of any normal space of $\mathcal{F}_v$
are derived. An inexact accelerated Riemannian proximal gradient method 
that allows adaptive step size 
is proposed and its global convergence is established. Numerical experiments on community detection in networks and normalized cut for image segmentation are used to demonstrate the performance of the proposed method.
\end{abstract}

\section{Introduction}
Optimization on Riemannian manifolds 
concerns optimizing a real-valued objective function defined on a Riemannian manifold. It has been of interest due to many important applications, e.g., image segmentation~\cite{RinWir2012} and recognition~\cite{TurVeeSri2011}, electrostatics and electronic structure calculation~\cite{WenYin2012,HJLWY2019}, computer vision~\cite{HGSA2014,SH2021}, signal processing~\cite{WCCL2016,Vandereycken2013,HH2017}, numerical linear algebra~\cite{SI2013}, community detection~\cite{WHGV2021}, and machine learning~\cite{CS2015}. 

Many Riemannian optimization methods for smooth objectives have been proposed and systemically analyzed, e.g., Riemannian trust-region Newton method~\cite{BAKER08}, Riemannian Broyden family method including BFGS method and its limited-memory version~\cite{RinWir2012,HUANG2013,HGA2014,HuaAbsGal2018}, Riemannian trust-region symmetric rank-one update method and its limited-memory version~\cite{HUANG2013,HAG13,HG2022}, Riemannian Newton method and Riemannian non-linear conjugate gradient method~\cite{AMS2008,SI2015,Sato2015,Zhu2016}. However, the work of Riemannian optimization for nonsmooth objective functions is relatively limited. Most work considers on the subgradient-based methods~\cite{GH2015a,GH2015b,ZS2016,BFM2017,HHY2018}. These methods either focus on geodesically convex objective functions or require solving a quadratic program to high accuracy, which limits the scope of applications. When the objective functions have structure,  more efficient optimization algorithms can be developed. Recently, Chen et al.~\cite{CMSZ2019} considered nonsmooth problems on the Stiefel manifold that has the splittable form $\min_{X \in \mathcal{M}} f(X) + g(X)$, where the manifold $\mathcal{M}$ is the Stiefel manifold $\St(q, n) = \{X \in \mathbb{R}^{n \times q} : X^T X = I_q\}$, $f$ is smooth, and $g$ is nonsmooth but admits a simple proximal mapping. A proximal gradient method is proposed therein with global convergence and is applicable for solving large-scale problems. The proximal mapping is solved by a semi-smooth Newton algorithm. However, no convergence rate analysis is given. In~\cite{HuaWei2019b}, a different version 
of the proximal gradient method in~\cite[Section~10.2]{Beck2017} 
for the splittable function is proposed without restricting the manifold to be the Stiefel manifold. A convergence rate analysis is also given. The proximal mapping in~\cite{HuaWei2019b} involves an iterative algorithm that uses a semi-smooth Newton algorithm in every iteration and therefore can be less efficient than that in~\cite{CMSZ2019}\footnote{In some cases, the proximal mapping in~\cite{HuaWei2019b} can be solved efficiently without resorting to the semi-smooth Newton algorithm, see~\cite[Section~5.2]{HuaWei2019b}.}. In~\cite{HuaWei2021}, an inexact Riemannian proximal gradient method without solving the Riemannian proximal mapping of~\cite{HuaWei2019b} exactly is proposed. It relaxes the requirements of solving the proximal mapping while preserving the convergence properties. 
However, the theoretical results therein rely on the knowledge of the constant $L$ of the $L$-retraction-smoothness of the function $f$, which may not be available. 
The present paper proposed a different inexact proximal gradient method that avoids this difficulty. 

In this paper, we consider the nonsmooth optimization problems over a subset of the Stiefel manifold, i.e.,
\begin{equation} \label{prob3}
\min_{X \in \mathcal{F}_v} f(X) + \lambda \|X\|_1,	
\end{equation}
where the gradient of $f$ is Lipschitz continuous, $\mathcal{F}_v = \{X \in \mathbb{R}^{n \times q} : X^T X = I_q, v \in \mathrm{span}(X)\}$, $v \in \mathbb{R}^n$ is a vector with all entries being positive, and $\mathrm{span}(X)$ denotes the columns space of~$X$. 


\subsection{Applications}
Problem~\eqref{prob3} can be viewed as an alternative formulation of the problem 
\begin{equation} \label{prob1}
\min_{X \in \mathcal{A}_{v}} f(X)
\end{equation}
under certain circumstances, where $\mathcal{A}_{v} = \{X \in \mathbb{R}^{n \times q} : X^T X = I_q, X \geq 0, v \in \mathrm{span}(X) \}$, and $X \geq 0$ denotes that all entries of $X$ are nonnegative. 
We next give a few important clustering problems that can be formulated in terms of~\eqref{prob1} and defer the discussions about connections between Problem~\eqref{prob1} and Problem~\eqref{prob3}.

\begin{example}[$k$-means model]
The $k$-means algorithm~\cite{Mac1967,HarWon1979} was proposed to cluster vectors in $\mathbb{R}^d$. It is one of the most popular clustering algorithms due to its simplicity and efficiency. Given $n$ points $a_i$ in $\mathbb{R}^d$ and $k$ initial estimations of the means of $k$ groups, the $k$-means algorithm first assigns $a_i, i = 1, \ldots, n$ to their closest means and creates $k$ groups. Then the $k$ means are updated by computing the means of the new $k$ groups. Such process is repeated until the algorithm converges. The $k$-means algorithm turns out to be a minimization algorithm for solving the optimization problem 
\begin{equation} \label{eq:app1}
\min_{X \in \mathcal{A}_{\mathbf{1}_n}} ||A-XX^TA||_F^2,
\end{equation}
where $A = [a_1, a_2, \ldots, a_n]^T$,  and $\mathbf{1}_n$ denotes the vector with all entries being one, see~\cite{boutsidis2009unsupervised}.
\end{example}

%


\begin{example}[Community Detection] 
Real-world network systems often have a community structure, which is the division of network nodes into groups such that the network connections are denser within the groups and are sparser between the groups, see \cite{newman2004finding}. These groups are called communities, or modules. A variety of community detection algorithms have been developed in recent years~\cite{newman2004fast,newman2006modularity,newman2007mixture,blondel2008fast,rosvall2008maps,yang2016modularity}. Among them, modularity optimization approaches have been shown to be highly effective in practical applications, see~\cite{fortunato2010community}. In~\cite{WHGV2021}, it is proven that in an ideal graph, the global minimizer of $f:\mathcal{A}_{\mathbf{1}_n} \rightarrow \mathbb{R}: X \mapsto - \mathrm{trace}(X^T M X)$ is a matrix that represents the ground truth, where $M = A - A \mathbf{1}_n \mathbf{1}_n^T A / (\mathbf{1}_n^T A \mathbf{1}_n)$ is the modularity matrix and $A$ is the adjacency matrix of the graph. In the presence of noise, the community detection is still formulated as the optimization problem
\begin{equation} \label{eq:app2}
\min_{X \in \mathcal{A}_{\mathbf{1}_n}} - \mathrm{trace}(X^T M X),
\end{equation}
under the assumption that the noise is not significant enough to change its minimizer.
\end{example}

\begin{example}[Other Graph Partitioning Techniques]
Beside community detection, other graph partitioning problems including general weighted graph cuts, such as ratio association, ratio cut, normalized cut, and Kernighan-Lin objective function,  can be formulated as an optimization problem, as shown in~\cite{DGK2005},
\begin{align} \label{eq:app6}
&\min_{\substack{Y^T D Y = I_q, Y^T Y \hbox{ is diagonal, }, Y \geq 0, \mathbf{1}_n \in \mathrm{span}(Y)}} - \mathrm{trace}(Y^T D K D Y),
\end{align}
where $K \in \mathbb{R}^{n \times n}$ is symmetric and $D \in \mathbb{R}^{n \times n}$ is a diagonal matrix with all entries being positive. These graph partitioning problems have been used in many areas, such as circuit layout~\cite{CS1994} and image segmentation~\cite{SM2000}. Letting $X$ denote $D^{1/2} Y$, it follows that $X \in \mathcal{A}_{v}$, where $v$ is a vector formed by the square roots of the diagonal entries in $D$, i.e., $v = \mathrm{diag}(D^{1/2})$. Therefore, Problem~\eqref{eq:app6} can be reformulated into
\begin{equation} \label{eq:app7}
\min_{X \in \mathcal{A}_v} - \mathrm{trace} (X^T D^{1/2} K D^{1/2} X),
\end{equation}
which is in the form of~\eqref{prob1}.
\end{example}

\paragraph{Connections between Problem~\eqref{prob3} and Problem~\eqref{prob1}} Problem~\eqref{prob1} and Problem~\eqref{prob3} are connected in such a way that the problems above can be solved with the latter using techniques developed in this paper. Problem~\eqref{prob1} can be reformulated by replacing the non-negative constraints $X \geq 0$ with a sparsity constraint $\|X\|_0 = n$ where $\|X\|_0$ corresponds to the total number of nonzero elements in $X$,  which yields
\begin{equation} \label{prob2}
\min_{X \in \mathcal{B}_v} f(X),
\end{equation}
where 
\begin{equation*}
\mathcal{B}_v = \{X \in \mathbb{R}^{n \times q} : X^T X = I_q, \|X\|_0 = n, v \in \mathrm{span}(X) \}.
\end{equation*}
Problem~\eqref{prob1} and Problem~\eqref{prob2} are essentially equivalent in the sense that their solutions are connected, see in Lemma~\ref{le06}.
\begin{lemma} \label{le06}
Consider Problem~\eqref{prob1} and Problem~\eqref{prob2} with the objective function $f$ satisfying $f(X) = f(X D_i)$ for any $i$, where $D_i = \mathrm{diag}(1, \ldots, 1, -1, 1, \ldots, 1)$, i.e., the $i$-th diagonal entry of $D_i$ is -1. The following two statements hold:
\begin{itemize}
	\item Let $X$ be any matrix in $\mathcal{B}_v$. Then for any column of $X$, denoted by $x_i$, the signs of all nonzero entries in $x_i$ are the same.
	\item Define a mapping $\vartheta: \mathbb{R}^{n \times q} \mapsto \mathbb{R}^{n \times q}: X \mapsto \hat{X} = X D_{j_1} D_{j_2} \ldots D_{j_s}$, where $j_1$, $j_2$, $\ldots$, $j_s$ are the indices of the columns of $X$ whose nonzero entries are all negative. Then $X_*$ is a global minimizer of Problem~\eqref{prob2} in the sense that $f(X_*) \leq f(Y), \forall Y \in \mathcal{B}_v$ if and only if $\vartheta(X_*)$ is a global minimizer of Problem~\eqref{prob1} in the sense that $f(\vartheta(X_*)) \leq f(Z), \forall Z \in \mathcal{A}_v$.
\end{itemize}
\end{lemma}
\begin{proof}
The first statement holds by the definition of $\mathcal{B}_v$. The second statement holds by the assumption on the function $f$.
\end{proof}
Due to the constraints of $\mathcal{B}_v$, the sparsest matrix in~$\mathcal{B}_v$ has $n$ nonzero entries. We reformulate Problem~\eqref{prob2} and use one norm penalization to promote the sparsity of $X$, which yields a continuous optimization problem in~\eqref{prob3}. Using one norm to promote sparsity on manifold is not new and has been widely used for the Stiefel manifold, see e.g.,~\cite{JoTrUd2003a,XBC2022}.
If the minimizer of~\eqref{prob3}, denoted by $X_*$, is sufficiently  close to $\mathcal{B}_v$, then one can find the closest matrix in $\mathcal{B}_v$ by a mapping $P_{\mathcal{B}_{v}} (X_*)$, see Lemma~\ref{le07}. If the element in the $i$-th row $j$-th column of $P_{\mathcal{B}_{v}} (X_*)$ is not zero, then the $i$-th object is in the $j$-th cluster. 
\begin{lemma} \label{le07}
Let $v \in \mathbb{R}^n$ be a positive vector, $W$ denote $\diag(v)$, $Y$ denote a matrix in $\mathcal{B}_v$, $d_i$ denote the number of nonzero entries in $i$-th column of $Y$, $u_i \in \mathbb{R}^{d_i}$ denote the vector forming by the nonzero entries of the $i$-th column of $Y$, and $u \in \mathbb{R}^n$ denote $(u_1^T \; u_2^T \; \ldots \; u_q^T)^T$. If $X_* \in \mathcal{F}_v$ is sufficiently close to $Y$, 
then 
it holds that
\[
Y = P_{\mathcal{B}_{v}} (X_*),
\]
where $P_{\mathcal{B}_{v}} (X_*) = W P_{\mathcal{B}_{\mathbf{1}_n}} ( W^{-1} X_*)$, 
$P_{\mathcal{B}_{\mathbf{1}_n}} (X_*) = \left( \frac{b_1}{\|b_1 \odot v\|} \; \ldots \; \frac{b_q}{\|b_q \odot v \|} \right)$, $\odot$ denotes the Hadamard product,
$b_j \in \mathbb{R}^n$ for $j = 1, 2, \ldots, q$, and
\[
(b_j)_i = 
\left\{
\begin{array}{cc}
\mathrm{sign}( (X_*)_{i j} ) & \hbox{if $(X_*)_{i j}$ has the largest magnitude in the $i$-th row;\footnotemark} \\
0 & \hbox{otherwise.}
\end{array}
\right.
\]
\end{lemma}

\begin{proof}
Without loss of generality, assume that $Y$ has the form
\begin{equation*}
Y = \diag(u_1, u_2, \ldots, u_q) :=
\begin{pmatrix}
	u_1 & 0 & \ldots & 0 \\
	0 & u_2 & \ldots & 0 \\
	\ldots & \ldots & \ddots & \vdots \\
	0 & 0 & \ldots & u_q
\end{pmatrix}	
\end{equation*}
Partition the vector $v$ by $v = (v_1^T \; v_2^T \; \ldots v_q^T)^T$, where $v_i \in \mathbb{R}^{d_i}$. It follows that $u_i = s_i v_i / \|v_i\|, i = 1, 2, \ldots, q$, where $s_i$ is either one or negative one. Therefore, $W^{-1} Y = \diag( \frac{s_1}{\|v_1\|} \mathbf{1}_{d_1}, \frac{s_2}{\|v_2\|} \mathbf{1}_{d_2}, \ldots, \frac{s_q}{\|v_q\|} \mathbf{1}_{d_q} )$.
If $W^{-1}X_*$ is sufficiently close to $W^{-1}Y$ in the sense that the location of the largest magnitude entry of each row does not change, then it holds that 
\[
W^{-1} Y = P_{\mathcal{B}_{\mathbf{1}_n}} (W^{-1} X_*),
\]
which implies $Y = P_{\mathcal{B}_{v}} (X_*)$.
\end{proof}
Overall, Problems~\eqref{prob3},~\eqref{prob1} and~\eqref{prob2} are closely connected and can be viewed as optimization models for solving the same applications.

\subsection{Our contribution}
In this paper, we propose a new optimization model given in~\eqref{prob3} to characterize clustering problems, 
including $k$-means model, community detection, normalized cut, and other graph partitioning techniques. 
It is proven here that the domain $\mathcal{F}_v$ forms an embedded submanifold of $\mathbb{R}^{n \times q}$. 
A family of computationally efficient retractions is developed. An orthonormal basis of any normal space of $\mathcal{F}_{v}$ is given. Such a basis yields a computationally efficient characterization of the normal space and is important in Riemannian proximal gradient methods. 
An accelerated Riemannian proximal gradient method that does not require solving the Riemannian proximal mapping exactly is developed, analyzed, and evaluated. 
Note that most existing Riemannian proximal gradient methods require solving their subproblems exactly~\cite{CMSZ2019,HuaWei2019b,HuaWei2019}. Compared to the only existing inexact Riemannian proximal gradient method~\cite{HuaWei2021}, the algorithm proposed in this paper allows adaptive step sizes rather than an unknown fixed step size and still guarantees the global convergence.
In the numerical experiments, the proposed model and optimization algorithm is shown to have performance superior to existing state-of-the-art algorithms in community detection and normalized cut problems. 

\subsection{Related work} 

To the best of our knowledge, literature does not consider Problems~\eqref{prob3},~\eqref{prob1}, or~\eqref{prob2} for generic functions $f$. They focus on some special formulations of the objective $f$. If the function $f$ is given by 
\begin{equation} \label{e09}
f(X) = -\trace(X^T M X)
\end{equation}
with positive semidefinite matrix $M$, then Problem~\eqref{prob2} with $v = \mathbf{1}_n$ is a commonly-encountered objective function in the task of clustering. Spectral-type of clustering algorithms and $k$-means-based algorithms have been proposed~\cite{YS2003,DGK2004,YZW2008}. For example, the spectral clustering algorithm in~\cite{YS2003} first finds a basis of the eigenspace that corresponds to the $q$ largest eigenvalues, and then finds a matrix in $\mathcal{C} := \{X \in \mathbb{R}^{n \times q} : X^T X \hbox{ is diagonal}, X \mathbf{1}_q = \mathbf{1}_n \}$ that is closest to the eigenspace. The kernel $k$-means algorithm views the matrix $M$ as a kernel matrix. It follows that the standard $k$-means algorithm can be used~\cite{DGK2004}. 

For a generic objective $f$, the closest formulation to Problem~\eqref{prob1} is given in~\cite{JMWC2019,QPX2021}. The formulation therein is Problem~\eqref{prob1} with the constraint $v \in \mathrm{span}(X)$ being dropped, that is, 
\begin{equation} \label{e41}
 \min_{X^TX = I_q, X\geq 0} f(X). 
\end{equation}
The papers~\cite{JMWC2019,QPX2021} use different approaches to reformulate Problem~\eqref{e41}. The former~\cite{JMWC2019} develops a penalty method by penalizing the orthogonal constraints in $X$ and keeping the nonnegative and multiple sphere constraints, and the latter~\cite{QPX2021} keeps the orthonormal constraints and penalizes the nonnegative constraints.
If the feasible set $\mathcal{F}_v$ is a manifold, then Problem~\eqref{prob3} can be optimized by Riemannian proximal gradient methods in~\cite{HuaWei2019b,HuaWei2021}. However, it is not considered in~\cite{HuaWei2019b,HuaWei2021} if the set $\mathcal{F}_v$ in Problem~\eqref{prob3} is a manifold.

A related work given in~\cite{CMV2017} considers reformulating the $k$-means clustering problem as an optimization problem on the manifold $\mathcal{F}_{\mathbf{1}_n}$. Rather than promoting the sparsity, the authors propose to penalize the negativity of entries in $X$ and reformulate the $k$-mean clustering problem as
\[
\min_{X \in \mathcal{F}_{\mathbf{1}_n}} -\trace(X^T M X) + \lambda \|X_{-}\|_{\F}^2,
\]
where $X_{-}$ indicates the negative entries of $X$. Though the paper~\cite{CMV2017} states that $\mathcal{F}_{\mathbf{1}_n}$ is a manifold, a rigorous proof is not given. Moreover, we consider a more general case for $\mathcal{F}_{v}$ with a positive vector $v$. The retraction proposed in~\cite{CMV2017} involves a computation of an $n$-by-$n$ matrix exponential ($O(n^3)$ flops), which is unacceptable for large $n$. The retraction in the present paper (see Theorem~\ref{th01}) only takes $O(n p^2)$ flops and can be used for large scale problems\footnote{Throughout this paper, the computational complexity is measured by flop counts. A flop is a floating point operation \cite[Section 1.2.4]{GV96}.}. In addition, this paper gives an orthonormal basis of the normal space of $\mathcal{F}_v$ which allows a computationally efficient characterization for the normal space. The well-known Riemannian steepest descent method is used in~\cite{CMV2017} while we propose a more sophisticated and effective inexact Riemannian proximal gradient method for the nonsmooth cost function in~\eqref{prob3}.

A preliminary version of this paper is given in~\cite{WHGV2021}, which focuses on the community detection problem~\cite{Newman2006} and therefore only considers $\mathcal{F}_{\mathbf{1}_n}$. In addition, the paper~\cite{WHGV2021} uses iterations of the existing Riemannian proximal gradient method in~\cite{CMSZ2019} on the Stiefel manifold for the problem
\[
\min_{x \in \St(q, n)} - \trace(X^T M X) + \lambda \|X\|_1,
\]
and projects every iterate $x_k$ onto $\mathcal{F}_{\mathbf{1}_n}$. Such a method, called Riemannian projected proximal gradient method, is neither guaranteed to generate descent iterates in the sense of the function value nor guaranteed to convergence globally. The geometric structure of the constrained set is not explored either.


An inexact Riemannian proximal gradient method has been proposed in~\cite{HuaWei2021} which is based on the Riemannian proximal gradient method~\cite{HuaWei2019b}. The version in~\cite{HuaWei2021} is insightful from theoretical aspects, that is, theoretical conditions that guarantee local convergence rate are given. However, those results rely on a sufficiently large parameter $\tilde{L}$ in the Riemannian proximal mapping
\begin{equation} \label{e43}
\hat{\eta}_x \approx \argmin_{\eta \in \T_x \mathcal{M}} \trace(\grad f(x)^T \eta) + \frac{\tilde{L}}{2} \|\eta\|^2 + g(R_x(\eta)),
\end{equation}
and the use of a fixed step size, where $R_x$ denotes a retraction on $\mathcal{M}$. In practice, a sufficiently large $\tilde{L}$ is usually unknown and estimating $\tilde{L}$ requires extra work. 
The inexact Riemannian proximal gradient method proposed in the present paper avoids this problem, allows adaptive step size, guarantees global convergence, and therefore is preferable from the point of view of computational efficiency and provable robustness.


\subsection{Organization}
This paper is organized as follows. Section~\ref{sect:notation} gives the notation and preliminaries. Section~\ref{sect:ManifoldF} proves that the set $\mathcal{F}_v$ is a manifold and the optimization-related geometry tools are also derived. Section~\ref{sec:Alg} gives an inexact Riemannian proximal gradient method and its global convergence analysis. The numerical experiments are shown in Section~\ref{sect:NumExp}. Finally, the conclusion and future work are stated in Section~\ref{sect:con}.

\section{Notation and Preliminaries} \label{sect:notation} 

Unless otherwise indicated, the Riemannian concepts of this paper follow from the standard literature, e.g.,~\cite{Boo1986,AMS2008} and the related notation  follows from~\cite{AMS2008}. A Riemannian manifold $\mathcal{M}$ is a manifold endowed with a Riemannian metric $(\eta_x, \xi_x) \mapsto \inner[x]{\eta_x}{ \xi_x} \in \mathbb{R}$, where $\eta_x$ and $\xi_x$ are tangent vectors in the tangent space of $\mathcal{M}$ at $x$.  The induced norm in the tangent space at $x$ is denoted by $\|\cdot\|_x$. 
Throughout this paper, unless otherwise indicated, we use the Euclidean metric, i.e., $\inner[]{U}{V} = \mathrm{trace}(U^T V)$ and $\|U\| = \sqrt{\inner[]{U}{V}} = \|U\|_{\F}$. Therefore, the subscript of $\inner[]{\cdot}{\cdot}$ and $\|\cdot\|$ can be omitted. The tangent space of the manifold $\mathcal{M}$ at $x$ is denoted by $\T_x \mathcal{M}$, and the tangent bundle, which is the disjoint union of all tangent spaces, is denoted by $\T \mathcal{M}$. An open ball on a tangent space is denoted by $\mathbb{B}(x, r) = \{\xi_x \in \T_x \mathcal{M} \mid \|\xi_x\| < r \}$.
The Riemannian gradient of a function $h:\mathcal{M} \rightarrow \mathbb{R}$
, denoted  $\grad h(x)$, is the unique tangent vector satisfying:
$
\D h(x) [\eta_x] = \inner[x]{\eta_x}{\grad h(x)}, \forall \eta_x \in \T_x \mathcal{M},
$ 
where $\D h(x) [\eta_x]$ denotes the directional derivative of $h$ along the direction $\eta_x$.



A retraction, by definition, is a smooth ($C^\infty$) mapping from the tangent bundle to the manifold such that 
\begin{align}
&\hbox{(i)~$R(0_x) = x$ for all $x \in \mathcal{M}$, and} \label{e06} \\
&\hbox{(ii) $\frac{d}{d t} R(t \eta_x) \vert_{t = 0} = \eta_x$ for all $\eta_x \in \T_x \mathcal{M}$} \label{e07}
\end{align}
where $0_x$ denotes the origin of $\T_x \mathcal{M}$. Moreover, $R_x$ denotes the restriction of $R$ to $\T_x \mathcal{M}$.
The domain of $R$ does not need to be the entire tangent bundle. If the manifold $\mathcal{M}$ is a compact embedded submanifold of $\mathbb{R}^n$, then by~\cite{BAC2018}, there exist two positive constants $M_1$ and $M_2$ such that
\begin{align}
&\|R_x(\eta_x) - x\| \leq M_1 \|\eta_x\| \label{e26} \\
&\|R_x(\eta_x) - x - \eta_x\| \leq M_2 \|\eta_x\|^2, \label{e27}
\end{align}
hold for any $x \in \mathcal{M}$ and $\eta_x \in \T_x \mathcal{M}$.
A vector transport $\mathcal{T}: \T \mathcal{M} \oplus \T \mathcal{M} \rightarrow \T \mathcal{M}: (\eta_x, \xi_x) \mapsto \mathcal{T}_{\eta_x} \xi_x$ associated with a retraction $R$ is a smooth mapping such that, for all $(x, \eta_x)$ in the domain of $R$ and all $\xi_x \in \T_x \mathcal{M}$, it holds that (i) $\mathcal{T}_{\eta_x} \xi_x \in \T_{R(\eta_x)} \mathcal{M}$, (ii) $\mathcal{T}_{0_x} = \mathrm{id}$, and (iii) $\mathcal{T}_{\eta_x}$ is a linear map, where $\id$ denotes the identity operator.
The vector transport by differential retraction $\mathcal{T}_R$ is defined by $\mathcal{T}_{R_{\eta_x}} \xi_x = \frac{d}{d t} R_{x}(\eta_x + t \xi_x) \vert_{t = 0}$. 

The Stiefel manifold $\St(q, n)$ is defined by $\St(q, n) = \{X \in \mathbb{R}^{n \times q} : X^T X = I_p\}$. The tangent space of $\St(q, n)$ at $X$ is 
\begin{equation} \label{StTangent}
\T_X \St(q, n) = \{X \Omega + X_{\perp} K : \Omega^T = - \Omega, K \in \mathbb{R}^{(n - q) \times q} \},
\end{equation}
where $X_{\perp} \in \mathbb{R}^{n \times (n - q)}$ is a matrix with orthonormal columns such that $\begin{pmatrix} X & X_{\perp} \end{pmatrix}$ is orthonormal.

$0_n$ denotes a vector with length $n$ and all entries zero and $0_{m \times n}$ denotes a $m$-by-$n$ matrix with all entries zero. $I_s$ denotes the $s$-by-$s$ identity matrix. The subscript of $0$ or $I$ is omitted if its size is clear from the context. Given an $n$-by-$n$ matrix $M$, $e^M$ denotes the matrix exponential.

\section{Manifold Structure of $\mathcal{F}_v$} \label{sect:ManifoldF}


In this section, we prove that the set $\mathcal{F}_v$ forms an embedded submanifold of $\mathbb{R}^{n \times q}$ and derive optimization-related tools. 
Theorem~\ref{thm:Feasible_manifold} show that $\mathcal{F}_v$ with $v > 0$ is an embedded submanifold of $\St(q, n)$. 


\begin{theorem}
\label{thm:Feasible_manifold}
The set $\mathcal{F}_v$ is an embedded submanifold of $\St(q, n)$ with dimension $\mathrm{dim}(\St(q, n)) - (n - q) = nq - q(q+1)/2 - n + q$. Furthermore, $\mathcal{F}_v$ is an embedded submanifold of $\mathbb{R}^{n \times q}$ with the same dimension and $\mathcal{F}_v$ is compact.
\end{theorem}

\begin{proof}
We verify that $\mathcal{F}_v$ is an embedded submanifold of $\St(q, n)$ by following~\cite[Definition~8.70]{boumal2020intromanifolds} using the notion of a local defining function. 
For any $X \in \mathcal{F}_v$, let $X_\perp$ be a matrix such that $\begin{pmatrix} X & X_\perp \end{pmatrix}^T \begin{pmatrix} X & X_\perp \end{pmatrix} = I_n$. Therefore, by~\cite[(2.23)]{EAS98}, we have that for any $V \in \T_X \St(q, n)$,
\[
\Exp_X(V) = 
\begin{pmatrix}
X & X_{\perp}
\end{pmatrix}
e^
{
\begin{pmatrix}
\Omega & -K^T \\
K & 0
\end{pmatrix}
}
\begin{pmatrix}
I_p \\
0
\end{pmatrix},
\]
defines the exponential mapping with respect to the canonical metric, where $\Exp$ denotes the matrix exponential,  $V = X \Omega + X_{\perp} K$, and the canonical metric is $\inner[X]{\eta_X}{\xi_X} = \mathrm{tr}(\eta_X^T (I_n - \frac{1}{2} X X^T) \xi_X)$ for $\eta_X, \xi_X \in \T_X \St(q, n)$. By~\cite[Theorem~3.7]{dC92}, there exists a positive constant $\delta > 0$ such that $\Exp_X$ is a diffeomorphism in $\mathbb{B}(X, \delta)$. It follows that for any $Y \in \Exp_X(\mathbb{B}(X, \delta))$, the mapping $\Exp_X^{-1}(Y)$ is well-defined and $Y = \Exp_X(\Exp_X^{-1}(Y))$, i.e.,
\begin{equation} \label{e01}
Y = 
\begin{pmatrix}
X & X_{\perp}
\end{pmatrix}
e^{
\begin{pmatrix}
X^T \Exp_X^{-1}(Y) & -\left( X_{\perp}^T \Exp_X^{-1}(Y) \right)^T \\
X_{\perp}^T \Exp_X^{-1}(Y) & 0
\end{pmatrix}
}
\begin{pmatrix}
I_q \\
0
\end{pmatrix}.
\end{equation}
Define a function $\phi: \Exp_X(\mathbb{B}(x, \delta)) \rightarrow \mathbb{R}^{n \times (n - q)}$ by
\begin{equation*}
\phi(Y) = 
\begin{pmatrix}
X & X_{\perp}
\end{pmatrix}
e^{
\begin{pmatrix}
X^T \Exp_X^{-1}(Y) & -\left( X_{\perp}^T \Exp_X^{-1}(Y) \right)^T \\
X_{\perp}^T \Exp_X^{-1}(Y) & 0
\end{pmatrix}
}
\begin{pmatrix}
0 \\
I_{n-q}
\end{pmatrix},
\end{equation*}
it follows from~\eqref{e01} that 
\begin{equation} \label{e02}
\phi(Y)^T Y = 0_{(n - q) \times q}.
\end{equation}
Since $\Exp_X^{-1}$ is smooth in $\Exp_X(\mathbb{B}(x, \delta))$, $\phi$ is a smooth function in its domain. Furthermore, it follows from $\phi(X)^T X_\perp = I_{n - q}$ that there exists a constant $\tilde{\delta} > 0$ such that $\phi(Z)^T X_{\perp}$ is full rank, i.e.,
\begin{equation} \label{e03}
\mathrm{rank} \left( \phi(Z)^T X_{\perp} \right) = n - q,
\end{equation}
for any $Z \in \Exp_X(\mathbb{B}(x, \tilde{\delta}))$. Let $\hat{\delta} = \min(\delta, \tilde\delta)$ and $\mathcal{N}_X = \Exp_X(\mathbb{B}(x, \hat\delta))$. We now define a function $h$ by
$
h:\mathcal{N}_X \rightarrow \mathbb{R}^{n - q}: Y \mapsto h(Y) = \phi(Y)^T v.	
$
Next, we verify that the function $h$ is a local defining function in the sense that $h^{-1}(0) = \mathcal{N}_X \cap \mathcal{F}_v$ and $\D h(Y): \T_Y \St(q, n) \rightarrow \mathbb{R}^{n - q}$ is surjective for any $Y \in \mathcal{N}_X$\footnote{Note that the local definition function only requires $\D h(Y)$ to be full rank at $Y = X$. Here, we prove a stronger result.}.

For any $Z \in h^{-1}(0)$, it holds that $\phi(Z)^T v = 0$ and $Z \in \St(q, n)$. Since $\phi(Z)^T v = 0$ implies $v \in \mathrm{span}(X)$, we have $Z \in \mathcal{F}_v$, which means $h^{-1}(0) \subseteq \mathcal{N}_X \cap \mathcal{F}_v$. On the other hand, for any $Z \in \mathcal{N}_X \cap \mathcal{F}_v$, it is obvious that $h(Z) = 0$, which means $\mathcal{N}_X \cap \mathcal{F}_v \subseteq h^{-1}(0)$. Overall, the equation $h^{-1}(0) = \mathcal{N}_X \cap \mathcal{F}_v$ holds.

Let $V$ denote $\Exp_X^{-1}(Y)$. For any $U \in \T_{\Exp_X(-V)} \St(q, n)$, let $\dot{V} = \mathcal{T}_{\Exp_{-V}}^{-1} (- U)$ and $W = \mathcal{T}_{\Exp_V} \dot{V}$, where $\mathcal{T}_{\Exp}$ denotes the vector transport by differentiating the exponential mapping~\eqref{e01}. Note that $\dot{V}$ is well-defined since $-V \in \mathbb{B}(X, \hat{\delta})$. We have
\begin{align*}
&\D h(Y) \left[ W \right] = \left( \D \phi(Y)\left[ W \right] \right)^T v \\
=& \left( 
\begin{pmatrix}
X & X_{\perp}	
\end{pmatrix}
\D e^{
\begin{pmatrix}
X^T \Exp_X^{-1}(Y) & -\left( X_{\perp}^T \Exp_X^{-1}(Y) \right)^T \\
X_{\perp}^T \Exp_X^{-1}(Y) & 0
\end{pmatrix}
}
\left[W\right]
\begin{pmatrix}
0 \\
I_{n-q}
\end{pmatrix}
\right)^T v \\
=&
\begin{pmatrix}
0 & I_{n-q}
\end{pmatrix}
\left\{
\D \left( 
e^{
\begin{pmatrix}
- X^T \Exp_X^{-1}(Y) & \left( X_{\perp}^T \Exp_X^{-1}(Y) \right)^T \\
- X_{\perp}^T \Exp_X^{-1}(Y) & 0
\end{pmatrix}
}
\begin{pmatrix}
I_p \\
0
\end{pmatrix}
\right)
\left[W\right]
\right\} \alpha,
\end{align*}
where $\alpha = X^T v$, $(e^A)^T = e^{-A}$ for any skew symmetric matrix $A$, and $X_{\perp} v = 0$.

Define the functions
\begin{gather*}
G(Y) = e^{
\begin{pmatrix}
- X^T \Exp_X^{-1}(Y) & \left( X_{\perp}^T \Exp_X^{-1}(Y) \right)^T \\
- X_{\perp}^T \Exp_X^{-1}(Y) & 0
\end{pmatrix}
}
\begin{pmatrix}
I_p \\
0
\end{pmatrix}, \hbox{ and } \\
H\begin{pmatrix}
\Omega \\
K
\end{pmatrix}
=
\Exp_X(X \Omega + X_\perp K).
\end{gather*}
Since $V, \dot{V} \in \T_X \St(q, n)$, there exist skew symmetric matrices $\Omega_V, \Omega_{\dot{V}}$ and matrices $K_V, K_{\dot{V}}$ such that $V = X \Omega_V + X_{\perp} K_V$ and $\dot{V} = X \Omega_{\dot{V}} + X_{\perp} K_{\dot{V}}$. By the chain rule $\D G\circ H \begin{pmatrix} \Omega_V \\ K_V \end{pmatrix} \left[ \begin{pmatrix} \Omega_{\dot{V}} \\ K_{\dot{V}} \end{pmatrix} \right] = \D G \left(H \begin{pmatrix} \Omega_V \\ K_V \end{pmatrix} \right) \left[ \D H \begin{pmatrix} \Omega_V \\ K_V \end{pmatrix} \left[ \begin{pmatrix} \Omega_{\dot{V}} \\ K_{\dot{V}} \end{pmatrix} \right] \right]$ and $W = \mathcal{T}_{\Exp_V} \dot{V}$, we have that
\begin{align*}
&\D e^{
\begin{pmatrix}
- X^T \Exp_X^{-1}(Y) & \left( X_{\perp}^T \Exp_X^{-1}(Y) \right)^T \\
- X_{\perp}^T \Exp_X^{-1}(Y) & 0
\end{pmatrix}
}
\begin{pmatrix}
I_p \\
0
\end{pmatrix}
\left[\mathcal{T}_{R_V} \dot{V}\right]
\\
=& \D \left(
e^{
\begin{pmatrix}
- \Omega_V & K_V^T \\
-K_V & 0
\end{pmatrix}
}
\begin{pmatrix}
I_p \\
0
\end{pmatrix}
\right)
\left[
\begin{pmatrix}
\Omega_{\dot{V}} \\
K_{\dot{V}}
\end{pmatrix}
\right].
\end{align*}
It follows that
\begin{align}
&\D h(Y) \left[ W \right] = 
-
\begin{pmatrix}
0 & I_{n-q}
\end{pmatrix}
\D \left(
e^{
\begin{pmatrix}
- \Omega_V & K_V^T \\
-K_V & 0
\end{pmatrix}
}
\begin{pmatrix}
I_p \\
0
\end{pmatrix}
\right)
\left[
\begin{pmatrix}
- \Omega_{\dot{V}} \\
- K_{\dot{V}}
\end{pmatrix}
\right]
\alpha \nonumber \\
=& \begin{pmatrix}
0 & I_{n-q}
\end{pmatrix}
\begin{pmatrix}
X & X_{\perp}
\end{pmatrix}^T
\mathcal{T}_{\Exp_{-V}} (\dot{V})
= \begin{pmatrix}
0 & I_{n-q}
\end{pmatrix}
\begin{pmatrix}
X & X_{\perp}
\end{pmatrix}^T
\mathcal{T}_{\Exp_{-V}} ( \mathcal{T}_{\Exp_{-V}}^{-1} (- U) )
 \nonumber \\
=& 
-
\begin{pmatrix}
0 & I_{n-q}
\end{pmatrix}
\begin{pmatrix}
X & X_{\perp}
\end{pmatrix}^T
U \alpha = - X_{\perp}^T U \alpha. \label{e04}
\end{align}
Since $U$ can be any tangent vector in $\T_{\Exp_X(-V)} \St(q, n)$ and $X_{\perp}^T \phi(\Exp_X(-V))$ is full rank by~\eqref{e03}, the vector $- X_{\perp}^T U \alpha$ can be any one in $\mathbb{R}^{n - q}$. Therefore, $\D h(Y)$ is surjective and also full rank. Therefore, by~\cite[Definition~8.70]{boumal2020intromanifolds}, $\mathcal{F}_v$ is an embedded submanifold of $\St(q, n)$. Furthermore, by~\cite[Exercise~3.33]{boumal2020intromanifolds}, $\mathcal{F}_v$ is also an embedded submanifold of $\mathbb{R}^{n \times q}$.

Since $\mathcal{F}_v$ is a subset of $\St(q, n)$, it is a bounded set. Moreover, it is easy to show that its complement set in $\mathbb{R}^{n \times q}$ is an open set. Therefore, $\mathcal{F}_v$ is a closed set. It follows that $\mathcal{F}_v$ is compact.
\end{proof}

Theorem~\ref{th02} gives the tangent space at any $X \in \mathcal{F}_v$ and its perpendicular space with respect to the Euclidean metric.

\begin{theorem} \label{th02}
Let $\mathcal{F}_v$ be the embedded submanifold of $\mathbb{R}^{n \times q}$. The tangent space of $\mathcal{F}_v$ at $X$ is given by
\[
\T_X \mathcal{F}_v = \{X \Omega + X_\perp K : \Omega^T = - \Omega, K \in \mathbb{R}^{(n - q) \times q}, K X^T v = 0 \}
\]
and the perpendicular space of $\T_X \mathcal{F}_v$ with respect to the Euclidean metric, called the normal space at $X$, is given by
\[
\N_X \mathcal{F} = \{X S + X_{\perp} u v^T X : S = S^T, u \in \mathbb{R}^{n - q} \}.	
\]
\end{theorem}
\begin{proof}
It follows from~\cite[Exercise~3.33]{boumal2020intromanifolds} that $\T_X \mathcal{F}_v = \mathrm{ker} \D h(X)$. By~\eqref{e04}, we have that for any $U \in \T_X \St(q, n)$, $\D h(X) [U] = - K X^T v$, where $U = X \Omega + X_{\perp} K$ and $\Omega$ is any skew symmetric matrix and $K$ is any $n$-by-$(n-q)$ matrix. Therefore, it holds that $\mathrm{ker} \D h(X) = \{X \Omega + X_\perp K : \Omega^T = - \Omega, K \in \mathbb{R}^{(n - q) \times q}, K X^T v = 0 \}$.

For any $V \in \T_X \mathcal{F}_v$ and $U \in N_X \mathcal{F}_v$, it is easy to verify that $\mathrm{trace}(U^T V) = 0$. In addition, $\mathrm{dim}(T_X \mathcal{F}_v) + \mathrm{dim}(\N_X \mathcal{F}_v) = nq - n - q - q (q + 1) / 2 + q (q + 1) / 2 + n + q = nq = \mathrm{dim}(\mathbb{R}^{ n \times q })$. Therefore, $N_X \mathcal{F}_v = \left( \T_X \mathcal{F}_v \right)^{\perp}$ which implies $\N_X \mathcal{F}_v$ is the normal space of $\mathcal{F}_v$ at $X$.
\end{proof}

\begin{theorem}
Given any $Z \in \mathbb{R}^{n \times q}$, the orthogonal projection to $\N_X \mathcal{F}_v$ is given by
\[
P_{\N_X} (Z) = X \frac{X^T Z + Z^T X}{2} + (I - X X^T) Z \hat{\alpha} \hat{\alpha}^T,
\]
where $\hat{\alpha} = X^T v / \|X^T v\|$. The orthogonal projection to $\T_X \mathcal{F}_v$ is therefore
\[
P_{\T_X} (Z) = X \frac{X^T Z - Z^T X}{2} + (I - X X^T) Z (I - \hat{\alpha} \hat{\alpha}^T).
\]
\end{theorem}
\begin{proof}
By observing the formats of $P_{\N_X} (Z)$ and $P_{\T_X}(Z)$, we have $P_{\N_X} (Z) \in \N_X \mathcal{F}_v$ and $P_{\T_X}(Z) \in \T_X \mathcal{F}_V$. Therefore, the result follows from
$
P_{\N_X} (Z) + P_{\T_X}(Z) = X X^T Z + (I - X X^T) Z = Z.
$
\end{proof}

Given $X \in \St(q, n)$, the orthonormal projection from $X$ to $\mathcal{F}_v$ with $v = \mathbf{1}_n$ has been derived in~\cite{WHGV2021}. The orthonormal projection with any $v > 0$ can be derived similarly. We state the result without proof in Lemma~\ref{le01}.
\begin{lemma} \label{le01}
For any $X \in \St(q, n)$ with $X^T v \neq 0$, the global minimizer of the problem
$
P_{\mathcal{F}_v}(X) = \argmin_{Y \in \mathcal{F}_v} \|X - Y\|_F^2	
$
is given by $Y_* = v q_*^T / \|v\|_2 + X (I - q_* q_*^T)$,
where $q_* = X^T v / \|X^T v\|_2$.
\end{lemma}

One way to define a retraction on $\mathcal{F}_v$ is by the orthogonal projection~\cite{AbsMal2012}, i.e.,
\begin{equation} \label{e05}
R_X^{\mathrm{proj}}(V) = P_{\mathcal{F}_v}(X + V),	
\end{equation}
where $X \in \mathcal{F}$ and $V \in \T_X \mathcal{F}$. However, we do not have a closed form solution of $P_{\mathcal{F}_v}(X + V)$ in general. A practical family of retractions is given in Theorem~\ref{th01}.

\begin{theorem} \label{th01}
For any $X \in \mathcal{F}_v$, there exists a positive number $\delta_X > 0$ such that the mapping
\begin{equation} \label{e08}
R_X: \mathbb{B}(x, \delta_X) \rightarrow \mathcal{F}_v: V \mapsto R_X(V) = P_{\mathcal{F}_v} \circ \tilde{R}_X
\end{equation}
satisfies the two conditions of the retraction, i.e.,~\eqref{e06} and~\eqref{e07}, where $\tilde{R}$ is any retraction on $\St(q, n)$. Moreover, if the retraction $\tilde{R}$ satisfies $\mathrm{span}(\tilde{R}_X(V)) = \mathrm{span}(X + V)$, then the domain of $\tilde{R}$ in~\eqref{e08} is the whole tangent bundle. Such retractions $\tilde{R}$ include the retraction by QR decomposition~\cite[(4.8)]{AMS2008} and the retraction by polar decomposition~\cite[(4.7)]{AMS2008}.
\end{theorem}
\begin{proof}
Since $X$ satisfies $X^T v \neq 0$, $\tilde{R}_X(0_X) = X$, and $\tilde{R}_X$ is smooth, there exists a positive $\delta_X>0$ such that $P_{\mathcal{F}_v} (\tilde{R}_X(V))$ is well-defined for any $V \in \mathbb{B}(X, \delta_X)$.
The smoothness of $R$ follows from the smoothness of $\tilde{R}$ and $P_{\mathcal{F}_v}$. We have $R_X(0_X) = P_{\mathcal{F}_v}(\tilde{R}_X(0_X)) = P_{\mathcal{F}_v}(X) = X$, where the second equality follows from the property of the retraction $\tilde{R}$ and the last equation follows from the definition of the projection $P_{\mathcal{F}_v}$.

In addition, we have
\begin{align*}
\frac{d}{d t} R_X(t V) \vert_{t=0} =& \frac{d}{d t} \left( P_{\mathcal{F}_v} \circ \tilde{R}_X \right) (t V) \vert_{t = 0} = \left( \D P_{\mathcal{F}_v} (\tilde{R}_X(t V) ) \left[ \frac{d}{d t} \tilde{R}_X(t V) \right] \right) \vert_{t = 0} \\
=& \D P_{\mathcal{F}_v} (X) [V] = \D R_X^{\mathrm{proj}}(0_X)[V] = V,
\end{align*}
where the second equality follows from the chain rule, the third equality follows from $\tilde{R}_X(0_V) = X$ and $\frac{d}{d t} \tilde{R}_X(t V) \vert_{t = 0} = V$, and the last equality follows from the fact that~\eqref{e05} is a retraction.

For the second part of the result, we only need to verify that $(X + V)^T v \neq 0$ for all $V \in \T_X \mathcal{F}_v$. Let $\alpha = X^T v \neq 0$. By the form of the tangent space $\T_X \mathcal{F}_v$ in Theorem~\ref{th02}, we have
$
\alpha^T (X + V)^T v = \alpha^T (X + X \Omega + X_{\perp} K)^T v = \alpha^T \alpha + \alpha^T \Omega \alpha = \|\alpha\|_2^2 \neq 0,
$
which implies $(X + V)^T v \neq 0$.
\end{proof}
By Theorem~\ref{th01}, two retractions of $\mathcal{F}_v$ based on QR decomposition and polar decomposition are respectively given by
\begin{align} \label{e44}
R_X^{\mathrm{qf}}(V) = v q_*^T / \|v\|_2 + \mathrm{qf}(X + V)	(I - q_* q_*^T)
\end{align}
where $q_* = \mathrm{qf}(X + V)^T v / \|\mathrm{qf}(X + V)^T v\|_2$ and $\mathrm{qf}(X + V)$ denotes the Q factor of the QR decomposition of $X + V$ that, moreover, has positive diagonal entries in the R factor; and
\begin{align} \label{e45}
R_X^{\mathrm{polar}}(V) = v q_*^T / \|v\|_2 + (X + V) (I + V^T V)^{-1/2} (I - q_* q_*^T),
\end{align}
where $q_* = (I + V^T V)^{-1/2} (X + V)^T  v / \|(I + V^T V)^{-1/2} (X + V)^T v\|_2$. Since the dominant part of the computations in~\eqref{e44} and~\eqref{e45} are respectively the QR decomposition and polar decomposition, their computations both takes $O(n p^2)$ flops. The retraction proposed in~\cite[(14)]{CMV2017} is computationally more expensive. Specifically, the retraction in~\cite{CMV2017} is given by
\begin{align} \label{e46}
R_X(V) = \exp(B) \exp(A') X,	
\end{align}
where $A = X^T V$, $A' = X A X^T$, and $B = V X^T - X V^T - 2 A'$. The computation of~\eqref{e46} requires an evaluation of an exponential of an $n$-by-$n$ matrix $B$ and therefore can be computationally unacceptable when $n$ is large.

The proposed proximal gradient method also relies on an orthonormal basis of the normal space of $\mathcal{F}_v$, which is given in Lemma~\ref{le05}.
\begin{lemma} \label{le05}
The set
\begin{align}
\mathcal{B}_X =& \{ X \mathbf{e}_i\mathbf{e}_i^T : i=1,\cdots, q \} \cup \{ \frac{1}{\sqrt{2}}X(\mathbf{e}_i\mathbf{e}_j^T +\mathbf{e}_j\mathbf{e}_i^T) : i=1,\cdots, q, j=i+1,\cdots, q\} \nonumber \\
&\cup \{ X_\perp \tilde{\mathbf{e}}_i \tilde{v}^T X ,  i=1, \cdots, n-q\}, \label{e47}
\end{align}
defines an orthonormal basis of $N_X \mathcal{F}$ with respect to the Euclidean metric, where $x \in \mathcal{F}_v$, $(\mathbf{e}_1, \cdots, \mathbf{e}_q)$ is the canonical basis of $\mathbb{R}^q$, $(\tilde{\mathbf{e}}_1, \cdots, \tilde{\mathbf{e}}_{n-q})$ is the canonical basis of $\mathbb{R}^{n-q}$ and $\tilde{v} = v / \|v\|$.
\end{lemma}
\begin{proof}
Let $T_{i j} = \frac{1}{\sqrt{2}}X(\mathbf{e}_i\mathbf{e}_j^T +\mathbf{e}_j\mathbf{e}_i^T)$ if $i \neq j$, $T_{i i} = X \mathbf{e}_i\mathbf{e}_i^T$ and $\tilde{T}_i =  X_\perp\tilde{\mathbf{e}}_i \tilde{v}^T X$.
It is easy to verify that if $i_1 = i_2$ and $j_1 = j_2$, then $\mathrm{trace}(T_{i_1 j_1}^T T_{i_2 j_2}) = 1$, otherwise $\mathrm{trace}(T_{i_1 j_1}^T T_{i_2 j_2}) = 0$ for all $i_1, j_1, i_2, j_2$; and that $\mathrm{trace}(T_{ij}^T \tilde{T}_{k}) = 0$, for all $i, j, k$. We also have
\begin{align*}
\mathrm{trace}(\tilde{T}_i^T \tilde{T}_j) =& \mathrm{trace}\left( X^T \tilde{v} \tilde{\mathbf{e}}_i^T X_{\perp}^T X_{\perp} \tilde{\mathbf{e}}_j \tilde{v}^T X  \right) = \mathrm{trace}\left(  \tilde{\mathbf{e}}_i^T \tilde{\mathbf{e}}_j \tilde{v}^T X X^T \tilde{v} \right) \\
=& \mathrm{trace}\left(  \tilde{\mathbf{e}}_i^T \tilde{\mathbf{e}}_j \tilde{v}^T \tilde{v} \right) = \tilde{\mathbf{e}}_i^T \tilde{\mathbf{e}}_j,
\end{align*}
where the second equality follows from $X_\perp^T X_\perp = I$ and the third equality follows from $\tilde{v} \in \mathrm{span}(X)$.
\end{proof}

Let $V_i$, $i = 1, \ldots q(q+1)/2 + n - q$ denote the entries in the basis $\mathcal{B}_X$. Define a function by
$
B_X: \mathbb{R}^{q (q + 1) / 2 + n - q}	\rightarrow \mathbb{R}^{n \times q}: u \rightarrow B_X u = \sum_{i = 1}^{q(q+1)/2 + n - q} u_i V_i \in N_x \mathcal{F}_v
$
and another function by
$
B_X^T: \mathbb{R}^{n \times p} \rightarrow \mathbb{R}^{q (q + 1) / 2 + n - q}: V \rightarrow u,
$
where $u_i = \mathrm{trace}(V^T V_i)$. These two functions are used in the proximal gradient method given in Section~\ref{sec:Alg}.

Though the matrix $X_{\perp}$ is used in the orthonormal basis~\eqref{e47}, one does not need to construct such a matrix and only needs to use two mappings $\alpha_X:\mathbb{R}^{n \times q} \rightarrow \mathbb{R}^{n \times q}: A \mapsto (X \; X_{\perp})^T A$ and $\beta_X: \mathbb{R}^{n \times q} \rightarrow \mathbb{R}^{n \times q}: A \mapsto (X \; X_{\perp}) A$. These two mappings can be computed efficiently ($O(n p^2)$), see details in~\cite[Algorithms~4 and~5]{HAG2016VT}. It follows that the mappsing $B_X^T$ and $B_X$ can be computed by $O(n p^2)$ flops and the detailed implementations are stated in Algorithms~\ref{algo:BXT} and~\ref{algo:BX}.

\begin{algorithm}
\caption{
Compute $B_X^T: \mathbb{R}^{n \times p} \rightarrow \mathbb{R}^{q (q + 1) / 2 + n - q}:V \mapsto B_X^T (V)$}
\label{algo:BXT}
\begin{algorithmic}[1]
\Require  $X \in \mathcal{F}_v$, $V \in \N_X \mathcal{F}_v$; a positive vector $v \in \mathbb{R}^n$; a function $\alpha_X:\mathbb{R}^{n \times p} \rightarrow \mathbb{R}^{n \times p}: A \mapsto (X \; X_\perp)^T A$.
\State $
          \begin{pmatrix}
            S \\
            K \\
          \end{pmatrix}
         = \alpha_X(V)$, where $S \in \mathbb{R}^{p \times p}$ and $K \in \mathbb{R}^{(n - q) \times q}$;                  
\State $k = 1$; 
\For {$j = 1, \ldots, p$}
 $u(k) = S_{ii}$ and $k \leftarrow k + 1$;
\EndFor
\For {$j = 2, \ldots, p$, $i = 1, \ldots j - 1$} 
 $u(k) = \sqrt{2} S_{i j}$, $u(k) = \sqrt{2} S_{j i}$, and $k \leftarrow k + 1$;
\EndFor
\State $z = K X^T v \|v\| / \| v X X^T v\|$;
\For {$j = 1 : n - q$}
 $u(k) = z_j$ and $k \leftarrow k + 1$;
\EndFor
\State return $u$;
\end{algorithmic}
\end{algorithm}

\begin{algorithm}
\caption{ 
Compute $B_X: \mathbb{R}^{q (q + 1) / 2 + n - q}	\rightarrow \mathbb{R}^{n \times q}: u \mapsto B_X(u)$}
\label{algo:BX}
\begin{algorithmic}[1]
\Require  $X \in \mathcal{F}_v$, $u \in \mathbb{R}^{q (q + 1) / 2 - n - q}$; a positive vector $v \in \mathbb{R}^n$; a function $\beta_X:\mathbb{R}^{n \times p} \rightarrow \mathbb{R}^{n \times p}: A \mapsto (X \; X_\perp) A$.
\State $k = 1$;
\For {$j = 1, \ldots, p$}
 $S_{i i} = u(k)$ and $k \leftarrow k + 1$;
\EndFor
\For {$j = 2, \ldots, p$, $i = 1, \ldots j - 1$} 
 $S_{i j} =  u(k) / \sqrt{2}$, $S_{j i} = u(k) / \sqrt{2}$, and $k \leftarrow k + 1$;
\EndFor
\For {$j = 1 : n - q$}
 $z_j = u(k)$ and $k \leftarrow k + 1$;
\EndFor
\State return $\beta_X
                                    \begin{pmatrix}
                                      S \\
                                      z v^T X / \|v\| \\
                                    \end{pmatrix}$; 
\end{algorithmic}
\end{algorithm}


\section{An Accelerated Manifold Proximal Gradient Method} \label{sec:Alg}

For simplicity of notation, throughout this section, we use lowercase $x$ to denote a point in the domain manifold $\mathcal{M} := \mathcal{F}_v$ and use $g(x)$ to denote $\lambda \|x\|_1$. Therefore, Problem~\eqref{prob3} becomes
\begin{equation} \label{prob4}
\min_{x \in \mathcal{M}}	 F(x) := f(x) + g(x).
\end{equation}

The proposed algorithm and convergence analysis rely on Assumptions~\ref{as01}. 
Note that $g(x) = \lambda \|x\|$ is a Lipschitz continuous function. 

\begin{assumption} \label{as01}
The gradient of $f$ is Lipschitz continuous on $\mathcal{M}$ with a Lipschitz constant $L_f$ and the function $g$ is Lipschitz continuous with Lipschitz constant $L_g$, where the Lipschitz continuity is defined in the sense of the Euclidean setting.
\end{assumption}


\subsection{Algorithm description} \label{sec:AlgDes}


The accelerated manifold proximal gradient method proposed for Problem~\eqref{prob4} is stated in Algorithm~\ref{alg:AManPG}. It is based on the AManPG algorithm in~\cite{HuaWei2019}. 

\begin{algorithm}[ht!]
\caption{Accelerated Manifold Proximal Gradient Method (AManPG)}
\label{alg:AManPG}
\begin{algorithmic}[1]
\Require Lipschitz constant $L_f$ on $\nabla f$, parameter $\mu >0$ in the proximal mapping, line search parameter $\sigma \in (0, 1 / (8 \mu)]$, shrinking parameter in line search $\nu \in (0, 1)$, positive integer $N$ for safeguard, initial iterates $\Lambda_y$ and $\Lambda_z$ for the semi-smooth Newton algorithm;  
\State $t_0 = 1$, $y_{0} = x_0$, $z_0 = x_0$;
\For {$k = 0, \ldots$}
\If {$\mod(k, N) = 0$} \Comment{Invoke safeguard every $N$ iterations}
\State \label{alg:AManPG:st0} Invoke Algorithm~\ref{alg:Safeguard}: {\small $[z_{k + N}, x_k, y_k, t_k, \Lambda_z] = Alg\ref{alg:Safeguard}(z_{k}, x_k, y_k, t_k, F(x_k), \Lambda_z)$;}
\EndIf
\State \label{alg:AManPG:st1} Approximately solve
$
\eta_{y_k} \approx \argmin_{\eta \in \T_{y_k} \mathcal{M}} \inner[]{\grad f(y_k)}{\eta} + \frac{1}{2 \mu} \|\eta\|^2 + g(y_k + \eta)
$
such that~\eqref{e42} holds;
\State \label{alg:AManPG:st2} $x_{k+1} = R_{y_k}(\eta_{y_k})$;
\State \label{alg:AManPG:st3} $t_{k + 1} = \frac{\sqrt{4 t_k^2 + 1} + 1}{2}$;
\State \label{alg:AManPG:st4} Compute
$
y_{k + 1} = R_{x_{k + 1}}\left( \frac{1 - t_k}{t_{k+1}} P_{T_{x_{k+1}} \mathcal{M}} (x_k - x_{k+1}) 
\right);
$
\EndFor
\end{algorithmic}
\end{algorithm}

\begin{algorithm}[ht!]
\caption{Safeguard for Algorithm 1}
\label{alg:Safeguard}
\begin{algorithmic}[1]
\Require $(z_{k}, x_k, y_k, t_k, F(x_k), \Lambda_z)$; The maximum number of iterations for line search $N_{\max} > 0$;
\Ensure $[z_{k + N}, x_k, y_k, t_k, \Lambda_z]$;
\State \label{alg:Safeguard:st1} Approximately solve
$
\eta_{z_k} \approx \argmin_{\eta \in \T_{z_k} \mathcal{M}} \inner[]{\grad f(z_k)}{\eta} + \frac{1}{2 \mu} \|\eta\|_{W_{z_k}}^2 + g(z_k + \eta)
$
such that~\eqref{e42} holds;
\State Set $\alpha = 1$, and $i_{\mathrm{iter}} = 0$;
\While {$F(R_{z_k}(\alpha\eta_{z_k})) > F(z_k) - \sigma \alpha {\|\eta_{z_k}\|_{\mathrm{F}}^2}$ and $i_{\mathrm{iter}} < N_{\max}$} \label{alg:Safeguard:st2}
\State \label{alg:Safeguard:st5} $\alpha = \nu \alpha$; $i_{\mathrm{iter}} = i_{\mathrm{iter}} + 1$;
\EndWhile
\If {$i_{\mathrm{iter}} = N_{\max}$}
\State \label{alg:Safeguard:st6} Line search fails;
\EndIf
\If {$F(R_{z_k}(\alpha\eta_{z_k})) < F(x_k)$} \Comment{Safeguard takes effect} \label{alg:Safeguard:st3}
\State $x_k = R_{z_k}(\alpha\eta_{z_k})$, $y_k = R_{z_k}(\alpha\eta_{z_k})$, and $t_k=1$;
\Else
\State $x_k$, $y_k$ and $t_k$ keep unchanged;
\EndIf \label{alg:Safeguard:st4}
\State $z_{k+N} = x_k$; \Comment{Update the compared iterate;}
\end{algorithmic}
\end{algorithm}


The main difference between Algorithm~\ref{alg:AManPG} and~\cite[Algorithm~1]{HuaWei2019} is the accuracy for solving the subproblem 
\begin{equation} \label{e10}
\argmin_{\eta \in \T_{y_k} \mathcal{M}} \inner[]{\grad f(y_k)}{\eta} + \frac{1}{2 \mu} \|\eta\|^2 + g(y_k + \eta).
\end{equation}
In~\cite{HuaWei2019}, the authors require the subproblem to be solved exactly for convergence analysis, whereas we only require solving it with sufficient accuracy. The quantitive accuracy that guarantees global convergence is also given.

A proximal mapping in the Euclidean setting often admits a computationally cheap closed-form solution. However, in the Riemannian setting, the proximal mapping does not usually have a closed form solution due to the existence of an extra linear constraint: $\eta \in \T_{y_k} \mathcal{M}$. The existing Riemannian proximal mappings in~\cite{CMSZ2019,HuaWei2019,HuaWei2019b,HuaWei2021} are solved by a semi-smooth Newton algorithm. In the theoretical analyses of~\cite{CMSZ2019,HuaWei2019,HuaWei2019b} for global convergence, the Riemannian proximal mappings are assumed to be solved exactly. 
In~\cite{HuaWei2021}, an inexact Riemannian proximal gradient (IRPG) is proposed. Though not solving the Riemannian proximal mapping exactly, IRPG assumes a sufficiently small $\mu$ in~\eqref{e10} which needs to be estimated in practice. 
Here, we also need the Riemannian proximal mapping to be solved approximately. It is shown later in Lemma~\ref{le02} that if the Riemannian proximal mapping is solved with sufficient accuracy, then the search direction is descent, independent of the choice of~$\mu$.

If $\eta_{x}^*$ is the exact solution of~\eqref{e10} (
we omit the subscript $k$, use $x$ instead of $y$, and use $\xi_x$ to denote $\grad f(x)$.), then it satisfies
\begin{equation} \label{e11}
\eta_x^* = \argmin_\eta \langle \xi_x,\eta \rangle+\frac{1}{2\mu}\langle \eta,\eta\rangle + g(x+\eta)\quad\mbox{subject to}\quad \eta \in \T_x \mathcal{M}.
\end{equation}
It follows that $\eta \in \T_x \mathcal{M}$ is equivalent to $B_x^T \eta = 0$. Therefore, the KKT condition for~\eqref{e11} is given by
\begin{align}
&\partial_{\eta} \mathcal{L}(\eta, \Lambda) = 0, \label{e12} \\
&B_x^T \eta = 0, \label{e13}
\end{align}
where $\mathcal{L}(\eta, \Lambda)$ is the Lagrangian function defined by
\begin{equation}
\mathcal{L}(\eta,\Lambda)=\langle \xi_x,\eta\rangle+\frac{1}{2\mu}\langle \eta,\eta\rangle+g(x+\eta)-\langle\Lambda,B_x^T \eta\rangle.
\end{equation}
Equation~\eqref{e12} yields
\begin{equation} \label{e14}
\eta= v(\Lambda) := \Prox_{\mu g}\left(x - \mu (\xi_x - B_x \Lambda ) \right) - x,
\end{equation}
where 
\begin{equation}\label{e16}
\Prox_{\mu g}(z) = \argmin_{v\in\mathbb{R}^{n\times p}} \frac{1}{2}\|v-z\|^2+\mu g(v)
\end{equation}
denotes the proximal mapping. Substituting~\eqref{e14} into~\eqref{e13} yields that
\begin{equation} \label{e15}
\Psi(\Lambda):=B_x^T \left( \Prox_{\mu g}\left(x - \mu (\xi_x - B_x \Lambda ) \right) - x \right)=0,
\end{equation}
which is a system of nonlinear equations with respect to $\Lambda$. Therefore, to solve~\eqref{e11}, one can first find any root of~\eqref{e15} and substitute it back to~\eqref{e14} to obtain~$\eta_x^*$.

The equation~\eqref{e15} can be solved efficiently by a semi-smooth Newton method. Analogous to the classical Newton method, the estimation of $\Lambda$ is updated by $\Lambda_{k + 1} = \Lambda_k + d_k$, where $d_k$ is computed by solving a Newton equation, i.e.,
\begin{equation} \label{e17}
J_\Psi(\Lambda_k)[d] = -\Psi(\Lambda_k),
\end{equation}
where $J_\Psi(\Lambda_k)$ is a generalized Jacobian of $\Psi$. Note that when 
 $g(x)=\lambda \|x\|_1$, it is well-known~\cite{Beck2017} that 
the solution to the proximal mapping \eqref{e16} can be computed by thresholding each entry of $z$. Moreover, by the chain rule, we have 
\begin{align*}{
J_\Psi(\Lambda_k)[d] = B_x^T \left(\partial\Prox_{\mu g} \left(x-\mu (\xi_x - B_x \Lambda_k)\right) \odot \left(\mu B_x d\right)\right),}
\end{align*}
where $\partial \Prox_{\mu g}(\cdot)$ denotes the generalized Clarke subdifferential of $\Prox_{\mu g}(\cdot)$ and $\odot$ denotes the entrywise product of two matrices. 
Once again,  when 
$g(x)=\lambda \|x\|_1$ the generalized Clarke subdifferential of $\Prox_{\mu g}(\cdot)$ can also be computed in an entrywise manner~\cite{CLA90,XLWZ2016,LiSunToh18}. 
Here, we do not require solving ~\eqref{e15} exactly but only find a $\Lambda$ such that $\|\Psi(\Lambda)\|$ is sufficiently small in the sense that
\begin{equation} \label{e42}
\|\Psi(\Lambda)\| \leq \sqrt{4 \mu^2 L_g^2 + \|\hat{v}(\Lambda)\|^2 / 2} - 2 \mu L_g,
\end{equation}
where $\hat{v}(\Lambda) = P_{\T_{x_k} \mathcal{M}} ( \Prox_{\mu g} \left(x - \mu (\xi_x - B_x \Lambda ) \right) - x )$ is the resulting search direction~$\eta_x$. 

\subsection{Global convergence analysis}

Lemma~\ref{le04} states the key result used to prove the global convergence. In~\cite[Lemma~5.1]{CMSZ2019}, the inequality~\eqref{e36} is proven up to a coefficient under the assumption that the subproblem~\eqref{e10} is solved exactly. Here, it is shown that if the subproblem is solved accurately enough such that~\eqref{e29} holds, then we also have the inequality~\eqref{e36}. 
The result~\eqref{e36} has not been given for the existing inexact Riemannian proximal gradient method~\cite{HuaWei2021} and is crucial to Algorithm~\ref{alg:AManPG}.

\begin{lemma} \label{le04}
Let $\ell_x(\eta) = \langle \xi_x, \eta \rangle+\frac{1}{2\mu}\langle \eta, \eta\rangle + g(x+\eta)$. Let $\epsilon$ denote $\Psi(\Lambda) = B_x^T v (\Lambda)$.
We then have
\[
g(x) \geq \inner[]{\xi_x}{ \hat{v}(\Lambda)} + \frac{1}{2 \mu} \| \hat{v}(\Lambda) \|^2 + g(x + \hat{v}(\Lambda)) - (2 L_g + \frac{1}{2 \mu} \|\epsilon\|) \|\epsilon\|,
\]
where $\hat{v}$ is defined in~\eqref{e42}. Furthermore, if $\epsilon$ is sufficiently close to 0 in the sense that 
\begin{equation} \label{e29}
\|\epsilon\| \leq \sqrt{ 4 \mu^2 L_g^2 + \|\hat{v}(\Lambda)\|^2/2 } - 2 \mu L_g, 
\end{equation}
then it holds that
\begin{equation} \label{e36}
\ell_x( \alpha \hat{v}(\Lambda) ) - \ell_x(0) \leq - \left( \frac{\alpha (1 - 2 \alpha) }{4 \mu} \right) \|\hat{v}(\Lambda)\|^2, \quad \forall \alpha \in [0, 1].
\end{equation}
\end{lemma}
\begin{proof}
Consider the optimization problem
\begin{align} \label{prob5}
\min_{B_x^T \eta = \epsilon} \ell_x(\eta).
\end{align}
Its KKT condition is given by
\begin{equation*}
\partial_{\eta} \mathcal{L}(\eta, \Lambda) = 0, \qquad B_x^T \eta = \epsilon,
\end{equation*}
which is satisfied by $v(\Lambda)$ defined in~\eqref{e14}. Therefore, $v(\Lambda)$ is the minimizer of $\ell_x(\eta)$ over the set $\mathcal{S} = \{v : B_x^T v = \epsilon \}$, i.e.,
\begin{equation} \label{e33}
v(\Lambda) = \argmin_{v \in \mathcal{S}} \ell_x(\eta) = \langle \xi_x, \eta \rangle+\frac{1}{2\mu}\langle \eta,\eta\rangle + g(x+\eta).
\end{equation}

Define the vector $v_0 = B_x \epsilon$. It can be easily verified that $B_x^T v_0 = B_x^T B_x \epsilon = \epsilon$. Therefore, it holds that $v_0 \in \mathcal{S}$.
By $\frac{1}{\mu}$-strong convexity of ${\ell}_x$, we have
\begin{equation} \label{e20}
\ell_x(v_0) \geq \ell_x(v(\Lambda)) + \inner[]{\partial \ell_x( v(\Lambda) )}{ v_0 - v(\Lambda) } + \frac{1}{2 \mu} \|v_0 - v(\Lambda)\|^2.
\end{equation}
From the optimality condition of Problem~\eqref{prob5}, we have that $0 \in P_{\T_{\eta} \mathcal{S}} \partial \ell_x( v(\Lambda) )$. Since $\T_{\eta} \mathcal{S} = \{u : B_x^T u = 0\}$ and $B_x^T (v_0 - v(\Lambda)) =\epsilon - \epsilon = 0$, it holds that $v_0 - v(\Lambda) \in \T_{\eta} \mathcal{S}$. Therefore, we have
\begin{equation} \label{e19}	
0 \in \inner[]{\partial \ell_x( v(\Lambda) )}{ v_0 - v(\Lambda) }.
\end{equation}
It follows from~\eqref{e20} and~\eqref{e19} that
\begin{equation} \label{e18}
\ell_x(B_x \epsilon) \geq \ell_x( v(\Lambda) ) + \frac{ 1 }{2 \mu} \| v(\Lambda) - B_x \epsilon \|^2.
\end{equation}
Substituting the definition of $\ell_x$ into inequality~\eqref{e18} and noting $\epsilon = B_x^T v(\Lambda)$, we have that
\begin{align}
\frac{1}{2 \mu} \|B_x B_x^T v(\Lambda)\|^2 + g(x + B_x B_x^T v(\Lambda) )& \geq \nonumber \\
\inner[]{\xi_x}{ v(\Lambda) } + \frac{1}{2 \mu} \| v(\Lambda) \|^2 +& g(x + v(\Lambda)) + \frac{1}{2 \mu} \|v(\Lambda) - B_x B_x^T v(\Lambda) \|.  \label{e21}
\end{align}
It follows that
\begin{align}
&g(x) \geq \inner[]{\xi_x}{P_{\T_x \mathcal{M}} v(\Lambda)} + \frac{1}{2 \mu} \| P_{T_x \mathcal{M}} v(\Lambda) \|^2+ g(x + v(\Lambda)) + g(x) \nonumber \\
&- g(x + B_x \epsilon) - \frac{1}{2 \mu} \|B_x B_x^T v(\Lambda)\|^2 \nonumber \\
\geq& \inner[]{\xi_x}{ \hat{v}(\Lambda)} + \frac{1}{2 \mu} \| \hat{v}(\Lambda) \|^2 + g(x + \hat{v}(\Lambda)) + g(x + v(\Lambda)) \nonumber\\
& - g(x + (I - B_x B_x^T) v(\Lambda)) + g(x) - g(x + B_x \epsilon) - \frac{1}{2 \mu} \|B_x \epsilon\|^2 \nonumber\\
\geq& \inner[]{\xi_x}{ \hat{v}(\Lambda)} + \frac{1}{2 \mu} \| \hat{v}(\Lambda) \|^2 + g(x + \hat{v}(\Lambda)) - |g(x + v(\Lambda)) - g(x + v(\Lambda) - B_x \epsilon )| \nonumber\\
& - |g(x) - g(x + B_x \epsilon)| - \frac{1}{2 \mu} \|B_x \epsilon\|^2 \nonumber\\
\geq& \inner[]{\xi_x}{ \hat{v}(\Lambda)} + \frac{1}{2 \mu} \| \hat{v}(\Lambda) \|^2 + g(x + \hat{v}(\Lambda)) - (2 L_g + \frac{1}{2 \mu} \|B_x \epsilon\|) \|B_x \epsilon\|, \label{e22}
\end{align}
where the first inequality follows from~\eqref{e21} and $\| v(\Lambda) \|^2 \geq 0$, the second inequality follows from $\hat{v}(\Lambda) = P_{\T_x \mathcal{M}} v(\Lambda) = (I - B_x B_x^T) v(\Lambda)$, and the fourth inequality follows from the Lipschitz continuity of $g$ with Lipschitz constant $L_g$. This completes the proof for the first result.

Since $g$ is convex, we have
\begin{equation} \label{e23}
g(x + \alpha \hat{v}(\Lambda) ) - g(x) = g( \alpha (x + \hat{v}(\Lambda)) + (1-\alpha) x ) - g(x) \leq \alpha ( g(x + \hat{v}(\Lambda)) - g(x) ).
\end{equation}
Combining~\eqref{e22} and~\eqref{e23} yields
\begin{align}
\ell_x( \alpha \hat{v}(\Lambda) ) -& \ell_x(0) = \inner[]{\xi_x}{ \alpha \hat{v}(\Lambda) } + \frac{1}{2 \mu} \| \alpha \hat{v}(\Lambda) \|^2 + g(x + \alpha \hat{v}(\Lambda)) - g(x) \nonumber \\
\leq& \alpha \left( \inner[]{\xi_x}{ \hat{v}(\Lambda) } + \frac{\alpha}{2 \mu} \| \hat{v}(\Lambda) \|^2 + g(x + \hat{v}(\Lambda)) - g(x)\right) \nonumber \\
\leq& \alpha \left( \frac{\alpha}{2 \mu} \| \hat{v}(\Lambda) \|^2 - \frac{1}{2 \mu} \| \hat{v}(\Lambda) \|^2 + (2 L_g + \frac{1}{2 \mu} \|B_x \epsilon\|) \|B_x \epsilon\| \right). \label{e24}
\end{align}
By $\|\epsilon\| = \|B_x \epsilon\| \leq \sqrt{ 4 \mu^2 L_g^2 + \|\hat{v}(\Lambda)\|^2/2 } - 2 \mu L_g$, we have 
\begin{equation}
(2 L_g + \frac{1}{2 \mu} \|B_x \epsilon\|) \|B_x \epsilon\| \leq \frac{1}{4 \mu} \|\hat{v}(\Lambda)\|^2. \label{e25}
\end{equation}
The second result follows from~\eqref{e24} and~\eqref{e25}. Finally,~\eqref{e29} follows from the definition of $\Psi(\Lambda)$.
\end{proof}

Lemma~\ref{le02} implies that the while loop in Step~\ref{alg:Safeguard:st2} of Algorithm~\ref{alg:Safeguard} terminates in a finite number of iterations. Given~\eqref{e36} in Lemma~\ref{le04}, the proof of Lemma~\ref{le02} follows the same steps as that of~\cite[Lemma~5.2]{CMSZ2019} is therefore omitted. 
\begin{lemma} \label{le02}
Suppose Assumption~\ref{as01} holds. Then for any $\mu > 0$, there exists a constant $\bar{\alpha} \in (0, 1]$ such that for any $0 < \alpha < \bar{\alpha}$, Step~\ref{alg:Safeguard:st2} of Algorithm~\ref{alg:Safeguard} is satisfied, and the sequence $\{z_k\}$ generated by Algorithm~\ref{alg:AManPG} satisfies
\[
F(R_{z_k}( \alpha \eta_{z_k} )) - F(z_k) \leq - \frac{\alpha}{8 \mu} \|\eta_{z_k}\|^2.
\]
Moreover, the step size $\alpha > \rho \bar{\alpha}$ for all $k$.
\end{lemma}

Though the subproblem is solved inexactly, a zero search direction given by $\hat{v}(\Lambda)$ with condition~\eqref{e42} implies that the current iterate $x$ is a stationary point, which coincides with~\cite[Lemma~5.3]{CMSZ2019}.
\begin{lemma} \label{le03}
If $\eta_x = \hat{v}(\Lambda) = 0$, 
then $x$ is a stationary point of Problem~\eqref{prob4}.
\end{lemma}
\begin{proof}
If $\eta_x = \hat{v}(\Lambda) = 0$, then by~\eqref{e42}, 
we have that $\Psi(\Lambda) = 0$, which implies that the subproblem~\eqref{e11} is solved exactly and $\eta_x^* = 0$. By~\cite[Lemma~5.3]{CMSZ2019}, $x$ is a stationary point of Problem~\ref{prob4}.
\end{proof}

The main convergence result is given in Theorem~\ref{thm:convergence}. The proof follows the spirit of~\cite[Theorem~1]{HuaWei2019}. Here, we only highlight their differences.
\begin{theorem} \label{thm:convergence}
Suppose Assumption~\ref{as01} holds, then any accumulation point of the sequence $\{z_{0}, z_{N}, z_{2N},$ $\ldots, z_{iN}, \ldots\}$ generated by Algorithm~\ref{alg:AManPG} is a stationary point, i.e., if $z_*$ is an accumulation point of the above sequence,  then $0 \in P_{\T_{z_*} \mathcal{M} } \partial F(z_*)$.
\end{theorem}

\begin{proof}
Since the subscript of $z_k$ in Algorithm~\ref{alg:AManPG} is a multiple of $N$, we use $\{\tilde{z}_i\}$ to denote $\{z_k\}$, where $\tilde{z}_i = z_{i N}$. Let $(\eta_{\tilde{z}_i},\Lambda_{\tilde{z}_i})$ denote the output of the Semi-smooth Newton algorithm when the input is $(\tilde{z}_i, \Lambda_{\tilde{z}_{i-1}}, \grad f(\tilde{z}_i) )$, i.e., the input and output of Step~\ref{alg:Safeguard:st1} of Algorithm~\ref{alg:Safeguard}. 

By the safeguard in Algorithm~\ref{alg:Safeguard} and Lemma~\ref{le02}, we have
$
F(\tilde{z}_{i+1}) - F(\tilde{z}_i)	\leq - \frac{\rho \bar{\alpha}}{8 \mu} \| \eta_{\tilde{z}_k} \|^2.
$
Since $F$ is continuous and $\mathcal{F}$ is compact, the function $F$ is bounded from below. It follows that
\begin{equation}
\infty > \sum_{i = 0}^\infty	 F(\tilde{z}_{i}) - F(\tilde{z}_{i+1}) \geq \frac{\rho \bar{\alpha}}{8 \mu} \| \eta_{\tilde{z}_k} \|^2,
\end{equation}
which implies
\begin{equation} \label{e30}
\lim_{i \rightarrow \infty} \| \eta_{\tilde{z}_k} \| = 0.
\end{equation}
By~\eqref{e30}, ~\eqref{e42}, and  $\eta_{\tilde{z}_k} = \hat{v}(\tilde{z}_i)$, we have
\begin{equation} \label{e31}
\lim_{i \rightarrow \infty} \| \Psi( \Lambda_{\tilde{z}_i} ) \| = 0.
\end{equation}
Since $\hat{v}(\Lambda_{\tilde{z}_i}) = P_{\T_{\tilde{z}_i} \mathcal{M} } v(\Lambda_{\tilde{z}_i}) = v(\Lambda_{\tilde{z}_i}) - B_{\tilde{z}_i} \Psi( \Lambda_{\tilde{z}_i} ) $, we have
\begin{align} \label{e32}
\|v( \Lambda_{\tilde{z}_i} )\| = \| \hat{v}(\Lambda_{\tilde{z}_i}) + B_{\tilde{z}_i} \Psi( \Lambda_{\tilde{z}_i} )\| \leq \|\hat{v}(\Lambda_{\tilde{z}_i})\| + \|B_{\tilde{z}_i} \Psi( \Lambda_{\tilde{z}_i} )\| = \| \eta_{\tilde{z}_k} \| + \|\Psi( \Lambda_{\tilde{z}_i} )\|.
\end{align}
Combining~\eqref{e30},~\eqref{e31} and~\eqref{e32} yields 
\begin{equation} \label{e34}
\lim_{i \rightarrow \infty} \|v( \Lambda_{\tilde{z}_i} )\| = 0.
\end{equation}
By~\eqref{e33}, we have
\begin{equation} \label{e35}
v( \Lambda_{\tilde{z}_i} ) = \argmin_{\eta \in \mathcal{S}_{\tilde{z}_i} } \ell_x(\eta) = \langle \grad f( \tilde{z}_{i} ), \eta \rangle+\frac{1}{2\mu}\langle \eta, \eta\rangle + g(\tilde{z}_i+\eta),
\end{equation}
where $\mathcal{S}_{\tilde{z}_i} = \{v : B_{\tilde{z}_i}^T v = B_{\tilde{z}_i}^T v( \Lambda_{\tilde{z}_i} ) \}$.
Using~\eqref{e34} and~\eqref{e35} and following the steps in the proof of~\cite[Theorem~1]{HuaWei2019}, we have that any accumulation point of $\{\tilde{z}_i\}$ is a stationary point.
\end{proof}


\section{Numerical Experiments} \label{sect:NumExp}

In this section, 
the performance of the proposed method, AManPG, is compared to other state-of-the-art methods using problems from community detection in networks and normalized cut for image segmentation. 


%


\subsection{Parameter setting and testing environment} \label{sect:para}


The subproblem in Algorithm~\ref{alg:AManPG} is solved by the regualized semi-smooth Newton algorithm in~\cite{XLWZ2018}. 
Let I-AManPG denote Algorithm~\ref{alg:AManPG} and let E-AManPG denote Algorithm~\ref{alg:AManPG} with the condition that $\|\Psi(\Lambda)\| \leq 10^{-10}$. 
Therefore, we can view I-AManPG as an inexact Riemannian proximal gradient method while E-AManPG is essentially an exact one since E-AManPG solves the Riemannian proximal mapping to high accuracy.

The parameters $L_f$ and $\lambda$ are problem-dependent and specified later. The other parameters are set to be $\mu = 1 / L_f$, $\sigma = 10^{-4}$, $\nu = 0.5$, $N = 5$, $\Lambda_y = \Lambda_z = 0$, $N_{\max} = 5$, $\nu_{\mathrm{SSN}} = 0.9999$, $\beta = 0.1$, $\kappa_1 = 0.2$, $\kappa_2 = 0.75$, $\gamma_1 = 2$, $\gamma_2 = 5$, $\underline{\lambda} = 10^{-5}$, and the initial value $\epsilon_{\Psi} = 1$.

Unless otherwise indicated, I-AManPG and E-AManPG stop if the value of $\|\eta_{z_k}\|$ reduces at least by a factor of $10^{3}$.  
Inspired by Lemma~\ref{le07},
the last iterate is projected to the set $\mathcal{B}_{v}$ by the 
mapping $P_{\mathcal{B}_{v}} (X)$.


I-AManPG and E-AManPG are implemented in Matlab R2019b. All the experiments are performed on an Apple Mac platform with 1.4 GHz Quad-Core Intel Core i5. 

\subsection{Community detection}

In this section, we evaluate the performance of community detection by optimizing the formulation
\begin{equation} \label{e40}
\min_{X \in \mathcal{F}_{\mathbf{1}_n}} - \mathrm{trace}(X^T M X) + \lambda \|X\|_1,
\end{equation}
with I-AManPG algorithm. 

\subsubsection{Data sets}

Results are presented for solving community detection problems on synthetic LFR benchmark networks~\cite{lancichinetti2008benchmark}. Specifically, LFR benchmark networks assume that the distributions of degree and community size are power laws with exponents $\tau_1$ and $\tau_2$ respectively. Each node shares a fraction $1-\mu_{\mathrm{LFR}}$ of its edges with the other nodes of its community and a fraction $\mu_{\mathrm{LFR}}$ with nodes of the other communities, where $0 \leq \mu_{\mathrm{LFR}} \leq 1$ is the mixing parameter.  A software package to generate the benchmark networks is available at \url{https://www.santofortunato.net/resources}.

\subsubsection{Comparison of I-AManPG and E-AManPG on LFR benchmark networks}

The first set of experiments compares the efficiency and effectiveness of I-AManPG and E-AManPG to demonstrate the utility of the inexact form of the algorithm.
Throughout this section, the parameters $\tau_1$, $\tau_2$, and $\mu_{\mathrm{LFR}}$ are set to $-2$, $-1$, and $0.1$ respectively.
Four sets of other parameters are used to define a range of networks. They are as follows:
\begin{itemize}
	\item $N = 500$, $d_{ave} = 10$, $d_{max} = 20$, $N_c = 50$, $n_c = 10$;
	\item $N = 1000$, $d_{ave} = 20$, $d_{max} = 40$, $N_c = 100$, $n_c = 10$;
	\item $N = 5000$, $d_{ave} = 40$, $d_{max} = 80$, $N_c = 500$, $n_c = 10$;
	\item $N = 10000$, $d_{ave} = 40$, $d_{max} = 80$, $N_c = 1000$, $n_c = 10$;
\end{itemize}
where $N$ denotes the number of nodes, $d_{ave}$ denotes the average node degree, $d_{max}$ denotes the maximum node degree, $N_c$ denotes the number of nodes that all communities have, and $n_c$ denotes the number of communities. The balancing parameter $\lambda$ in~\eqref{e40} is set to $0.3$. This value roughly balances the two terms of the cost function for the LFR benchmark networks defined by the parameters sets above.

\begin{table}[ht!]
\renewcommand\tabcolsep{2.0pt}
\begin{center}
\caption{Comparison of the efficiency of I-AManPG and E-AManPG.}
\label{table1}
\vspace{0.25cm}
\small{ 
\begin{tabular}{c|cccccccc}
  \hline
  \hline
 & I-A & E-A & I-A & E-A & I-A & E-A & I-A & E-A \\
  \hline
$(N, n_c)$ & \multicolumn{2}{c}{$(500, 10)$} & \multicolumn{2}{c}{$(1000, 10)$} & \multicolumn{2}{c}{$(5000, 10)$} & \multicolumn{2}{c}{$(10000, 10)$} \\
  \hline
iter   & 64 & 52 & 50 & 59 & 63 & 58 & 55 & 55 \\
SSNiter  & 28 & 212 & 13 & 248 & 34 & 311 & 52 & 330 \\
nf & 143 & 115 & 112 & 131 & 140 & 128 & 123 & 122\\
ng  & 83 & 65 & 64 & 73 & 81 & 72 & 71 & 68 \\
nR  & 142 & 114 & 111 & 130 & 139 & 127 & 122 & 121\\
nSG & 4 & 14 & 2 & 15 & 4 & 13 & 3 & 10 \\
$F$ &-67.00 & -67.0 & -149 & -149 & -284 & -284 & -251 & -251 \\
$\frac{\|\eta_{z_k}\|}{\|\eta_{z_0}\|}$ & $7.0_{-4}$  & $5.7_{-4}$  & $5.5_{-4}$  & $5.1_{-4}$  & $6.3_{-4}$  & $5.8_{-4}$  & $5.2_{-4}$  & $6.9_{-4}$ \\
time &  0.15 & 0.31 & 0.17 & 0.75 & 0.84 & 3.03 & 1.54 & 5.19\\
\hline
\hline
\end{tabular}
}
\end{center}
\end{table}

Table~\ref{table1} contains the data to compare I-AManPG and E-AManPG for LFR benchmark networks
For each of the four LFR parameter sets, 10 random networks were generated. The computational time for each network was computed as described earlier and the quantities reported in the table are the averages over the 10 networks from each parameter set.
I-A denotes I-AManPG and E-A denotes E-AManPG. The LFR parameter set used for each pair of columns is identified by values $(N, n_c)$.
The labels iter, SSNiter, nf, ng, nR, nSG, F, $\frac{\|\eta_{z_k}\|}{\|\eta_{z_0}\|}$, and time, respectively, denote the number of iterations in AManPG, the number of iterations in semi-smooth Newton method, the number of function evaluations, the number of gradient evaluations, the number of retraction evaluations,  the number of safeguards (Step~\ref{alg:Safeguard:st3}) that are taken, the function value at the final iterate, the reduction of the norm of search directions $\frac{\|\eta_{z_k}\|}{\|\eta_{z_0}\|}$, and the computational time in seconds. The subscript $k$ denotes a scale of $10^{k}$.

As shown in Table~\ref{table1}, I-AManPG and E-AManPG find the same solutions in the sense that their function values are the same up to three significant digits.  In fact, though not reported in the table, in our experimental setting, both I-AManPG and E-AManPG always converge to the same solution which always represents the ground truth partition. So both are equally effective. In addition, I-AManPG requires less work when solving the Riemannian proximal mapping in the sense that the number of semi-smooth Newton iterations in each outer iteration is small compared to E-AManPG.  
Moreover, less accuracy for solving the Riemannian proximal mapping does not influence the number of outer iterations significantly. Therefore, I-AManPG is more efficient than E-AManPG in terms of computational time. For the rest of the experiments, we use I-AManPG as the representative method.


\subsubsection{Comparison of the effectiveness of I-AManPG and state-of-the-art community detection methods}

In this section, community detection by the optimization model~\eqref{e40} using I-AManPG is compared to three state-of-the-art methods: Danon et al.'s algorithm~\cite{danon2006effect}, the Louvain method~\cite{blondel2008fast} and Newman's spectral optimization method~\cite{newman2006modularity}. 

These algorithms all aim to maximize the modularity $Q=\frac{1}{2m}\tr(X^TMX)$ over the set of indicator matrices, where $m$ is the number of edges, and an indicator matrix $X$ is defined by $X \geq 0, X^TX$ is diagonal, and each row of $X$ has single entry with value one. Each indicator matrix specifies a partitioning of the nodes into communities. 
Danon et al.'s algorithm is a agglomerative method that is a variant of Newman's fast greedy method~\cite{newman2004fast} which starts with each node as a singleton community. Intra-community edges are added one-by-one by choosing the edge such that the modified partition gives the maximum increase of modularity with respect to the previous configuration. 
The Louvain method is an agglomerative method where each iteration comprises two phases. The first phase creates intermediate-communities by merging pairs of nodes such that the modularity increases. The first phase terminates when no such pair merging increases modularity. In the second phase, a smaller graph, called the reduced graph, is created where each node in this graph represents an intermediate-community. Newman's spectral optimization method is a divisive method that computes the eigenvectors of the modularity matrix $M$ corresponding to the largest positive eigenvalues. The nodes are grouped into two parts based on the signs of the component of the eigenvectors. The process is then repeated for each of the parts 
until splitting a given community of nodes makes a zero or negative contribution to the total modularity.  

To make fair comparisons, we use publicly-available Matlab implementations of these algorithms.
The codes for Danon et al.'s algorithm, Louvain's algorithm, and Newman's spectral method are respectively from~\cite{AKehagiascdtoolbox}, ~\cite{matlab_Louvain_AS} and~\cite{matlab_network_tools_GB}. Note that the codes for Newman's spectral method in \cite{matlab_network_tools_GB} do not embed the fine-tuning stage and use a different stopping criteria as in paper \cite{newman2006modularity}. We modified the stopping criteria in the codes such that it has the same stopping criteria as in paper \cite{newman2006modularity}. The codes in \cite{matlab_network_tools_GB} use dense matrix computations, and we modified them with significantly more efficient sparse computations. In this way, the computational efficiency was improved to the point where it produced times that were reasonable to include in these comparisons.

To compare the effectiveness of the four methods, we consider three quality measurements: normalized mutual information (NMI)~\cite{danon2005comparing}, adjusted mutual information (AMI) \cite{vinh2010information}, and  purity~{\cite{manning2008utze}.} NMI is a similarity measure between two partitions that represents their normalized mutual entropy. AMI further corrects the measure for randomness by adopting a hypergeometric model of randomness. We refer interested readers to~\cite[(2)]{danon2005comparing} and~\cite[Section~4.1]{vinh2010information} for the definitions. Both NMI and AMI take on values between 0 and 1. Values closer 1 indicate greater consistency between the partitions. 
Given two partitions $X$ and $Y$ of $N$ nodes, the purity is given by
$
\mbox{purity}( X, Y ) = \frac{1}{N} \sum_k \max_j \vert X_k \cap Y_j\vert,
$
where $X_k$ denotes the set of nodes in $k$-th community of partition $X$, and likewise for $Y_j$, and $\vert X_k \cap Y_j\vert$ denotes the number of nodes in $X_k \cap Y_j$. The value of purity is also between 0 and 1.  
The closer it is to one, the better the two partitions are nested. 
In our numerical experiments, the ground truth is known and therefore the computed partition is compared to the ground truth. When the two partitionings do not have same number of communities good nesting of the partition with the larger number of communities in the partition with the smaller number is an indication that further division or agglomeration could yield a closer approximation of ground truth. Such a pair of partitions is therefore preferred to a pair with lower purity. Since purity is not symmetric, we take $X$ to be the partition with the larger number of communities.


For LFR benchmark networks used in the comparisons in this section,  the parameters $\tau_1$, $\tau_2$, $N$, $d_{ave}$, $d_{max}$, $N_c$, and $n_c$ are respectively set to $-2$, $-1$, $1000$, $20$, $40$, $50$, and $20$. The value of $\lambda$ in~\eqref{e40} is 0.3, as in the previous set of experiments. The empirical results with multiple values of $\mu_{\mathrm{LFR}}$ are reported in Table~\ref{table2}. As before, each result is an average over 10 randomly selected LFR benchmark networks.

\begin{table}[htbp]
\renewcommand\tabcolsep{2.0pt}
\begin{center}
\caption{Compare the effectiveness of I-AManPG to other state-of-the-art methods. } \label{table2}
\resizebox{\textwidth}{!}
{
\begin{tabular}{l |c |c| c |c |c |c| c| c|c|c}
\hline
\hline
\multirow{2}{*}{ } & \multirow{2}{*}{} & \multicolumn{9}{c}{$\mu_{\mathrm{LFR}}$}\\
\cline{3-11}
&  & 0 & 0.1 & 0.2 & 0.3 & 0.4 & 0.5 & 0.6 & 0.7 &  0.8\\
 \hline
 \multirow{5}{*}{Danon } & NMI &  0.9998 & 0.9891 & 0.9394 & 0.8504 & 0.7331 & 0.5808 & 0.3781 & 0.1412 & 0.0548\\
 & AMI &  0.9998 & 0.9870 & 0.9166 & 0.7922 & 0.6399 & 0.4736 & 0.2878 & 0.0935 & 0.0215\\
 & Mod. &  0.9496 & 0.8436 & 0.7225 & 0.5920 & 0.4687 & 0.3452 & 0.2458 & 0.1892 & 0.1814\\
 & purity &  0.9999 & 0.9938 & 0.9739 & 0.9414 & 0.9058 & 0.8344 & 0.6886 & 0.4269 & 0.3095\\
 & time &  2.8304 & 2.8597 & 2.8199 & 2.7648 & 2.7287 & 2.8318 & 2.7739 & 2.7625 & 2.7318\\
 & $q_c$ & 20 & 20 & 18 & 15 & 11 & 9 & 7 & 7 & 8\\
 \hline
  \multirow{5}{*}{Danon\_force\_q} & NMI &  0.9998 & 0.9889 & 0.9397 & 0.8494 & 0.7309 & 0.5794 & 0.3939 & 0.1778 & 0.0954\\
& AMI &  0.9998 & 0.9870 & 0.9201 & 0.7962 & 0.6415 & 0.4727 & 0.2950 & 0.1035 & 0.0282\\
 & Mod. &  0.9496 & 0.8433 & 0.7209 & 0.5868 & 0.4615 & 0.3394 & 0.2431 & 0.1872 & 0.1794\\
& purity & 0.9999 & 0.9935 & 0.9714 & 0.9349 & 0.8954 & 0.8198 & 0.6778 & 0.4201 & 0.3065\\
& time & 2.8353 & 2.9606 & 2.8222 & 2.8068 & 2.8055 & 2.8959 & 2.8183 & 2.8459 & 2.7432\\
& $q_c$ & 20 & 20 & 20 & 20 & 20 & 20 & 20 & 20 & 20\\
\hline
 \multirow{5}{*}{Louvain}  & NMI &1.0000 & 1.0000 & 1.0000 & 1.0000 & 1.0000 & 0.9987 & 0.9805 & 0.2862 & 0.0784\\
 & AMI & 1.0000 & 1.0000 & 1.0000 & 1.0000 & 1.0000 & 0.9974 & 0.9652 & 0.2249 & 0.0358\\
 & Mod. & 0.9497 & 0.8499 & 0.7503 & 0.6500 & 0.5499 & 0.4496 & 0.3477 & 0.2098 & 0.1967\\
 & purity & 1.0000 & 1.0000 & 1.0000 & 1.0000 & 1.0000 & 0.9999 & 0.9950 & 0.4734 & 0.2660\\
 & time & 0.5444 & 0.7291 & 1.3703 & 1.8963 & 2.6797 & 3.3418 & 4.5768 & 9.2669 & 8.8130\\
  & $q_c$ & 20 & 20 & 20 & 20 & 20 & 20 & 19 & 11 & 11\\
 \hline
 \multirow{5}{*}{Louvain\_force\_q}  & NMI &1.0000 & 1.0000 & 1.0000 & 1.0000 & 1.0000 & 0.9987 & 0.9805 & 0.2981 & 0.0847\\
& AMI & 1.0000 & 1.0000 & 1.0000 & 1.0000 & 1.0000 & 0.9974 & 0.9652 & 0.2382 & 0.0392\\
 & Mod. & 0.9497 & 0.8499 & 0.7503 & 0.6500 & 0.5499 & 0.4496 & 0.3477 & 0.2098 & 0.1967\\
& purity & 1.0000 & 1.0000 & 1.0000 & 1.0000 & 1.0000 & 0.9999 & 0.9950 & 0.4642 & 0.2418\\
 & time & 0.5440 & 0.7478 & 1.0333 & 1.2042 & 1.7002 & 2.0761 & 2.7676 & 5.4529 & 5.5061\\
 & $q_c$ & 20 & 20 & 20 & 20 & 20 & 20 & 19 & 12 & 12\\
\hline
\multirow{5}{*}{Newman\_Eig Sparse} & NMI & 0.9988 & 0.7225 & 0.7132 & 0.6760 & 0.5498 & 0.3912 & 0.2807 & 0.1340 & 0.0497\\
 & AMI &0.9985 & 0.6521 & 0.6396 & 0.6122 & 0.4704 & 0.3071 & 0.2098 & 0.0907 & 0.0235\\
& Mod. & 0.9482 & 0.5493 & 0.5057 & 0.4157 & 0.3051 & 0.2379 & 0.1917 & 0.1578 & 0.1461\\
& purity & 0.9994 & 0.7828 & 0.8134 & 0.7460 & 0.6723 & 0.6366 & 0.5654 & 0.4367 & 0.3594\\
& time & 0.6346 & 0.4391 & 0.4197 & 0.4522 & 0.4140 & 0.3333 & 0.3565 & 0.3130 & 0.3026\\
& $q_c$ & 20 & 24 & 21 & 19 & 15 & 9 & 7 & 6 & 6\\
\hline
\multirow{5}{*}{Newman\_Eig\_force\_q Sparse} & NMI & 0.9988 & 0.6831 & 0.6787 & 0.6674 & 0.5498 & 0.3912 & 0.2807 & 0.1340 & 0.0497\\
& AMI &0.9985 & 0.5996 & 0.5991 & 0.6026 & 0.4704 & 0.3071 & 0.2098 & 0.0907 & 0.0235\\
& Mod. & 0.9482 & 0.4747 & 0.4463 & 0.4004 & 0.3051 & 0.2379 & 0.1917 & 0.1578 & 0.1461\\
& purity & 0.9994 & 0.7998 & 0.8001 & 0.7358 & 0.6723 & 0.6366 & 0.5654 & 0.4367 & 0.3594\\
& time & 0.6458 & 0.4666 & 0.4373 & 0.4528 & 0.4231 & 0.3418 & 0.3656 & 0.3213 & 0.3114\\
& $q_c$ & 20 & 18 & 17 & 18 & 15 & 9 & 7 & 6 & 6\\
\hline
\multirow{5}{*}{I-AManPG}  & NMI & 1.0000 & 1.0000 & 1.0000 & 1.0000 & 1.0000 & 0.9998 & 0.9600 & 0.4517 & 0.1294\\
& AMI & 1.0000 & 1.0000 & 1.0000 & 1.0000 & 1.0000 & 0.9998 & 0.9539 & 0.4037 & 0.0563\\
& Mod. & 0.9497 & 0.8499 & 0.7503 & 0.6500 & 0.5499 & 0.4498 & 0.3416 & 0.1735 & 0.1113\\
& purity & 1.0000 & 1.0000 & 1.0000 & 1.0000 & 1.0000 & 0.9999 & 0.9679 & 0.5605 & 0.3044\\
& time &  0.6357 & 0.4693 & 0.5870 & 0.9494 & 0.6749 & 0.4720 & 1.0332 & 1.6307 & 1.6757\\
& $q$ & 20 & 20 & 20 & 20 & 20 & 20 & 20 & 20 & 20\\
 \hline
\end{tabular}
}
\end{center}
\end{table}

The input parameter for I-AManPG determining the number of communities to be produced is $q$ and $q_c$ is the number of communities computed by each method. The "force\_q" label denotes the versions with modified termination criteria so that $q_c$ is as close to $q_{true}$ as the methods allow. 

From the results in Table \ref{table2}, we observe that when $\mu_{\mathrm{LFR}}=0$,  I-AManPG yields $\mathrm{NMI} = \mathrm{AMI} = \mathrm{purity} = 1$, the same modularity value and the same assignment to $q_{true} = 20$ strongly connected communities.  The Louvain method also has the same results with ground-truth communities while Danon et al.'s algorithm and Newman's spectral algorithm can get the results which are very close to the ground-truth communities, specifically they can detect exactly ground-truth communities for 9 of 10 random LFR graphs.  When $\mu_{\mathrm{LFR}}$ takes 0.1 to 0.4,  I-AManPG and Louvain algorithm can detect the exact ground-truth communities.  When $\mu_{\mathrm{LFR}}=0.5, 0.6$, I-AManPG can get results very close to ground-truth partitions and the results are competitive results with the Louvain algorithm, but with less time.  When $\mu_{\mathrm{LFR}}=0.7, 0.8$,  the results for all of these four algorithms are far away from the ground-truth partitions because the community structures in these cases are not strong.  Danon's algorithm and Newman spectral algorithm detect relatively inaccurate communities and relatively small qualifying external or internal measurements, i.e., NMI, AMI, purity and modularity for all noisy cases.  From the computational times, we can see that I-AManPG requires relatively less time than the others. 
It is worth noting that the number of edges $m$ does not change much and only the distribution of edges changes a lot as the mixing parameter increases. So, the computational time for Danon's algorithm and Newman's spectral method do not change much as the mixing parameter increases because the computational time of these two algorithms depends more on $m$ rather than on the distribution of edges.

As the mixing parameter increases, the difficulty level of detecting the correct number of communities increases as well.  I-AManPG requires the desired number of communities as an input parameter value, $q$,  and the choice of an initial $q$ and the development of a dynamic adaptation strategy are key ongoing research tasks for I-AManPG. Since the experiments in Table \ref{table2} use $q=q_{true}$ for I-AManPG and the other methods that are not "forced" are given no indication of $q_{true}$, experiments where I-AManPG uses $q \neq q_{true}$ probe the quality of the $q \neq q_{true}$ communities produced by I-AManPG compared to ground truth.
For each value of the input parameter $q =10,  17, \dotsc , 23$ and mixing parameter $\mu_{\mathrm{LFR}}=0, 0.1, \dotsc , 0.8$,  I-AManPG was applied to 10 randomly generated LFR benchmark networks. The results are shown in Table \ref{table3}.

\begin{table}[htbp]
\renewcommand\tabcolsep{2.0pt}
\begin{center}
\caption{Testing the effectiveness of I-AManPG for input parameter values $q\neq q_{true}$.}\label{table3}
\small{
\begin{tabular}{l |c |c| c |c |c |c| c| c|c|c}
\hline
\hline
\multirow{2}{*}{ } & \multirow{2}{*}{} & \multicolumn{9}{c}{$\mu_{\mathrm{LFR}}$}\\
\cline{3-11}
&  & 0 & 0.1 & 0.2 & 0.3 & 0.4 & 0.5 & 0.6 & 0.7 &  0.8\\
\multirow{5}{*}{I-AManPG}  & NMI & 0.7178 & 0.7178 & 0.7178 & 0.7176 & 0.7137 & 0.6966 & 0.6230 & 0.2957 & 0.0795\\
& AMI & 0.5452 & 0.5452 & 0.5452 & 0.5451 & 0.5428 & 0.5327 & 0.4856 & 0.2202 & 0.0324\\
& Mod. & 0.6924 & 0.6198 & 0.5474 & 0.4733 & 0.4005 & 0.3251 & 0.2531 & 0.1534 & 0.1007\\
 & purity &  1.0000 & 1.0000 & 1.0000 & 0.9999 & 0.9971 & 0.9842 & 0.9243 & 0.6363 & 0.5118\\
& time &  0.1939 & 0.2175 & 0.5020 & 0.3373 & 0.2593 & 0.4983 & 0.3297 & 0.7154 & 0.6616\\
 & $q$ & 10 & 10 & 10 & 10 & 10 & 10 & 10 & 10 & 10\\
 \hline
\multirow{5}{*}{I-AManPG}  & NMI & 0.9531 & 0.9523 & 0.9531 & 0.9523 & 0.9531 & 0.9493 & 0.9121 & 0.4132 & 0.1124\\
& AMI & 0.9049 & 0.9034 & 0.9049 & 0.9034 & 0.9049 & 0.8999 & 0.8675 & 0.3535 & 0.0466\\
& Mod. & 0.9245 & 0.8266 & 0.7306 & 0.6324 & 0.5363 & 0.4367 & 0.3362 & 0.1724 & 0.1100\\
& purity &  1.0000 & 1.0000 & 1.0000 & 1.0000 & 1.0000 & 0.9985 & 0.9704 & 0.5765 & 0.3498\\
& time &  1.0201 & 0.7729 & 0.6772 & 0.6584 & 0.7162 & 0.7193 & 1.0411 & 1.7427 & 1.5559\\
& $q$ & 17 & 17 & 17 & 17 & 17 & 17 & 17 & 17 & 17\\
 \hline
\multirow{5}{*}{I-AManPG}  & NMI & 0.9717 & 0.9722 & 0.9665 & 0.9726 & 0.9717 & 0.9709 & 0.9411 & 0.4298 & 0.1183\\
& AMI & 0.9415 & 0.9424 & 0.9368 & 0.9433 & 0.9415 & 0.9417 & 0.9131 & 0.3733 & 0.0498\\
& Mod. & 0.9362 & 0.8382 & 0.7367 & 0.6423 & 0.5427 & 0.4442 & 0.3424 & 0.1721 & 0.1089\\
& purity &  1.0000 & 1.0000 & 0.9947 & 1.0000 & 1.0000 & 0.9989 & 0.9773 & 0.5746 & 0.3379\\
& time &  1.0509 & 0.9585 & 0.7564 & 0.7069 & 0.8014 & 0.7777 & 1.1587 & 2.0853 & 1.9374\\
& $q$ & 18 & 18 & 18 & 18 & 18 & 18 & 18 & 18 & 18\\
\hline
\multirow{5}{*}{I-AManPG}  & NMI & 0.9883 & 0.9883 & 0.9845 & 0.9883 & 0.9796 & 0.9828 & 0.9520 & 0.4339 & 0.1264\\
& AMI & 0.9753 & 0.9753 & 0.9716 & 0.9753 & 0.9667 & 0.9699 & 0.9365 & 0.3803 & 0.0557\\
& Mod. & 0.9453 & 0.8460 & 0.7446 & 0.6474 & 0.5445 & 0.4468 & 0.3427 & 0.1728 & 0.1111\\
& purity & 1.0000 & 1.0000 & 0.9962 & 1.0000 & 0.9914 & 0.9949 & 0.9714 & 0.5634 & 0.3203\\
& time &  0.9389 & 0.9057 & 0.8584 & 0.9534 & 0.7417 & 1.0226 & 1.0085 & 2.1798 & 1.8913\\
& $q$ & 19 & 19 & 19 & 19 & 19 & 19 & 19 & 19 & 19\\
 \hline
\multirow{5}{*}{I-AManPG}  & NMI & 1.0000 & 1.0000 & 1.0000 & 1.0000 & 1.0000 & 0.9998 & 0.9600 & 0.4517 & 0.1294\\
& AMI & 1.0000 & 1.0000 & 1.0000 & 1.0000 & 1.0000 & 0.9998 & 0.9539 & 0.4037 & 0.0563\\
& Mod. & 0.9497 & 0.8499 & 0.7503 & 0.6500 & 0.5499 & 0.4498 & 0.3416 & 0.1735 & 0.1113\\
& purity & 1.0000 & 1.0000 & 1.0000 & 1.0000 & 1.0000 & 0.9999 & 0.9679 & 0.5605 & 0.3044\\
& time &  0.6577 & 0.5052 & 0.6423 & 0.9890 & 0.7112 & 0.4951 & 1.0755 & 1.7648 & 1.7855\\
& $q$ & 20 & 20 & 20 & 20 & 20 & 20 & 20 & 20 & 20\\
 \hline
\multirow{5}{*}{I-AManPG}  & NMI & 0.9944 & 0.9946 & 0.9946 & 0.9945 & 0.9945 & 0.9927 & 0.9620 & 0.4541 & 0.1322\\
& AMI & 0.9881 & 0.9885 & 0.9885 & 0.9881 & 0.9883 & 0.9863 & 0.9547 & 0.4056 & 0.0566\\
& Mod. & 0.9303 & 0.8338 & 0.7370 & 0.6380 & 0.5402 & 0.4422 & 0.3384 & 0.1731 & 0.1107\\
& purity & 1.0000 & 1.0000 & 1.0000 & 1.0000 & 1.0000 & 0.9985 & 0.9671 & 0.5498 & 0.1751\\
& time &  1.6041 & 1.2957 & 1.4677 & 1.1984 & 2.4890 & 1.3366 & 1.0639 & 1.9795 & 2.0597\\
& $q$ & 21 & 21 & 21 & 21 & 21 & 21 & 21 & 21 & 21\\
\hline
\multirow{5}{*}{I-AManPG}  & NMI & 0.9890 & 0.9890 & 0.9894 & 0.9890 & 0.9887 & 0.9868 & 0.9649 & 0.4649 & 0.1352\\
& AMI & 0.9764 & 0.9765 & 0.9773 & 0.9765 & 0.9763 & 0.9742 & 0.9500 & 0.4159 & 0.0574\\
& Mod. & 0.9124 & 0.8161 & 0.7239 & 0.6263 & 0.5315 & 0.4349 & 0.3358 & 0.1728 & 0.1123\\
& purity & 1.0000 & 1.0000 & 1.0000 & 1.0000 & 0.9996 & 0.9981 & 0.9831 & 0.5642 & 0.1770\\
& time &  1.7685 & 1.6543 & 1.4107 & 0.8671 & 1.7232 & 1.4291 & 0.9652 & 2.4757 & 2.5771\\
& $q$ & 22 & 22 & 22 & 22 & 22 & 22 & 22 & 22 & 22\\
\hline
\multirow{5}{*}{I-AManPG}  & NMI & 0.9835 & 0.9838 & 0.9842 & 0.9834 & 0.9836 & 0.9805 & 0.9532 & 0.4622 & 0.1402\\
& AMI & 0.9649 & 0.9654 & 0.9663 & 0.9647 & 0.9655 & 0.9618 & 0.9317 & 0.4105 & 0.0606\\
& Mod. & 0.8937 & 0.8009 & 0.7114 & 0.6148 & 0.5233 & 0.4278 & 0.3288 & 0.1710 & 0.1120\\
& purity & 1.0000 & 1.0000 & 1.0000 & 1.0000 & 0.9997 & 0.9975 & 0.9784 & 0.5662 & 0.1798\\
& time &  1.9044 & 1.7584 & 1.6944 & 1.2244 & 1.8282 & 2.0833 & 1.3244 & 2.9445 & 2.5133\\
& $q$ & 23 & 23 & 23 & 23 & 23 & 23 & 23 & 23 & 23\\
 \hline
\end{tabular}
}
\end{center}
\end{table}

Consider the results of different algorithms for the cases with the same computed number of communities. From the results in Table \ref{table2} and Table \ref{table3}, we observe that NMI, AMI, modularity and purity of I-AManPG are larger than the results of Danon's algorithm for $q=18$ and $\mu=0.2$,  Newman's spectral algorithm for $q=21$ and $\mu=0.2$ and Newman's spectral algorithm for $q=19$ and $\mu=0.3$.  NMI, AMI, modularity and purity of I-AManPG are competitive with the results of Louvain's algorithm for $q=19$ and $\mu=0.6$ and I-AManPG requires less time. 

Table \ref{table3} shows that for any particular $\mu_{\mathrm{LFR}}$, as the value of the input parameter $q$ moves away from $q_{true}=20$ the modularity decreases, while the NMI and AMI achieved I-AManPG move away from desireable values close to $1$. This does not mean that the partitions are not good relative to the ground-truth partition. For $\mu_{\mathrm{LFR}}=0$ and $ 0.1$ the community partitions for $q = 17, 18, 19, 20, 21, 22$ are perfectly nested, i.e., the extra communities of partition of $q+1$ are refinements of the partition of $q$ by splitting without crossing the ideal community boundaries. Since $q_{true}$ is in this set, this says that for values of $\mu_{\mathrm{LFR}}$ that imply strong community structure I-AManPG produces communities that respect the affinities of the ground-truth partition. For these two values of $\mu_{\mathrm{LFR}}$, when $q=10$, the farthest from $q_{true}$ in the set, the I-AManPG partitionis perfectly nested relative to the next partition, i.e., $q=17$, and the ground-truth partition. When $\mu_{\mathrm{LFR}}$ has values from  0.2 to 0.6, each partitioning for $q = 10, 17, 18, 19,21, 22, 23$ is well-nested with partitioning of $q_{true}$ and the associated purity values are very close to $1$. For $\mu_{\mathrm{LFR}}=0.7$ and $ 0.8$, the community structure is not strong in the ground-truth LFR networks and therefore purity would be expected to degrade. 
These results provide promising evidence for the possibility of development of a dynamic adaptation strategy for I-AManPG. Since the sparsification to project a Stiefel element to an assignment applies, in general, to a dense $N \times q$ matrix, storage and computation can become excessive when a large number of communities must be produced. Effective nesting means this can be avoided efficiently as is done with other divisive projection-based algorithms.

\subsection{Normalized cut}

Normalized cut has been widely used for image segmentation.
Its optimization formulation is given by 
\begin{equation} \label{e37}
\min_{X \in \mathcal{A}_v} {f_{\mathrm{NC}}(X) = - \mathrm{trace} ( X^T D^{-1/2} W D^{-1/2} X )}.
\end{equation}
This problem assumes graph-based data that is represented by an appropriate matrix characterizing the relationships between the basic data elements from the application problem. 
In the case of gray image segmentation, the matrix $W \in \mathbb{R}^{mn \times mn}$ is an affinity matrix of an $m$ by $n$ pixels gray image, $D \in \mathbb{R}^{mn \times mn}$ is a diagonal matrix with $D_{ii} = \sum_{j = 1}^{mn} W_{i j}$, and $v = \mathrm{diag}(D^{1/2})$. Here, we use the approach in~\cite{CYS2004} to choose $W$ and $D$. 

Problem~\eqref{e37} can be optimized by the weighted kernel $k$-means algorithm, see e.g.,~\cite[Algorithm~1]{DGK2005}. Note that Problem~\eqref{e37} has many low-quality local minimizers and descent optimization algorithms usually are not able to escape from them. Thus, initialization plays an important role in finding an acceptable solution. Let $U$ be the $n \times q$ matrix of the $q$ leading eigenvector of the matrix $D^{-1/2} W D^{-1/2}$. If $X$ is only required to be orthonormal, then $U$ is a global minimizer of~\eqref{e37}. Since $U$ is unlikely to be in $\mathcal{A}_v$, one approach is to find a matrix in $\mathcal{A}_v$ that is close to $U$. Different notions of closeness yield different methods. Next, we introduce four initialization methods,  including the proposed one based on AManPG. 

Bach and Jordan~\cite{BJ2003} seek to find a matrix $Y \in \mathcal{A}_v$ that minimizes
\begin{equation} \label{e38}
\|U U^T - Y Y^T\|_F.
\end{equation}
In other words, the difference between $U$ and $Y$ is measured by the orthogonal projection matrix. The weighted kernel $k$-means is suggested to solve~\eqref{e38} see~\cite[Figure~1]{BJ2003}. However, similar to~\eqref{e37}, the kernel $k$-means for~\eqref{e38} may also get stuck in a local minimizer. We use $k$-means++ in Matlab for the initialization of the kernel $k$-means for~\eqref{e38}.

Shi and Malik~\cite{SM2000} propose to find an indicator matrix that is closest to $U$ up to a rotation.
Specifically, let $\tilde{U}$ denote the matrix formed by normalizing all rows of $U$. The task is to find an indicator matrix $Z$ and a $q$-by-$q$ orthonormal matrix $Q$ that minimize
\[
\|Z - \tilde{U} Q\|_F.
\]
Shi and Malik~\cite{SM2000} use an alternating minimization algorithm to find $Z$ and $Q$. Note that this approach neither guarantees to find the global optimum nor use the weight vector $v$. Therefore, this approach may not find a satisfactory solution. Here, we use the C and Matlab hybrid implementation from~\cite{CYS2004}.

Karypis and Kumar~\cite{KK1998} developed METIS, a fast, multi-level graph partitioning algorithm that produces equally-sized clusters. It is shown to be an effective method for the kernel $k$-means initialization. Note that METIS does not aim to minimize the objective~\eqref{e37}. We use the C implementation from~\url{http://glaros.dtc.umn.edu/gkhome/metis/metis/download} with the Matlab interface from~\url{https://github.com/dgleich/metismex}.

We propose to initialize the weighted kernel $k$-means algorithm by I-AManPG. 
Specifically, Problem~\eqref{e37} can be reformulated as
\begin{equation} \label{e39}
\min_{X \in \mathcal{F}_v} - \mathrm{trace} ( X^T D^{-1/2} W D^{-1/2} X ) + \lambda \|X\|_1,
\end{equation}
which can be optimized by I-AManPG. We further propose to gradually increasing $\lambda$ rather than a fixed value of $\lambda$ since increasing $\lambda$ tends to give better solutions in our experiments\footnote{The $\lambda$ in I-AManPG increases by 0.01, 0.04, and 0.2}. The clusters are specified by $P_{\mathcal{A}_v}(X_*)$, as described in Section~\ref{sect:para}. Such clusters are then used as initializations for the weighted kernel $k$-means algorithm.

The four initialization methods are denoted, respectively, by BJ, SM, ME, and AM. Their combinations with the weighted kernel $k$-means algorithms are denoted, respectively, by BJ-k, SM-k, ME-k, and AM-k. The implementation of the weighted kernel $k$-means algorithm is modified from~\cite{Chen2021}\footnote{The implementation in~\cite{Chen2021} is for unweighted kernel $k$-means. We modified it for weighted kernel $k$-means.}. The test images are from~\cite{CYS2004} and the built-in images in Matlab. We further resize them to have 160-by-160 pixels as shown in Figure~\ref{fig:testedimages}.

\begin{figure}
\includegraphics[width=1\textwidth]{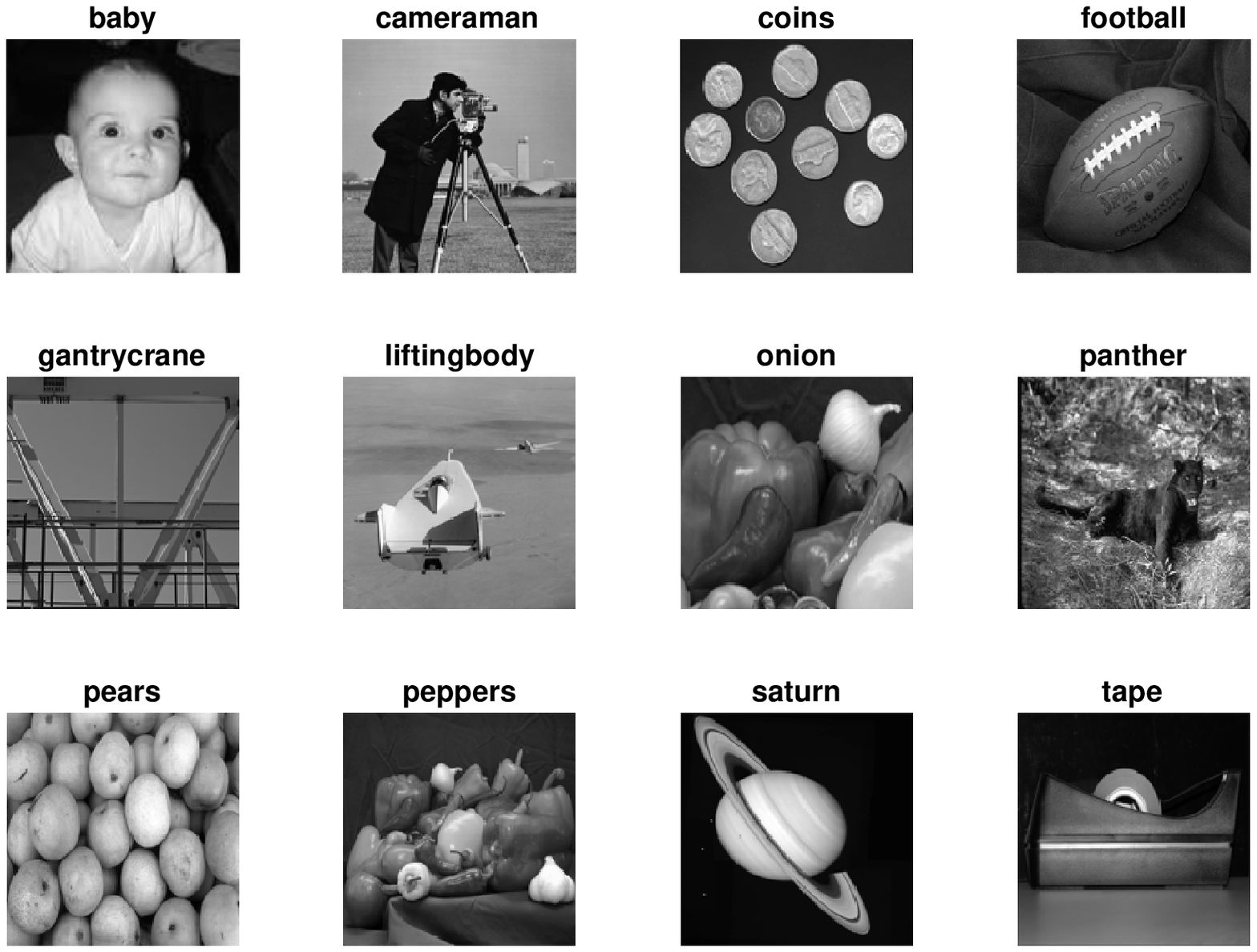}
\vspace{-4em}
\caption{The test images.}
\label{fig:testedimages}
\end{figure}

An average of the negative function values $-f_{\mathrm{NC}}$ of 10 random runs are reported in Figure~\ref{fig:FNC}. We do not report the computational time since the implementations of these methods use different languages and 
their computational time cannot be rigorously compared. 
The qualities of these methods are compared based on the objective function value $f_{\mathrm{NC}}$.
As shown in the figure, METIS initializations are not preferred since they do not aim to minimize $f_{\mathrm{NC}}$. Though SM, SM-k, BJ, BJ-k are competitive to AM and AM-k in many cases, they do not perform well in certain images, such as ME and ME-k for the football image with 3 clusters, and BJ and BJ-k for the tape image with 3 clusters. AManPG based methods are clearly most robust in the sense of minimizing the function $f_{\mathrm{NC}}$ over $\mathcal{A}_v$. The values of $-f_{\mathrm{NC}}$ by AM-k are often the highest one. Even if they are not, they are still close to the highest ones. 
The empirical evidence supports the expectations given in the motivation discussion above that I-AManPG is competitive with or superior to initialization strategies in the current literature. 

\begin{figure}
\centerline{
\includegraphics[width=0.35\textwidth]{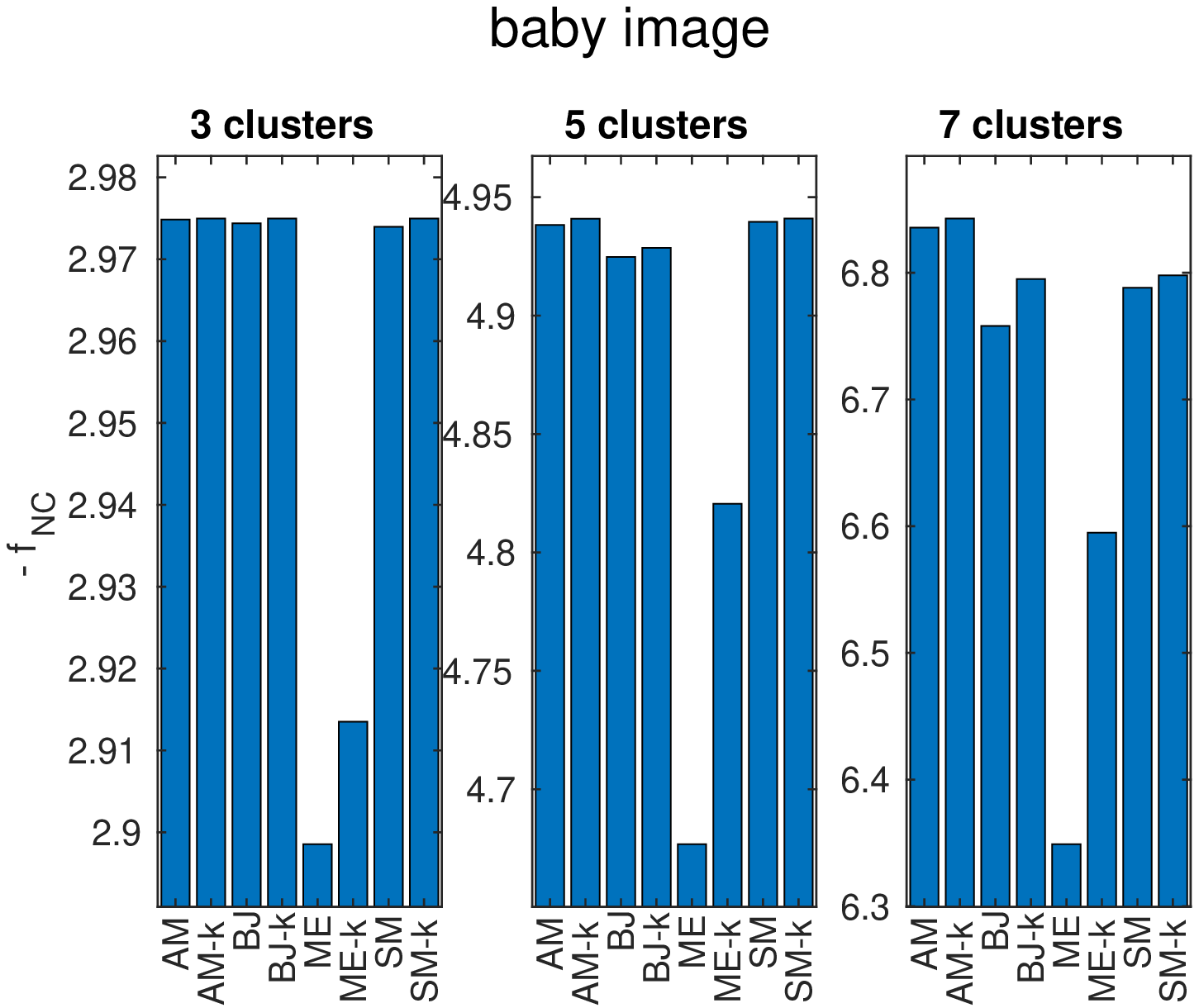}
\includegraphics[width=0.35\textwidth]{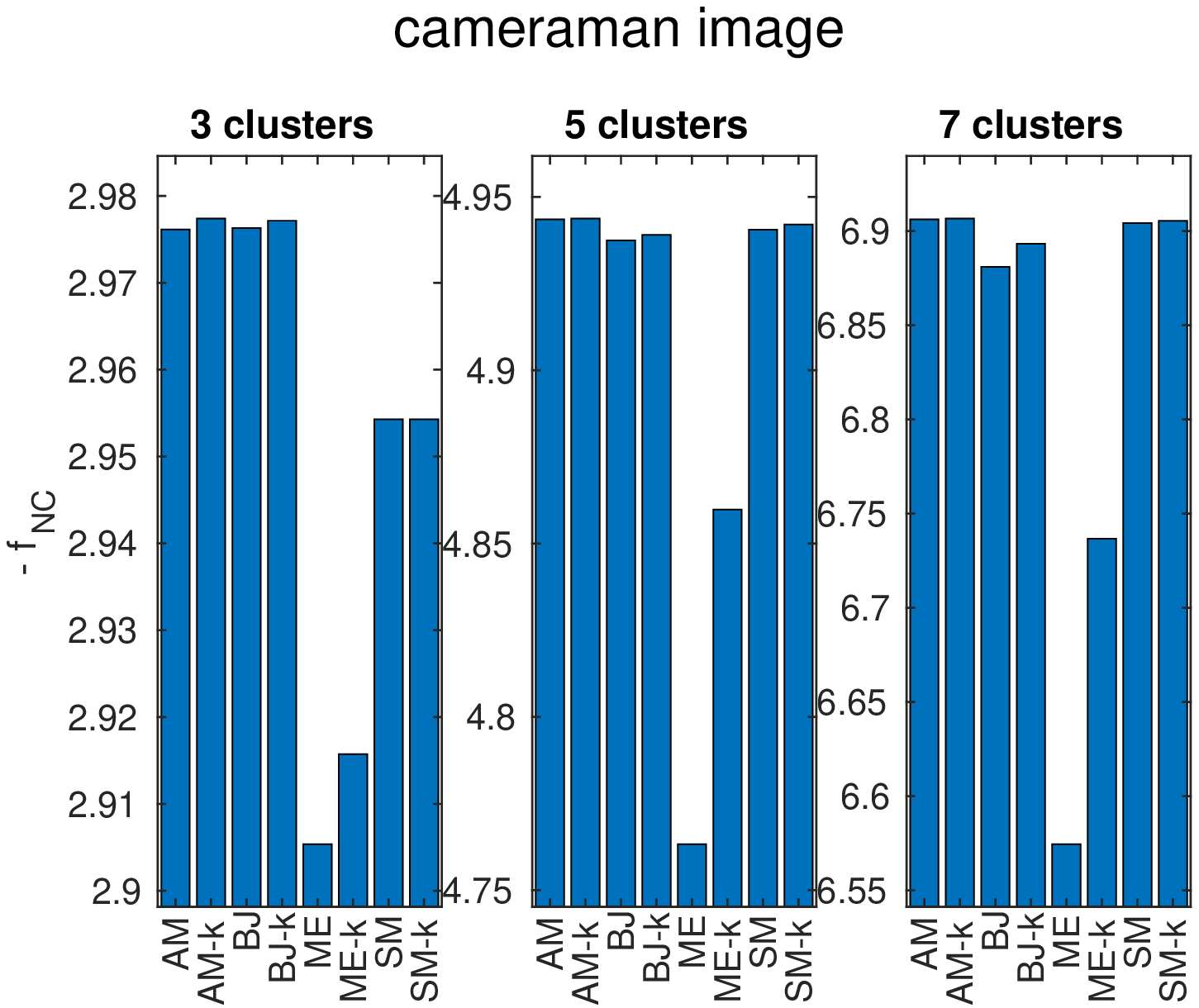}
\includegraphics[width=0.35\textwidth]{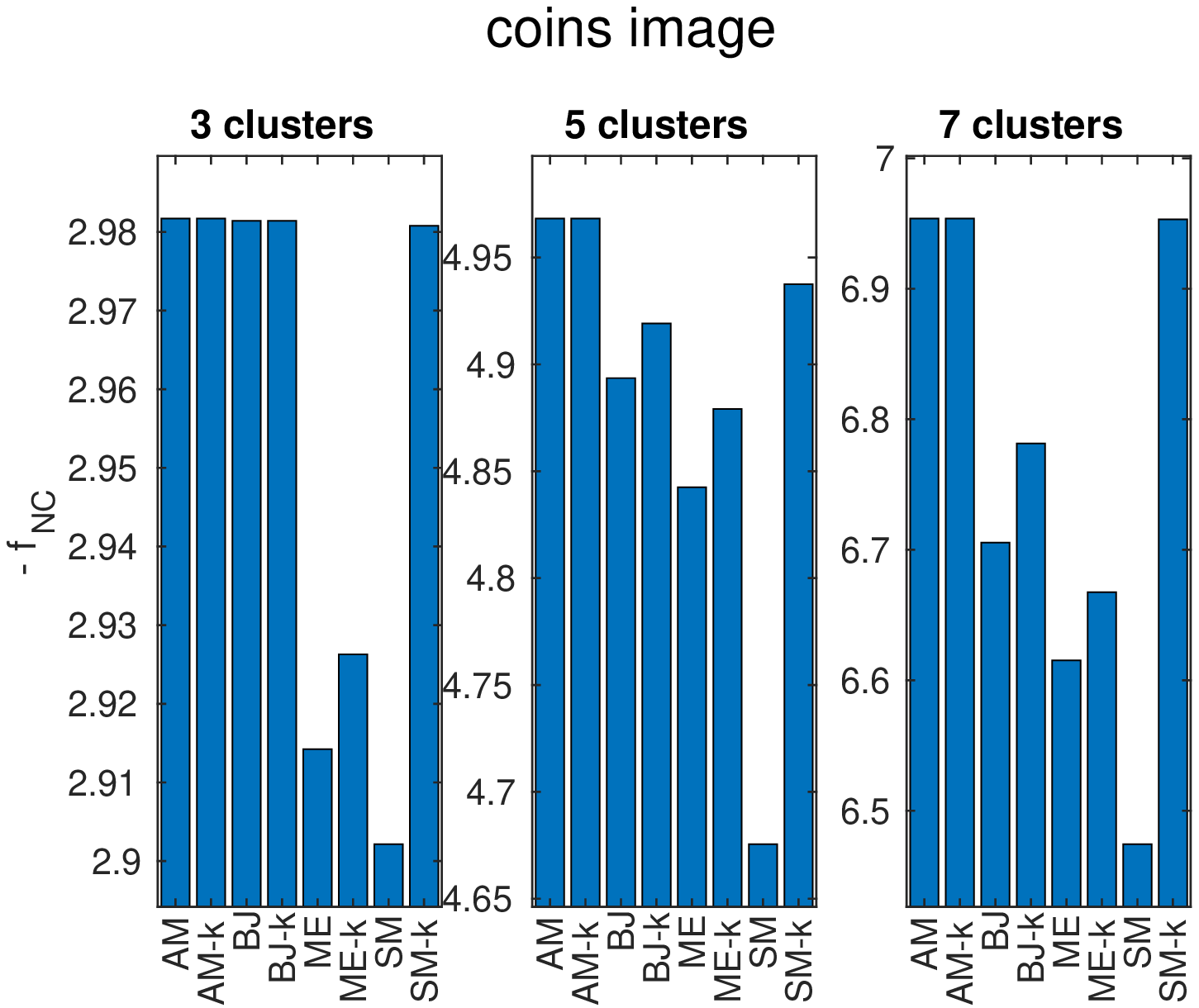}
}
\centerline{
\includegraphics[width=0.35\textwidth]{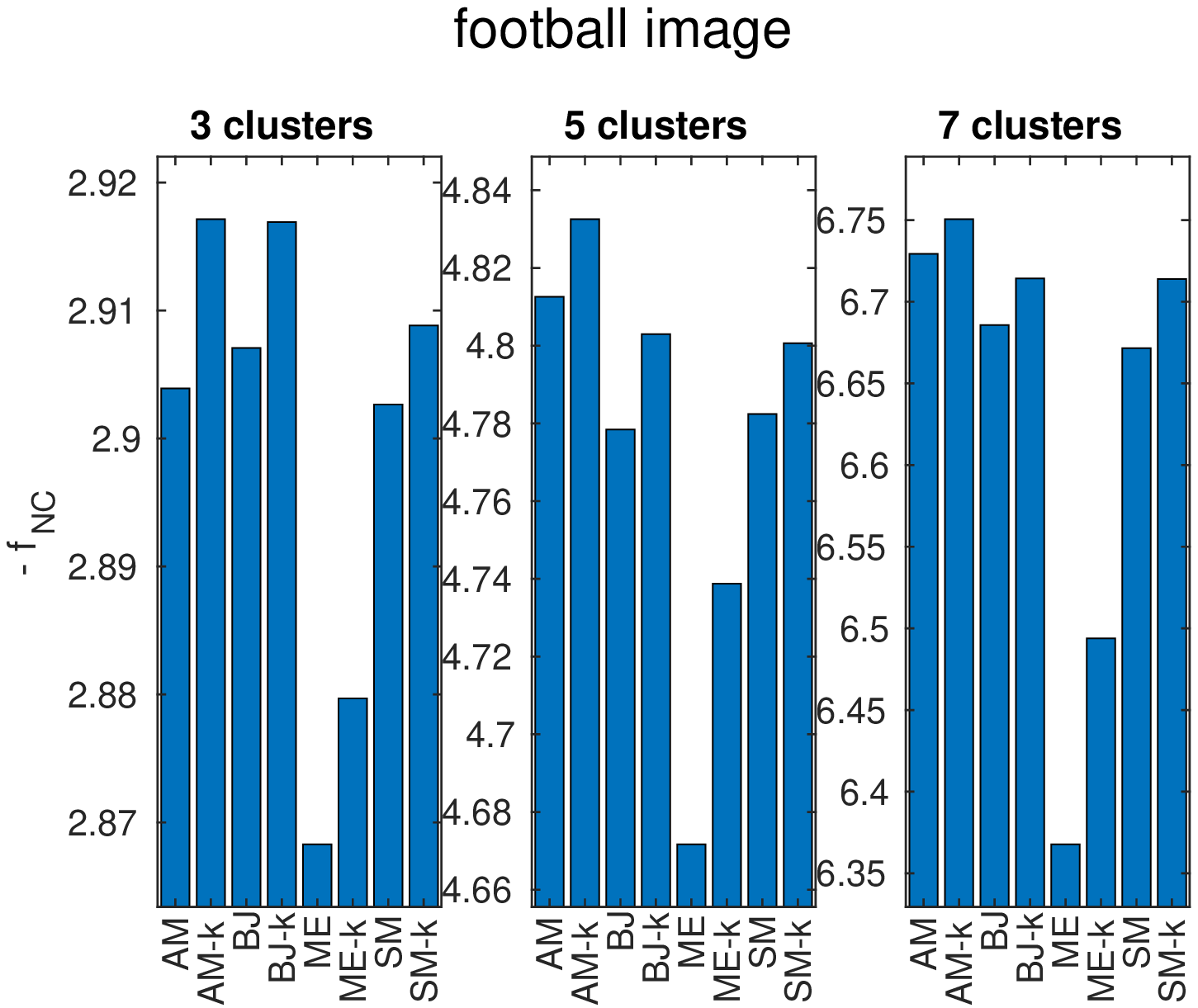}
\includegraphics[width=0.35\textwidth]{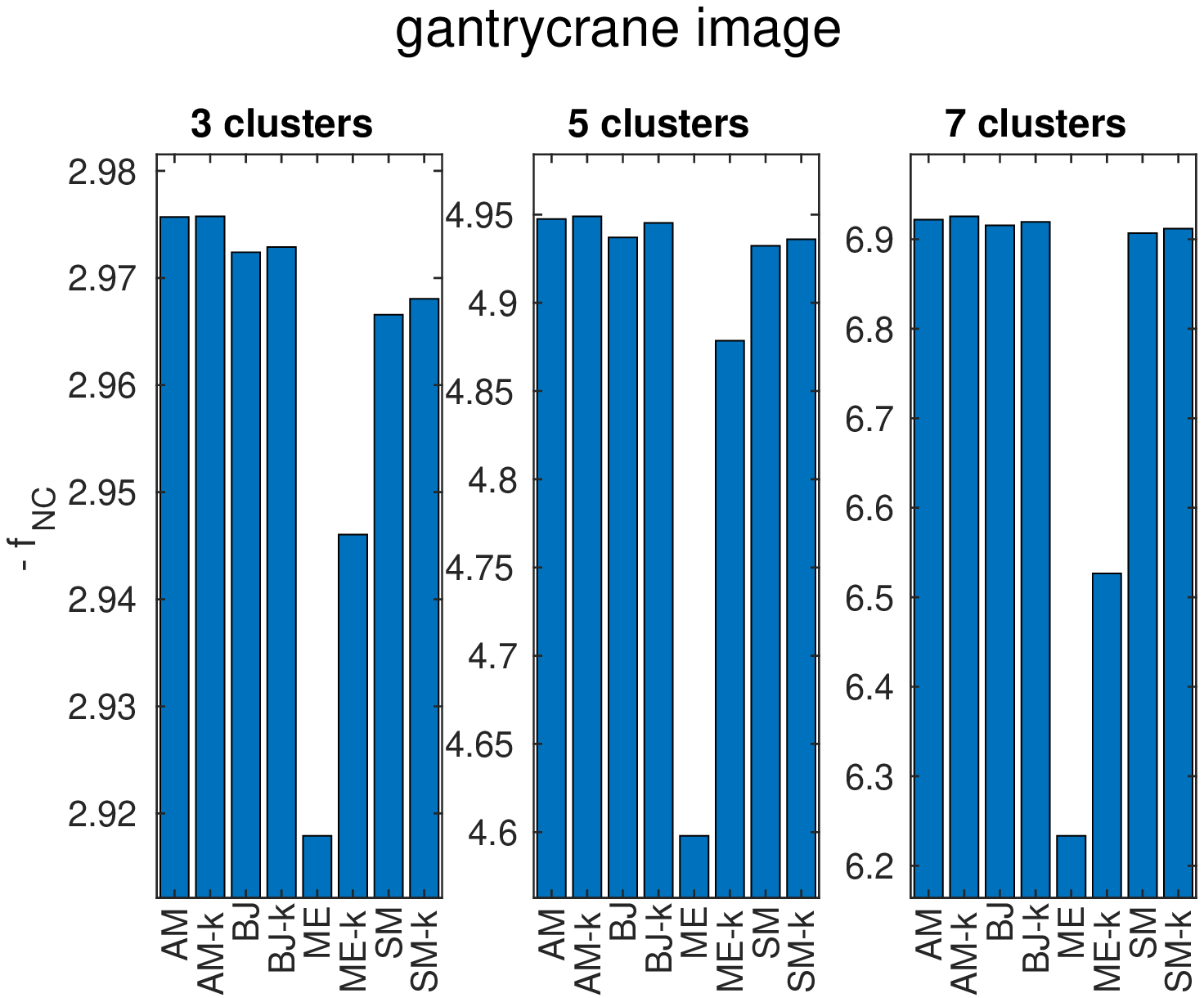}
\includegraphics[width=0.35\textwidth]{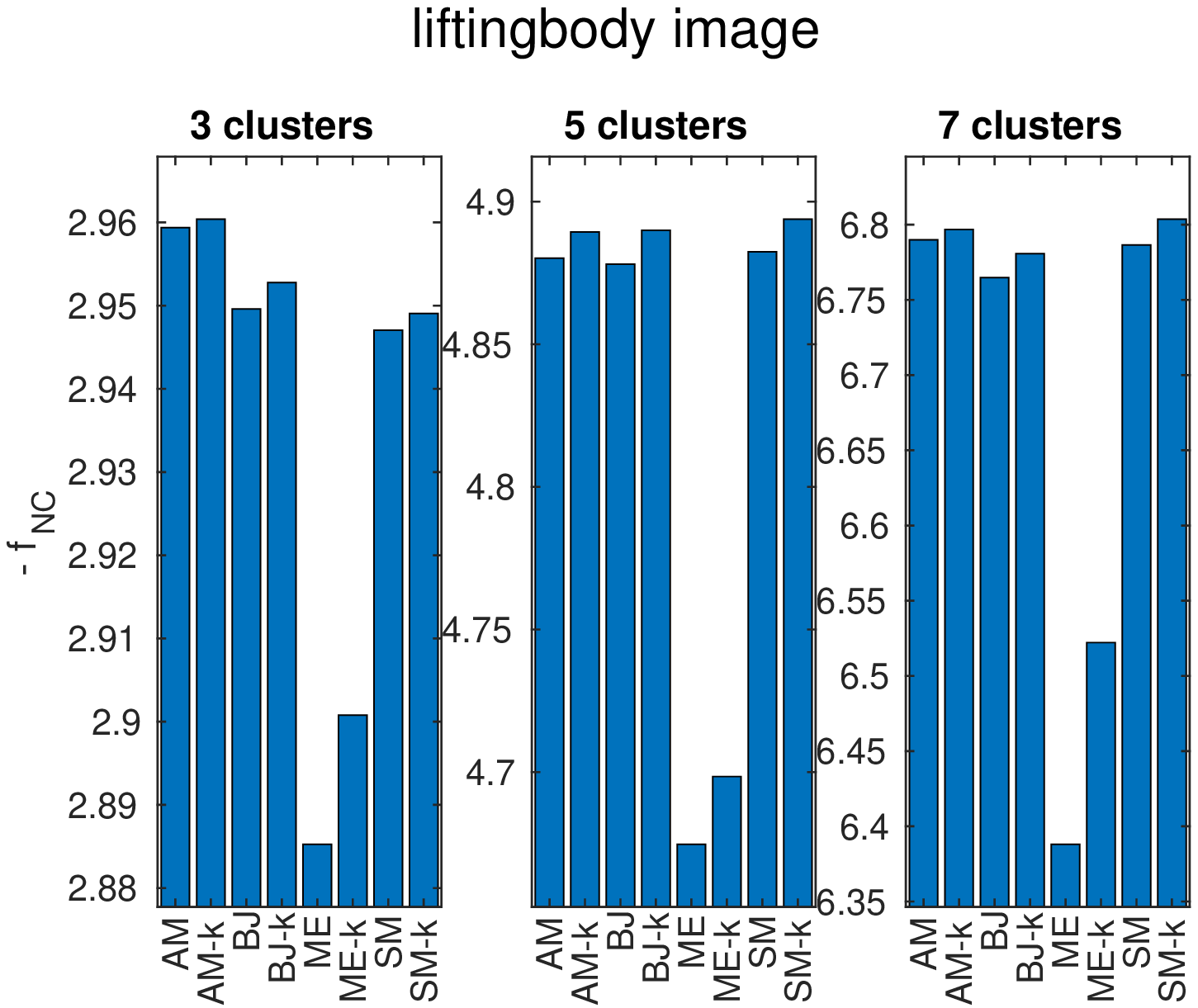}
}
\centerline{
\includegraphics[width=0.35\textwidth]{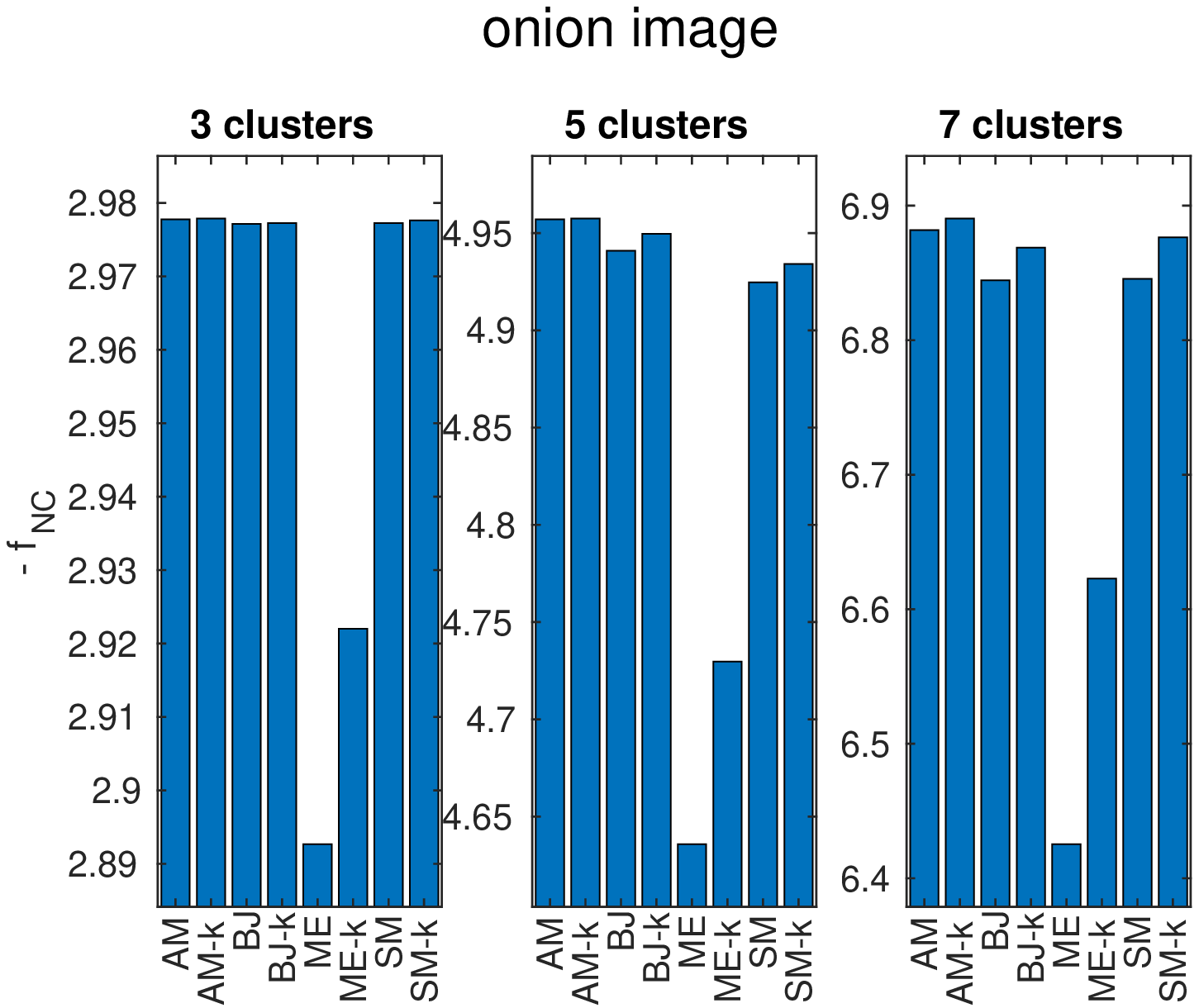}
\includegraphics[width=0.35\textwidth]{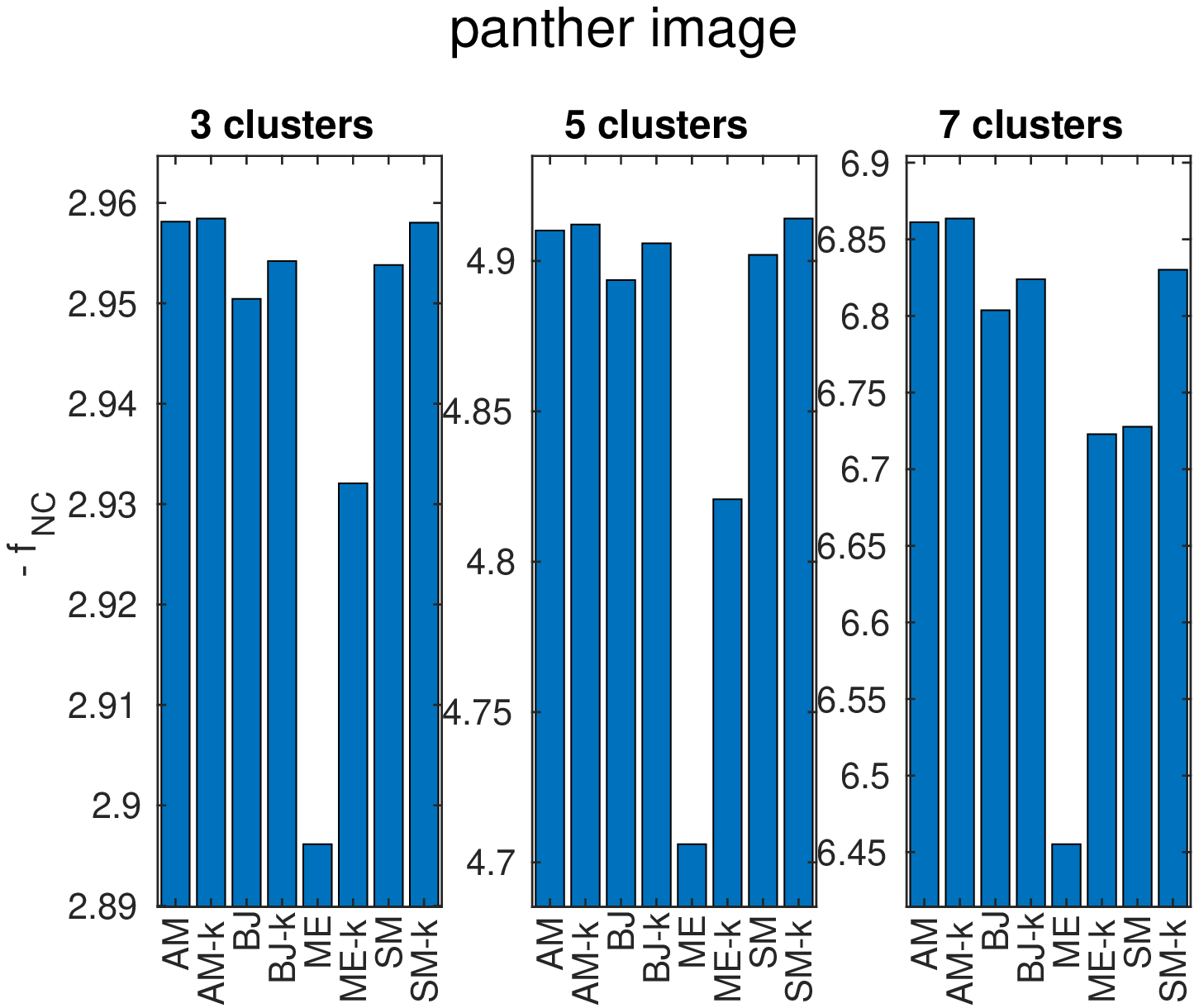}
\includegraphics[width=0.35\textwidth]{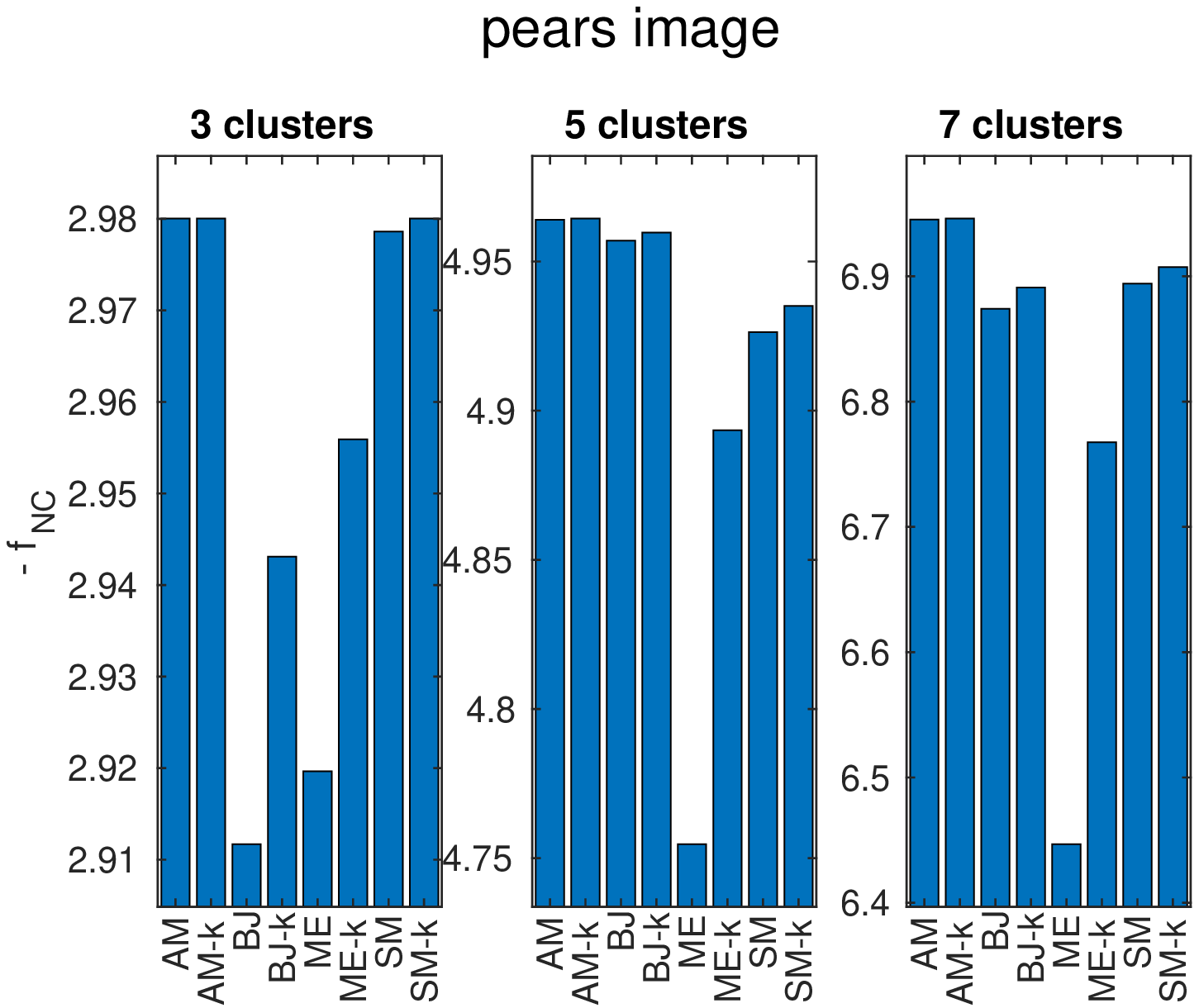}
}
\centerline{
\includegraphics[width=0.35\textwidth]{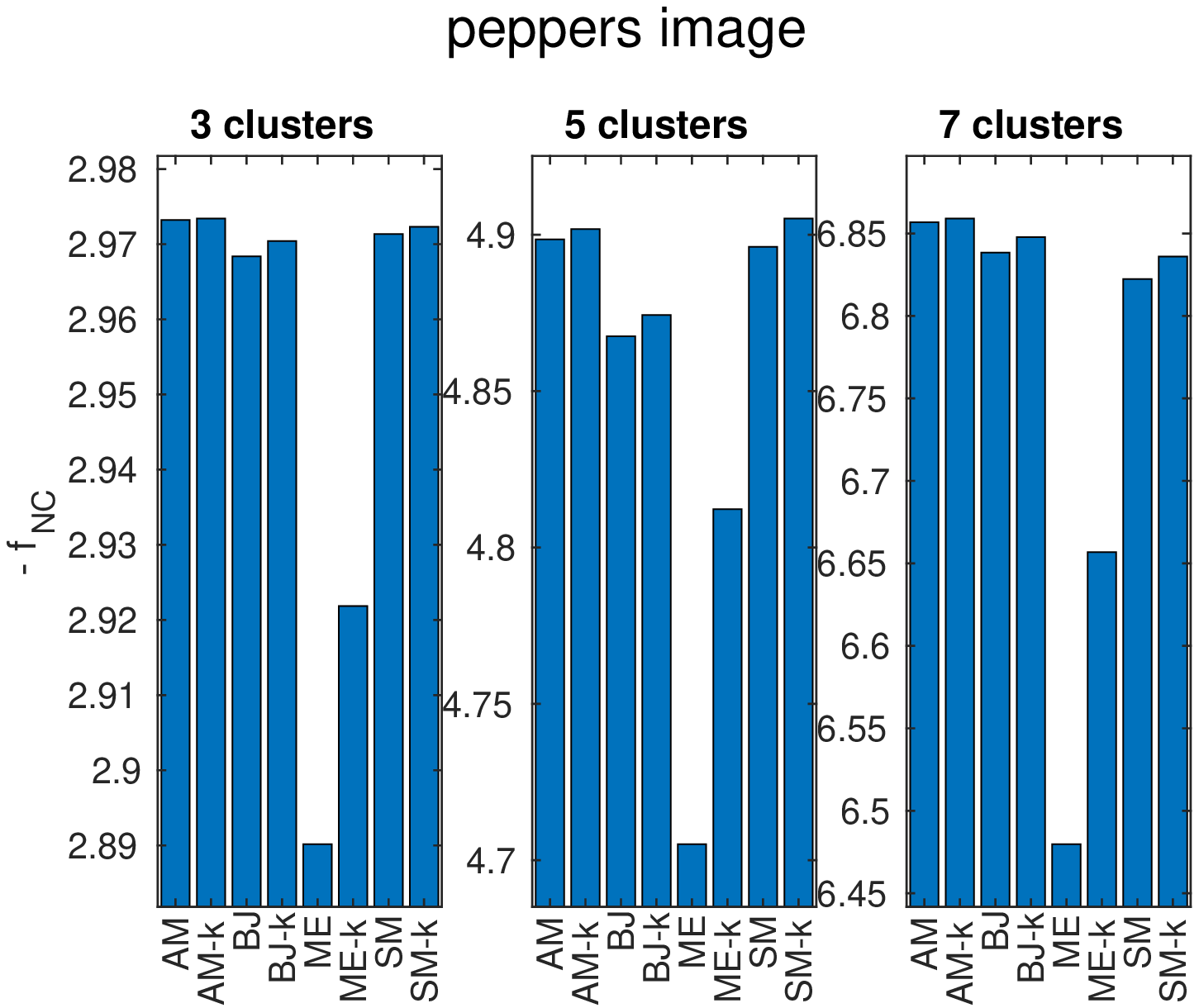}
\includegraphics[width=0.35\textwidth]{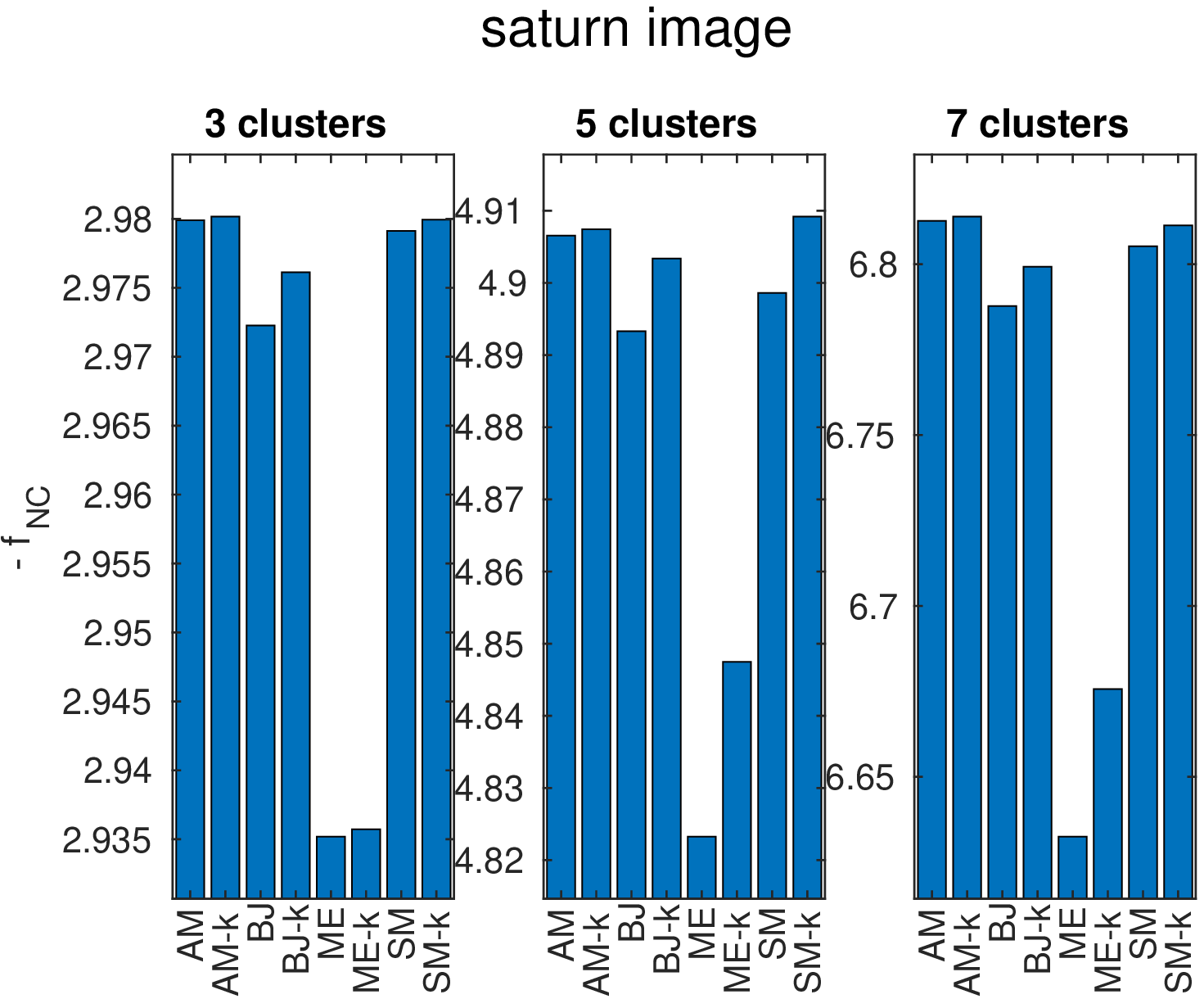}
\includegraphics[width=0.35\textwidth]{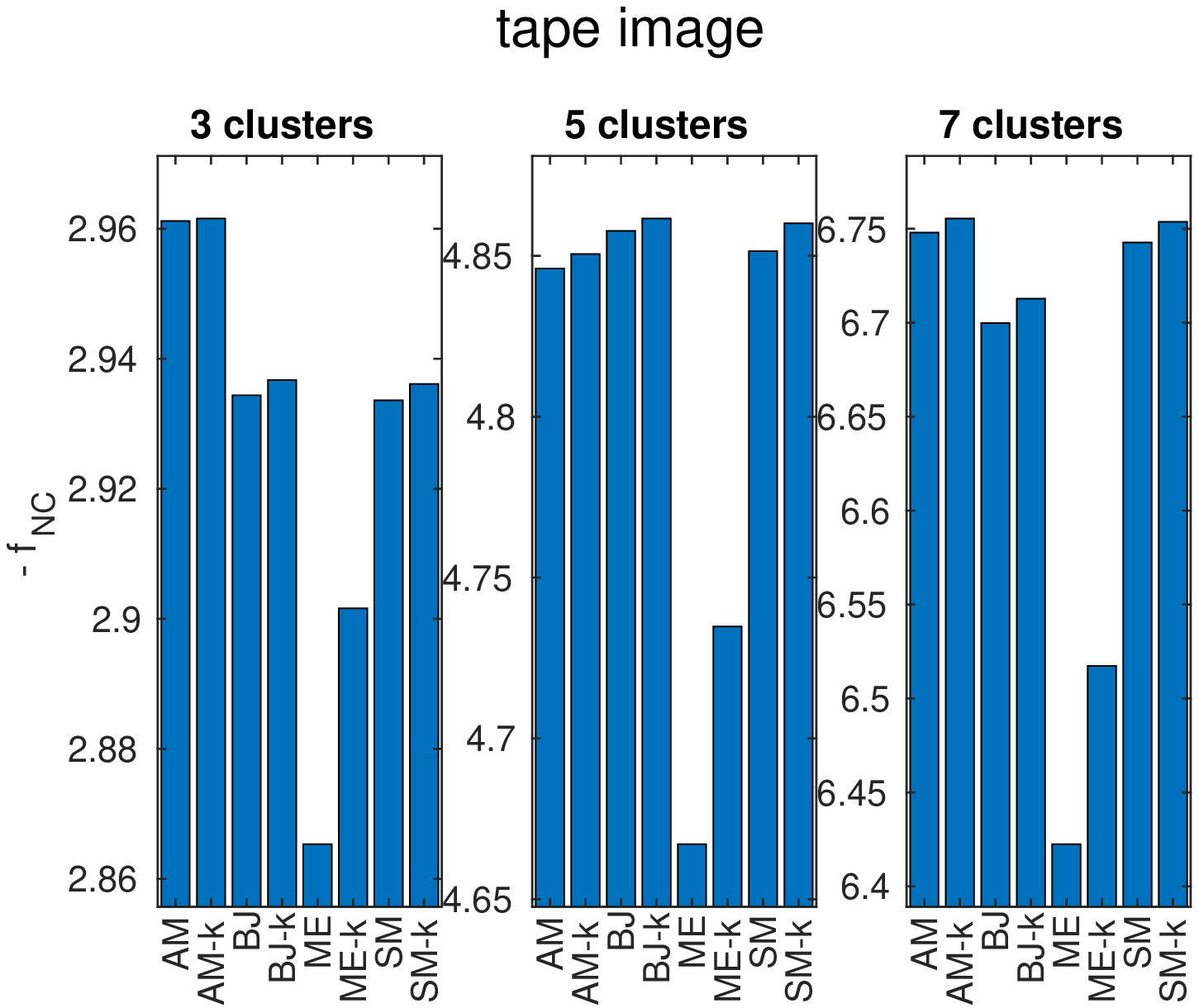}
}
\caption{An average of 10 random runs is reported. $y$-axis represents the function values. Multiple numbers of clusters are tested.
}
\label{fig:FNC}
\end{figure}

\delete{

\subsubsection{Comparison of the effectiveness of I-AManPG and state-of-the-art community detection methods}

In this section, community detection by the optimization model~\eqref{e40} using I-AManPG is compared to two state-of-the-art methods: Danon et al.'s algorithm~\cite{danon2006effect} and the Louvain method~\cite{blondel2008fast}. 

To make fair comparisons, we use publicly-available Matlab implementations of these algorithms.
The codes for Danon et al.'s algorithm and Louvain's algorithm 
are respectively from~\cite{AKehagiascdtoolbox} and ~\cite{matlab_Louvain_AS}. 

To compare the effectiveness of the three methods, we consider three quality measurements: normalized mutual information (NMI)~\cite{danon2005comparing}, adjusted mutual information (AMI) \cite{vinh2010information}, and  purity~{\cite{manning2008utze}.} NMI is a similarity measure between two partitions that represents their normalized mutual entropy. AMI further corrects the measure for randomness by adopting a hypergeometric model of randomness. We refer interested readers to~\cite[(2)]{danon2005comparing} and~\cite[Section~4.1]{vinh2010information} for the definitions. Both NMI and AMI take on values between 0 and 1. Values closer 1 indicate greater consistency between the partitions. 
Given two partitions $X$ and $Y$ of $N$ nodes, the purity is given by
$
\mbox{purity}( X, Y ) = \frac{1}{N} \sum_k \max_j \vert X_k \cap Y_j\vert,
$
where $X_k$ denotes the set of nodes in $k$-th community of partition $X$, and likewise for $Y_j$, and $\vert X_k \cap Y_j\vert$ denotes the number of nodes in $X_k \cap Y_j$. The value of purity is also between 0 and 1.  
The closer it is to one, the better the two partitions are nested. 
In our numerical experiments, the ground truth is known and therefore the computed partition is compared to the ground truth. When the two partitionings do not have same number of communities good nesting of the partition with the larger number of communities in the partition with the smaller number is an indication that further division or agglomeration could yield a closer approximation of ground truth. Such a pair of partitions is therefore preferred to a pair with lower purity. Since purity is not symmetric, we take $X$ to be the partition with the larger number of communities.


For LFR benchmark networks used in the comparisons in this section,  the parameters $\tau_1$, $\tau_2$, $N$, $d_{ave}$, $d_{max}$, $N_c$, and $n_c$ are respectively set to $-2$, $-1$, $1000$, $20$, $40$, $50$, and $20$. The value of $\lambda$ in~\eqref{e40} is 0.3, as in the previous set of experiments. The empirical results with multiple values of $\mu_{\mathrm{LFR}}$ are reported in Table~\ref{table2}. As before, each result is an average over 10 randomly selected LFR benchmark networks.

\begin{table}[htbp]
\renewcommand\tabcolsep{2.0pt}
\begin{center}
\caption{Compare the effectiveness of I-AManPG to other state-of-the-art methods. } \label{table2}
\resizebox{\textwidth}{!}
{
\begin{tabular}{l |c |c| c |c |c |c| c| c|c|c}
\hline
\hline
\multirow{2}{*}{ } & \multirow{2}{*}{} & \multicolumn{9}{c}{$\mu_{\mathrm{LFR}}$}\\
\cline{3-11}
&  & 0 & 0.1 & 0.2 & 0.3 & 0.4 & 0.5 & 0.6 & 0.7 &  0.8\\
 \hline
 \multirow{5}{*}{Danon } & NMI &  0.9998 & 0.9891 & 0.9394 & 0.8504 & 0.7331 & 0.5808 & 0.3781 & 0.1412 & 0.0548\\
 & AMI &  0.9998 & 0.9870 & 0.9166 & 0.7922 & 0.6399 & 0.4736 & 0.2878 & 0.0935 & 0.0215\\
 & Mod. &  0.9496 & 0.8436 & 0.7225 & 0.5920 & 0.4687 & 0.3452 & 0.2458 & 0.1892 & 0.1814\\
 & purity &  0.9999 & 0.9938 & 0.9739 & 0.9414 & 0.9058 & 0.8344 & 0.6886 & 0.4269 & 0.3095\\
 & time &  2.8304 & 2.8597 & 2.8199 & 2.7648 & 2.7287 & 2.8318 & 2.7739 & 2.7625 & 2.7318\\
 & $q_c$ & 20 & 20 & 18 & 15 & 11 & 9 & 7 & 7 & 8\\
 \hline
  \multirow{5}{*}{Danon\_force\_q} & NMI &  0.9998 & 0.9889 & 0.9397 & 0.8494 & 0.7309 & 0.5794 & 0.3939 & 0.1778 & 0.0954\\
& AMI &  0.9998 & 0.9870 & 0.9201 & 0.7962 & 0.6415 & 0.4727 & 0.2950 & 0.1035 & 0.0282\\
 & Mod. &  0.9496 & 0.8433 & 0.7209 & 0.5868 & 0.4615 & 0.3394 & 0.2431 & 0.1872 & 0.1794\\
& purity & 0.9999 & 0.9935 & 0.9714 & 0.9349 & 0.8954 & 0.8198 & 0.6778 & 0.4201 & 0.3065\\
& time & 2.8353 & 2.9606 & 2.8222 & 2.8068 & 2.8055 & 2.8959 & 2.8183 & 2.8459 & 2.7432\\
& $q_c$ & 20 & 20 & 20 & 20 & 20 & 20 & 20 & 20 & 20\\
\hline
 \multirow{5}{*}{Louvain}  & NMI &1.0000 & 1.0000 & 1.0000 & 1.0000 & 1.0000 & 0.9987 & 0.9805 & 0.2862 & 0.0784\\
 & AMI & 1.0000 & 1.0000 & 1.0000 & 1.0000 & 1.0000 & 0.9974 & 0.9652 & 0.2249 & 0.0358\\
 & Mod. & 0.9497 & 0.8499 & 0.7503 & 0.6500 & 0.5499 & 0.4496 & 0.3477 & 0.2098 & 0.1967\\
 & purity & 1.0000 & 1.0000 & 1.0000 & 1.0000 & 1.0000 & 0.9999 & 0.9950 & 0.4734 & 0.2660\\
 & time & 0.5444 & 0.7291 & 1.3703 & 1.8963 & 2.6797 & 3.3418 & 4.5768 & 9.2669 & 8.8130\\
  & $q_c$ & 20 & 20 & 20 & 20 & 20 & 20 & 19 & 11 & 11\\
 \hline
 \multirow{5}{*}{Louvain\_force\_q}  & NMI &1.0000 & 1.0000 & 1.0000 & 1.0000 & 1.0000 & 0.9987 & 0.9805 & 0.2981 & 0.0847\\
& AMI & 1.0000 & 1.0000 & 1.0000 & 1.0000 & 1.0000 & 0.9974 & 0.9652 & 0.2382 & 0.0392\\
 & Mod. & 0.9497 & 0.8499 & 0.7503 & 0.6500 & 0.5499 & 0.4496 & 0.3477 & 0.2098 & 0.1967\\
& purity & 1.0000 & 1.0000 & 1.0000 & 1.0000 & 1.0000 & 0.9999 & 0.9950 & 0.4642 & 0.2418\\
 & time & 0.5440 & 0.7478 & 1.0333 & 1.2042 & 1.7002 & 2.0761 & 2.7676 & 5.4529 & 5.5061\\
 & $q_c$ & 20 & 20 & 20 & 20 & 20 & 20 & 19 & 12 & 12\\
\hline
\multirow{5}{*}{I-AManPG}  & NMI & 1.0000 & 1.0000 & 1.0000 & 1.0000 & 1.0000 & 0.9998 & 0.9600 & 0.4517 & 0.1294\\
& AMI & 1.0000 & 1.0000 & 1.0000 & 1.0000 & 1.0000 & 0.9998 & 0.9539 & 0.4037 & 0.0563\\
& Mod. & 0.9497 & 0.8499 & 0.7503 & 0.6500 & 0.5499 & 0.4498 & 0.3416 & 0.1735 & 0.1113\\
& purity & 1.0000 & 1.0000 & 1.0000 & 1.0000 & 1.0000 & 0.9999 & 0.9679 & 0.5605 & 0.3044\\
& time &  0.6357 & 0.4693 & 0.5870 & 0.9494 & 0.6749 & 0.4720 & 1.0332 & 1.6307 & 1.6757\\
& $q$ & 20 & 20 & 20 & 20 & 20 & 20 & 20 & 20 & 20\\
 \hline
\end{tabular}
}
\end{center}
\end{table}

The input parameter for I-AManPG determining the number of communities to be produced is $q$ and $q_c$ is the number of communities computed by each method. The "force\_q" label denotes the versions with modified termination criteria so that $q_c$ is as close to $q_{true}$ as the methods allow. 

From the results in Table \ref{table2}, we observe that when $\mu_{\mathrm{LFR}}=0$,  I-AManPG yields $\mathrm{NMI} = \mathrm{AMI} = \mathrm{purity} = 1$, the same modularity value and the same assignment to $q_{true} = 20$ strongly connected communities.  The Louvain method also produces the ground-truth communities. Danon's method 
produces results that are very close to the ground-truth. Specifically it detects the ground-truth communities for 9 of 10 random LFR graphs.  When $\mu_{\mathrm{LFR}}$ is 0.1 to 0.4, I-AManPG and the Louvain method detect the ground-truth communities while Danon's method is not as successful. When $\mu_{\mathrm{LFR}}=0.5, 0.6$, I-AManPG produces results very close to ground-truth and the results are competitive with those of the Louvain method, but require less computational time. When $\mu_{\mathrm{LFR}}=0.7, 0.8$,  the results for all four methods are far from ground-truth partitions because the community structure is not strong for such $\mu_{\mathrm{LFR}}$ values.  Danon's method 
detects relatively inaccurate communities and has relatively small qualifying external or internal measurements, i.e., NMI, AMI, purity and modularity, for all noisy cases.  From the reported computational times, it can be seen that I-AManPG requires relatively less time than the other methods.  Note that the number of edges, $m$, does not change much while the distribution of edges changes significantly as the mixing parameter increases. So, the computational times for Danon's method 
does not change much as the mixing parameter increases since their computational times depend more on $m$ than on the distribution of edges. 

As the mixing parameter increases, the difficulty level of detecting the correct number of communities increases as well.  I-AManPG requires the desired number of communities as an input parameter value, $q$,  and the choice of an initial $q$ and the development of a dynamic adaptation strategy are key ongoing research tasks for I-AManPG. Since the experiments in Table \ref{table2} use $q=q_{true}$ for I-AManPG and the other methods that are not "forced" are given no indication of $q_{true}$, experiments where I-AManPG uses $q \neq q_{true}$ probe the quality of the $q \neq q_{true}$ communities produced by I-AManPG compared to ground truth.
For each value of the input parameter $q =10,  18, 19, 20, 23$ and mixing parameter $\mu_{\mathrm{LFR}}=0, 0.1, \dotsc , 0.8$,  I-AManPG was applied to 10 randomly generated LFR benchmark networks. The results are shown in Table \ref{table3}.

\begin{table}[htbp]
\renewcommand\tabcolsep{2.0pt}
\begin{center}
\caption{Testing the effectiveness of I-AManPG for input parameter values $q\neq q_{true}$.}\label{table3}
\small{
\begin{tabular}{l |c |c| c |c |c |c| c| c|c|c}
\hline
\hline
\multirow{2}{*}{ } & \multirow{2}{*}{} & \multicolumn{9}{c}{$\mu_{\mathrm{LFR}}$}\\
\cline{3-11}
&  & 0 & 0.1 & 0.2 & 0.3 & 0.4 & 0.5 & 0.6 & 0.7 &  0.8\\
\multirow{5}{*}{I-AManPG}  & NMI & 0.7178 & 0.7178 & 0.7178 & 0.7176 & 0.7137 & 0.6966 & 0.6230 & 0.2957 & 0.0795\\
& AMI & 0.5452 & 0.5452 & 0.5452 & 0.5451 & 0.5428 & 0.5327 & 0.4856 & 0.2202 & 0.0324\\
& Mod. & 0.6924 & 0.6198 & 0.5474 & 0.4733 & 0.4005 & 0.3251 & 0.2531 & 0.1534 & 0.1007\\
 & purity &  1.0000 & 1.0000 & 1.0000 & 0.9999 & 0.9971 & 0.9842 & 0.9243 & 0.6363 & 0.5118\\
& time &  0.1939 & 0.2175 & 0.5020 & 0.3373 & 0.2593 & 0.4983 & 0.3297 & 0.7154 & 0.6616\\
 & $q$ & 10 & 10 & 10 & 10 & 10 & 10 & 10 & 10 & 10\\
 \hline
\multirow{5}{*}{I-AManPG}  & NMI & 0.9717 & 0.9722 & 0.9665 & 0.9726 & 0.9717 & 0.9709 & 0.9411 & 0.4298 & 0.1183\\
& AMI & 0.9415 & 0.9424 & 0.9368 & 0.9433 & 0.9415 & 0.9417 & 0.9131 & 0.3733 & 0.0498\\
& Mod. & 0.9362 & 0.8382 & 0.7367 & 0.6423 & 0.5427 & 0.4442 & 0.3424 & 0.1721 & 0.1089\\
& purity &  1.0000 & 1.0000 & 0.9947 & 1.0000 & 1.0000 & 0.9989 & 0.9773 & 0.5746 & 0.3379\\
& time &  1.0509 & 0.9585 & 0.7564 & 0.7069 & 0.8014 & 0.7777 & 1.1587 & 2.0853 & 1.9374\\
& $q$ & 18 & 18 & 18 & 18 & 18 & 18 & 18 & 18 & 18\\
\hline
\multirow{5}{*}{I-AManPG}  & NMI & 0.9883 & 0.9883 & 0.9845 & 0.9883 & 0.9796 & 0.9828 & 0.9520 & 0.4339 & 0.1264\\
& AMI & 0.9753 & 0.9753 & 0.9716 & 0.9753 & 0.9667 & 0.9699 & 0.9365 & 0.3803 & 0.0557\\
& Mod. & 0.9453 & 0.8460 & 0.7446 & 0.6474 & 0.5445 & 0.4468 & 0.3427 & 0.1728 & 0.1111\\
& purity & 1.0000 & 1.0000 & 0.9962 & 1.0000 & 0.9914 & 0.9949 & 0.9714 & 0.5634 & 0.3203\\
& time &  0.9389 & 0.9057 & 0.8584 & 0.9534 & 0.7417 & 1.0226 & 1.0085 & 2.1798 & 1.8913\\
& $q$ & 19 & 19 & 19 & 19 & 19 & 19 & 19 & 19 & 19\\
 \hline
\multirow{5}{*}{I-AManPG}  & NMI & 1.0000 & 1.0000 & 1.0000 & 1.0000 & 1.0000 & 0.9998 & 0.9600 & 0.4517 & 0.1294\\
& AMI & 1.0000 & 1.0000 & 1.0000 & 1.0000 & 1.0000 & 0.9998 & 0.9539 & 0.4037 & 0.0563\\
& Mod. & 0.9497 & 0.8499 & 0.7503 & 0.6500 & 0.5499 & 0.4498 & 0.3416 & 0.1735 & 0.1113\\
& purity & 1.0000 & 1.0000 & 1.0000 & 1.0000 & 1.0000 & 0.9999 & 0.9679 & 0.5605 & 0.3044\\
& time &  0.6577 & 0.5052 & 0.6423 & 0.9890 & 0.7112 & 0.4951 & 1.0755 & 1.7648 & 1.7855\\
& $q$ & 20 & 20 & 20 & 20 & 20 & 20 & 20 & 20 & 20\\
 \hline
\multirow{5}{*}{I-AManPG}  & NMI & 0.9835 & 0.9838 & 0.9842 & 0.9834 & 0.9836 & 0.9805 & 0.9532 & 0.4622 & 0.1402\\
& AMI & 0.9649 & 0.9654 & 0.9663 & 0.9647 & 0.9655 & 0.9618 & 0.9317 & 0.4105 & 0.0606\\
& Mod. & 0.8937 & 0.8009 & 0.7114 & 0.6148 & 0.5233 & 0.4278 & 0.3288 & 0.1710 & 0.1120\\
& purity & 1.0000 & 1.0000 & 1.0000 & 1.0000 & 0.9997 & 0.9975 & 0.9784 & 0.5662 & 0.1798\\
& time &  1.9044 & 1.7584 & 1.6944 & 1.2244 & 1.8282 & 2.0833 & 1.3244 & 2.9445 & 2.5133\\
& $q$ & 23 & 23 & 23 & 23 & 23 & 23 & 23 & 23 & 23\\
 \hline
\end{tabular}
}
\end{center}
\end{table}

First note that, from the results in Table \ref{table2} and Table \ref{table3}, the NMI, AMI, modularity and purity from I-AManPG are better than or competetive with the results for the other methods, regular or modified, when they produce a $q_c$ in the set of the $q$ values used for I-AManPG in Table \ref{table3}; specifically, when Danon's method with $\mu_{\mathrm{LFR}}=0.2$ produces $q_c=18$,  
and the Louvain method for $\mu_{\mathrm{LFR}}=0.6$ produces $q_c=19$. Furthermore, in all of these cases I-AManPG requires less computational time.
In general, even with $q \neq q_{true}$ the partitions produced by I-AManPG tend to be better than those computed by the other methods. However, since not all of these partitions achieve high NMI and AMI, a more detailed consideration is required to assess the effectiveness of I-AManPG. This is done using purity and assessing the partitions nesting relative to ground truth.
Table \ref{table3} shows that for any particular $\mu_{\mathrm{LFR}}$, as the value of the input parameter $q$ moves away from $q_{true}=20$ the modularity decreases, while the NMI and AMI achieved I-AManPG move away from desireable values close to $1$. This does not mean that the partitions are not good relative to the ground-truth partition. For $\mu_{\mathrm{LFR}}=0$ and $ 0.1$ the community partitions for $q = 18, 19, 20$ are perfectly nested, i.e., the extra communities of partition of $q+1$ are refinements of the partition of $q$ by splitting without crossing the ideal community boundaries. Since $q_{true}$ is in this set, this says that for values of $\mu_{\mathrm{LFR}}$ that imply strong community structure I-AManPG produces communities that respect the affinities of the ground-truth partition. For these two values of $\mu_{\mathrm{LFR}}$, when $q=10$, the farthest from $q_{true}$ in the set, the I-AManPG partitions perfectly still nested relative to 
the ground-truth partition. When $\mu_{\mathrm{LFR}}$ has values from  0.2 to 0.6, each partitioning for $q = 10, 18, 19, 23$ is well-nested with partitioning of $q_{true}$ and the associated purity values are very close to $1$. For $\mu_{\mathrm{LFR}}=0.7$ and $ 0.8$, the community structure is not strong in the ground-truth LFR networks and therefore purity would be expected to degrade. This is observed but note that, as in Table \ref{table2}, I-AManPG requires less computational time to produce its community partition.
These results provide promising evidence for the possibility of development of a dynamic adaptation strategy for I-AManPG. Since the sparsification to project a Stiefel element to an assignment applies, in general, to a dense $N \times q$ matrix, storage and computation can become excessive when a large number of communities must be produced. Effective nesting means this can be avoided efficiently as is done with other divisive projection-based algorithms.

\subsection{Normalized cut}

Normalized cut has been widely used for image segmentation.
Its optimization formulation is given by 
\begin{equation} \label{e37}
\min_{X \in \mathcal{A}_v} {f_{\mathrm{NC}}(X) = - \mathrm{trace} ( X^T D^{-1/2} W D^{-1/2} X )}.
\end{equation}
In the case of gray image segmentation, the matrix $W \in \mathbb{R}^{mn \times mn}$ is an affinity matrix of an $m$ by $n$ pixels gray image, $D \in \mathbb{R}^{mn \times mn}$ is a diagonal matrix with $D_{ii} = \sum_{j = 1}^{mn} W_{i j}$, and $v = \mathrm{diag}(D^{1/2})$. Here, we use the approach in~\cite{CYS2004} to choose $W$ and $D$. 

Problem~\eqref{e37} can be optimized by the weighted kernel $k$-means algorithm, see e.g.,~\cite[Algorithm~1]{DGK2005}. Note that Problem~\eqref{e37} has many low-quality local minimizers and descent optimization algorithms usually are not able to escape from them. Thus, initialization plays an important role in finding an acceptable solution. Let $U$ be the $n \times q$ matrix of the $q$ leading eigenvector of the matrix $D^{-1/2} W D^{-1/2}$. If $X$ is only required to be orthonormal, then $U$ is a global minimizer of~\eqref{e37}. Since $U$ is unlikely to be in $\mathcal{A}_v$, one approach is to find a matrix in $\mathcal{A}_v$ that is close to $U$. Different notions of closeness yield different methods. Next, we introduce four initialization methods,  including the proposed one based on AManPG. 

Bach and Jordan~\cite{BJ2003} seek to find a matrix $Y \in \mathcal{A}_v$ that minimizes
\begin{equation} \label{e38}
\|U U^T - Y Y^T\|_F.
\end{equation}
In other words, the difference between $U$ and $Y$ is measured by the orthogonal projection matrix. The weighted kernel $k$-means is suggested to solve~\eqref{e38} see~\cite[Figure~1]{BJ2003}. However, similar to~\eqref{e37}, the kernel $k$-means for~\eqref{e38} may also get stuck in a local minimizer. We use $k$-means++ in Matlab for the initialization of the kernel $k$-means for~\eqref{e38}.

Shi and Malik~\cite{SM2000} propose to find an indicator matrix that is closest to $U$ up to a rotation.
Specifically, let $\tilde{U}$ denote the matrix formed by normalizing all rows of $U$. The task is to find an indicator matrix $Z$ and a $q$-by-$q$ orthonormal matrix $Q$ that minimize
$
\|Z - \tilde{U} Q\|_F.
$
Shi and Malik~\cite{SM2000} use an alternating minimization algorithm to find $Z$ and $Q$. Note that this approach neither guarantees to find the global optimum nor use the weight vector $v$. Therefore, this approach may not find a satisfactory solution. Here, we use the C and Matlab hybrid implementation from~\cite{CYS2004}.

Karypis and Kumar~\cite{KK1998} developed METIS, a fast, multi-level graph partitioning algorithm that produces equally-sized clusters. It is shown to be an effective method for the kernel $k$-means initialization. Note that METIS does not aim to minimize the objective~\eqref{e37}. We use the C implementation from~\url{http://glaros.dtc.umn.edu/gkhome/metis/metis/download} with the Matlab interface from~\url{https://github.com/dgleich/metismex}.

We propose to initialize the weighted kernel $k$-means algorithm by I-AManPG. 
Specifically, Problem~\eqref{e37} can be reformulated as
\begin{equation} \label{e39}
\min_{X \in \mathcal{F}_v} - \mathrm{trace} ( X^T D^{-1/2} W D^{-1/2} X ) + \lambda \|X\|_1,
\end{equation}
which can be optimized by I-AManPG. We further propose to gradually increasing $\lambda$ rather than a fixed value of $\lambda$ since increasing $\lambda$ tends to give better solutions in our experiments\footnote{The $\lambda$ in I-AManPG increases by 0.01, 0.04, and 0.2}. The clusters are specified by $P_{\mathcal{A}_v}(X_*)$, as described in Section~\ref{sect:para}. Such clusters are then used as initializations for the weighted kernel $k$-means algorithm.

The four initialization methods are denoted, respectively, by BJ, SM, ME, and AM. Their combinations with the weighted kernel $k$-means algorithms are denoted, respectively, by BJ-k, SM-k, ME-k, and AM-k. The implementation of the weighted kernel $k$-means algorithm is modified from~\cite{Chen2021}\footnote{The implementation in~\cite{Chen2021} is for unweighted kernel $k$-means. We modified it for weighted kernel $k$-means.}. The test images are from~\cite{CYS2004} and the built-in images in Matlab. We further resize them to have 160-by-160 pixels as shown in Figure~\ref{fig:testedimages}. For tests of more images, we refer to~\cite{HWGV2022}.

\begin{figure}
\hspace{-3em}\includegraphics[width=1.3\textwidth]{testedimages6}
\vspace{-2em}
\caption{The test images.}
\label{fig:testedimages}
\end{figure}

An average of the negative function values $-f_{\mathrm{NC}}$ of 10 random runs are reported in Figure~\ref{fig:FNC}. We do not report the computational time since the implementations of these methods use different languages and comparing their computational time is not fair. The qualities of these methods are compared based on the objective function value $f_{\mathrm{NC}}$.
As shown in the figure, METIS initializations are not preferred since they do not aim to minimize $f_{\mathrm{NC}}$. Though SM, SM-k, BJ, BJ-k are competitive to AM and AM-k in many cases, they do not perform well in certain images, such as ME and ME-k for the football image with 3 clusters, and BJ and BJ-k for the tape image with 3 clusters. AManPG based methods are clearly most robust in the sense of minimizing the function $f_{\mathrm{NC}}$ over $\mathcal{A}_v$. The values of $-f_{\mathrm{NC}}$ by AM-k are often the highest one. Even if they are not, they are still close to the highest ones.

\begin{figure}
\centerline{
\includegraphics[width=0.35\textwidth]{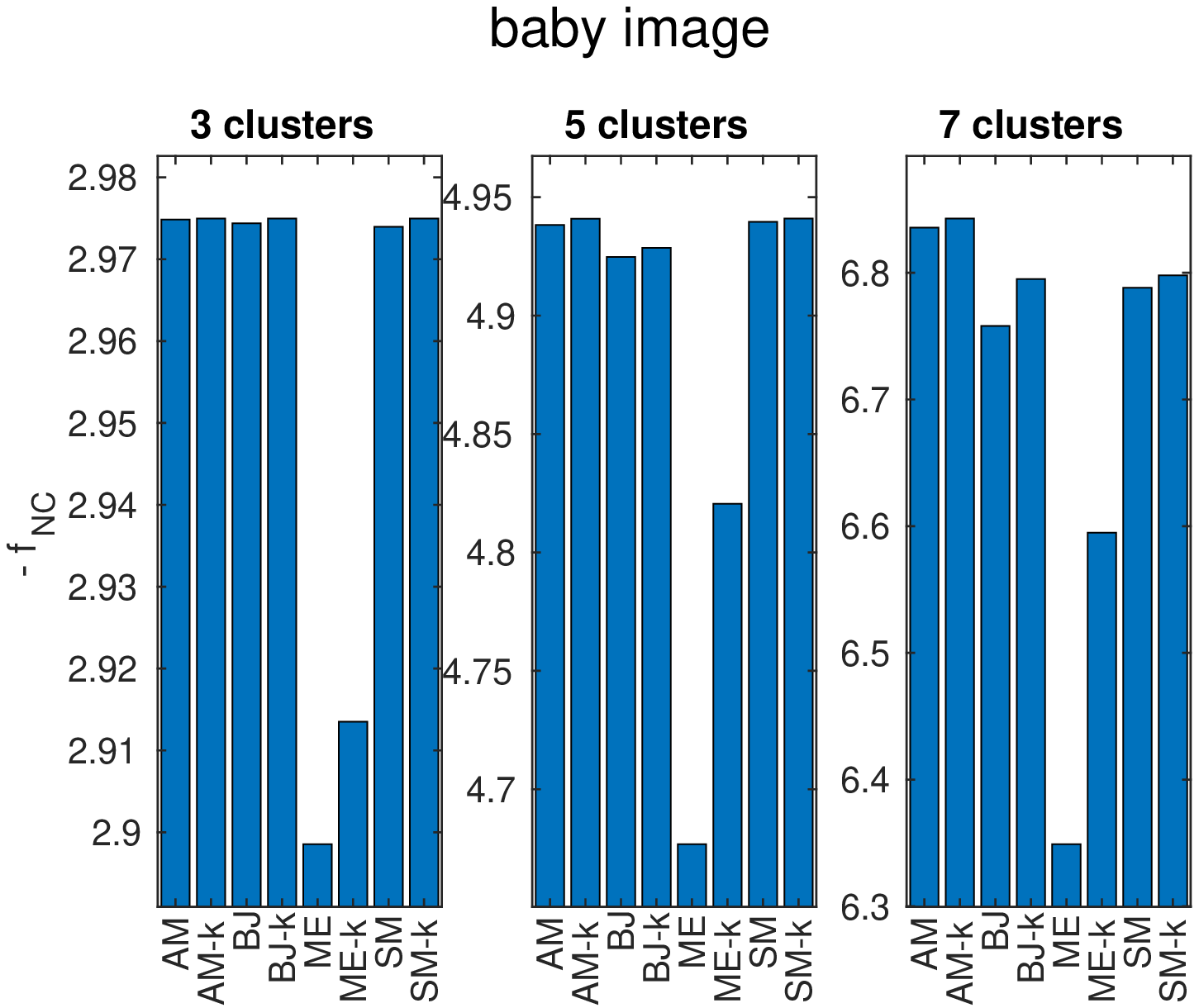}
\includegraphics[width=0.35\textwidth]{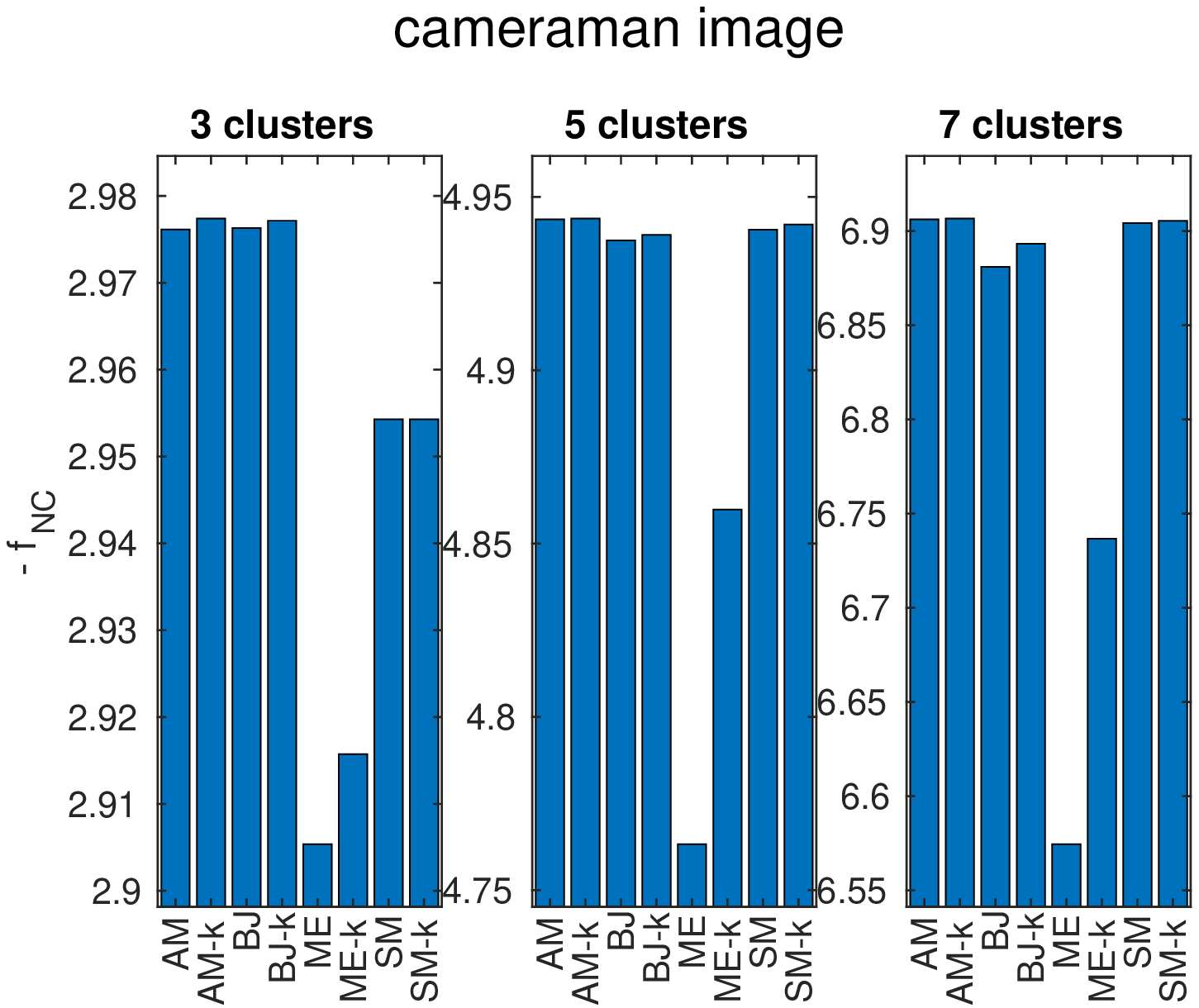}
\includegraphics[width=0.35\textwidth]{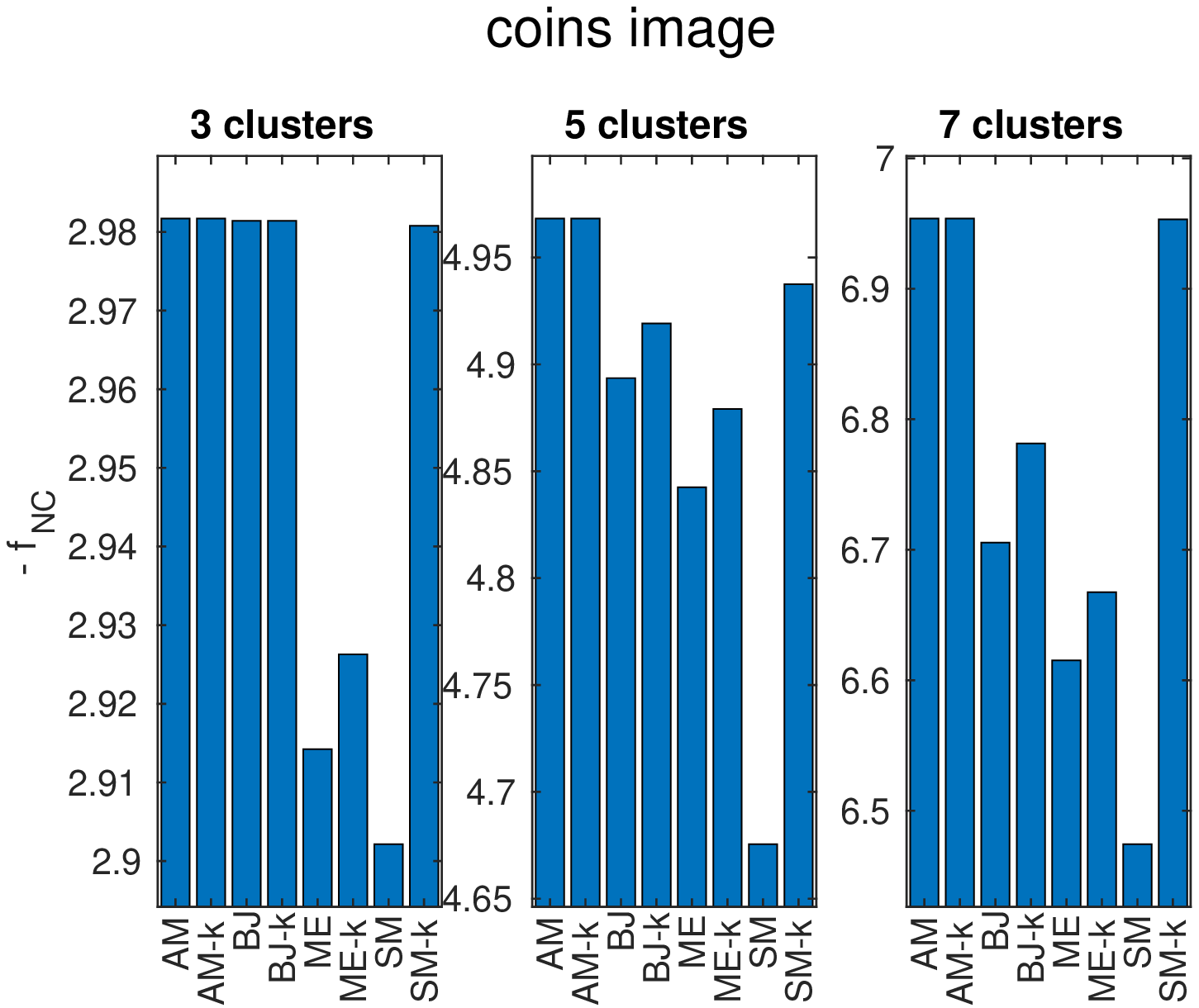}
}
\centerline{
\includegraphics[width=0.35\textwidth]{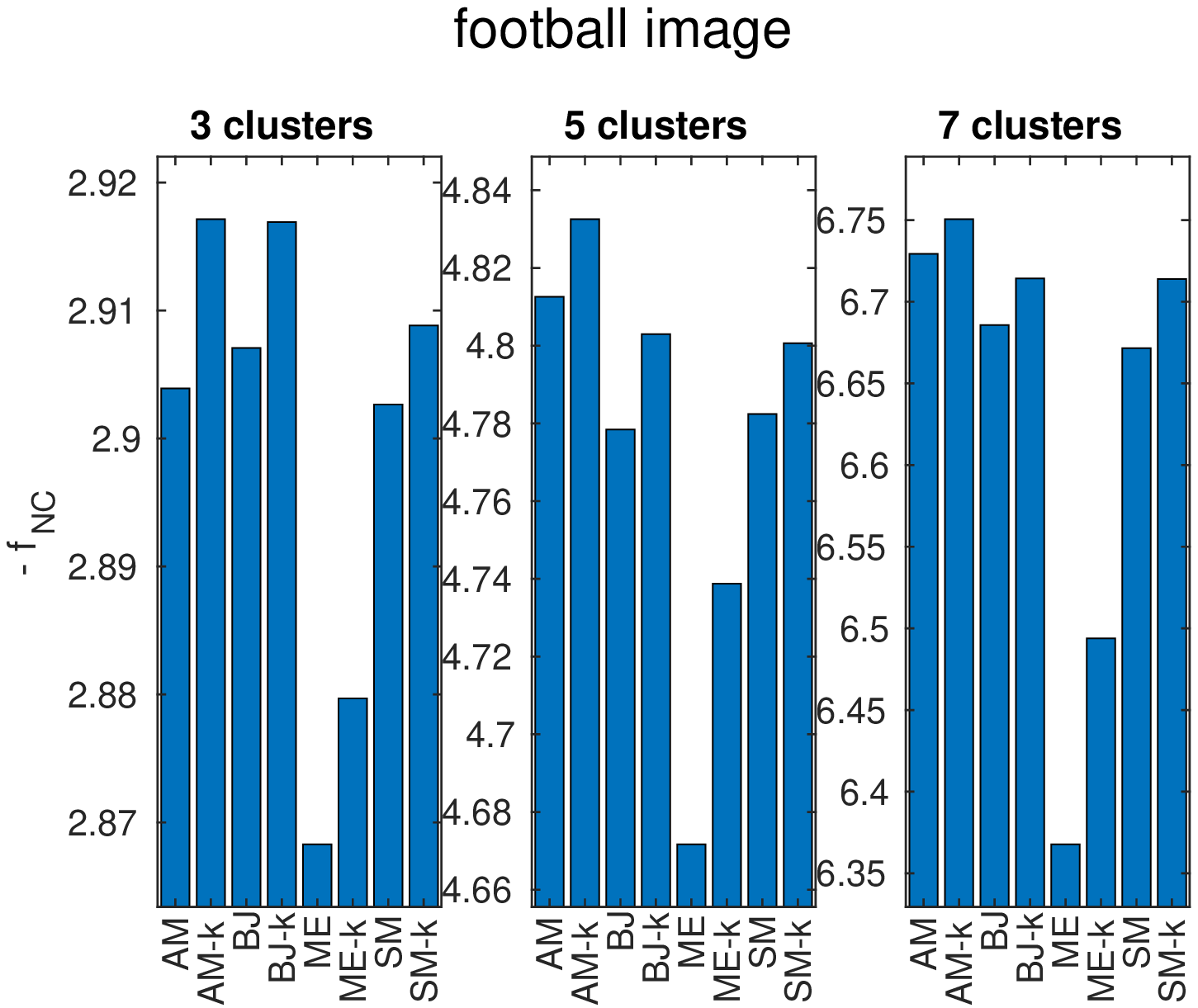}
\includegraphics[width=0.35\textwidth]{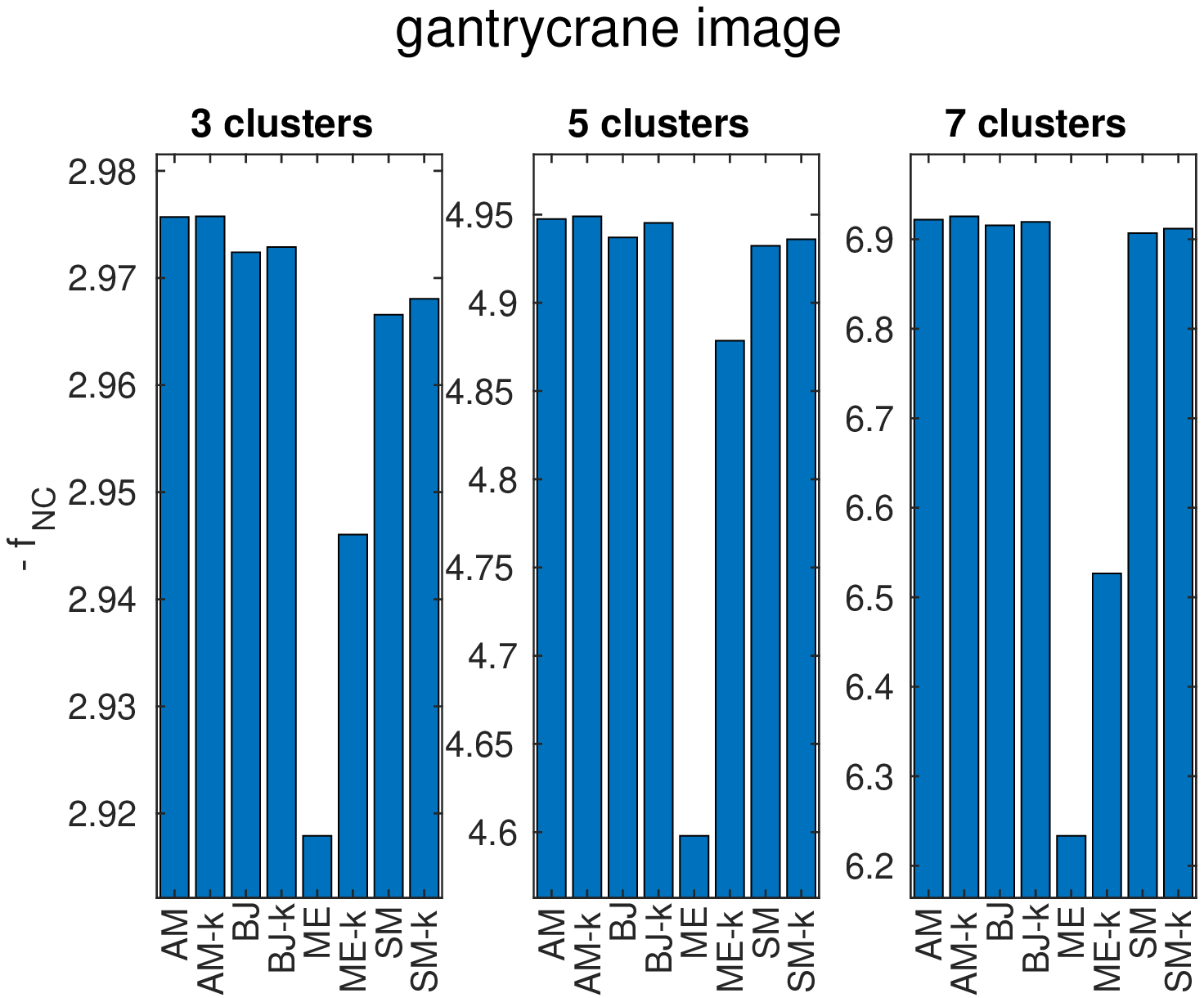}
\includegraphics[width=0.35\textwidth]{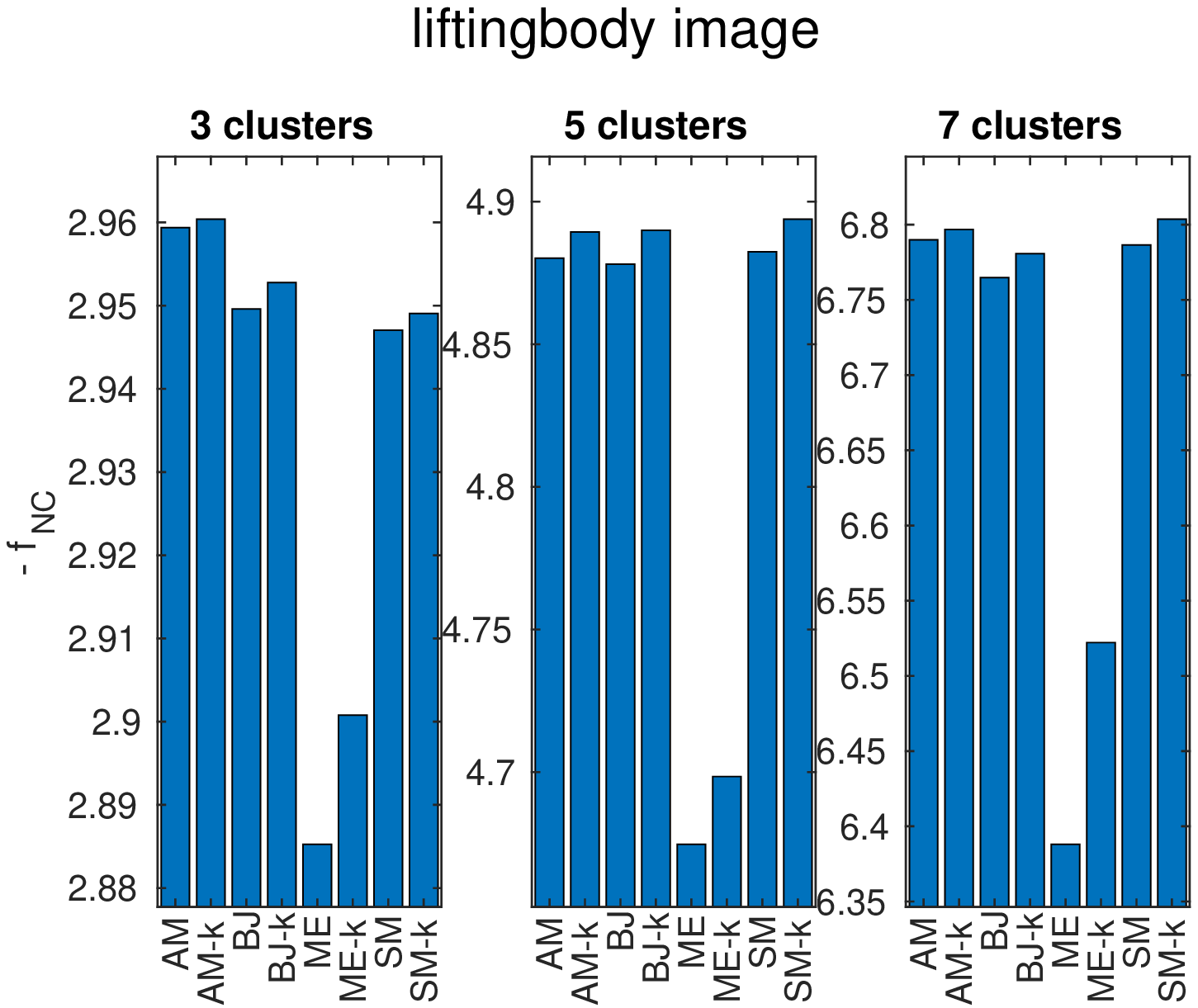}
}
\caption{An average of 10 random runs is reported. $y$-axis represents the function values. Multiple numbers of clusters are tested.
}
\label{fig:FNC}
\end{figure}

}

\section{Conclusions and Future Work} \label{sect:con}

We proposed an optimization model for clustering problems. The domain $\mathcal{F}_v$ was proven to be an embedded submanifold and its geometry structures were derived. An inexact accelerated Riemannian proximal gradient method was proposed and its global convergence proved. It was shown empirically that the proposed optimization model was more effective that the state-of-the-art methods in community detection and normalized cut for image segmentation.

Future work will address more comprehensive analysis of the choice of the parameters $\lambda$ and $q$. The current method requires an estimation of the number of clusters $q$. A critical future task is to develop a strategy to dynamically update the number of clusters thereby enabling more efficient computation for problems with a large number of communities.

\bibliographystyle{alpha}
\bibliography{WHlibrary}

\end{document}